\documentclass[12pt]{amsart}
\usepackage{tikz-cd}
\usepackage{verbatim}
\usepackage{soul}
\usepackage{xcolor}
\usepackage{url}
\usepackage{hyperref}
\numberwithin{equation}{section}
\usepackage{amssymb}
\usepackage{enumitem}
\usepackage[margin=3.3cm]{geometry}
\usepackage{mathtools}
\usepackage{setspace}
\usepackage{mathabx}
\usetikzlibrary{calc,intersections,arrows.meta,positioning,decorations.markings}
\usepackage{pgfplots}
\pgfplotsset{compat=1.17}
\numberwithin{figure}{section}
\let\cal\mathcal

\def\Cscr{{\cal C}}
\def\Dscr{{\cal D}}
\def\Escr{{\cal E}}
\def\Fscr{{\cal F}}

\def\Jscr{{\cal J}}

\def\Lscr{{\cal L}}
\def\Mscr{{\cal M}}

\def\Oscr{{\cal O}}

\def\Tscr{{\cal T}}

\let\blb\mathbb
\def\CC{{\blb C}} 
 
\def\EE{{\blb E}} 
 
\def\FF{{\blb F}} 
\def\QQ{{\blb Q}}

\def \PP{{\blb P}}

\def \ZZ{{\blb Z}}

\def \NN{{\blb N}}
\def \RR{{\blb R}}
\def \HH{{\blb H}}

\def \pp{{\mathfrak p}}
\def\num{\operatorname{num}}

\let\oldmarginpar\marginpar
\def\marginpar#1{\setlength{\marginparwidth}{0.15\textwidth}\oldmarginpar{\tiny\def\tiny{} \baselineskip 0pt
  \raggedright #1}}

\def\Bl{\operatorname{Bl}}

\def\Id{\operatorname{id}}

\def\mod{\operatorname{mod}}
\def\Gr{\operatorname{Gr}}

\def\Gr{\operatorname{Gr}}
\def\gr{\operatorname{gr}}

\def\coh{\mathop{\text{\upshape{coh}}}}

\def\gr{\operatorname {gr}}
\def\Spec{\operatorname {Spec}}

\def\depth{\operatorname {depth}}

\def\Ext{\operatorname {Ext}}
\def\Hom{\operatorname {Hom}}

\def\End{\operatorname {End}}
\def\RHom{\operatorname {RHom}}

\def\coker{\operatorname {coker}}
\def\ker{\operatorname {ker}}

\def\End{\operatorname {End}}

\def\add{\operatorname {add}}
\def\rk{\operatorname {rk}}

\def\Tot{\operatorname {Tot}}
\def\reff{\operatorname {ref}}

\def\Pic{\operatorname {Pic}}

\def\r{\rightarrow}

\DeclareMathOperator{\HHom}{\mathcal{H}\mathit{om}}

\DeclareMathOperator{\Perf}{Perf}
\def\QCoh{\operatorname{QCoh}}
\def\conv{\operatorname{conv}}
\def\area{\operatorname{\Delta}}
\def\red{\operatorname{/}}
\def\redd{\operatorname{red}}
\def\ee{\mathfrak{e}}

\newtheorem{lemma}{Lemma}[section]
\newtheorem{proposition}[lemma]{Proposition}
\newtheorem{theorem}[lemma]{Theorem}
\newtheorem{corollary}[lemma]{Corollary}

\newtheorem{convention}[lemma]{Convention}
\newtheorem{lemmas}{Lemma}[subsection]

\newtheorem{conventions}[lemmas]{Convention}

\theoremstyle{definition}

\newtheorem{example}[lemma]{Example}
\newtheorem{definition}[lemma]{Definition}

{

}

\theoremstyle{remark}

\newtheorem{remark}[lemma]{Remark}

\newdimen\uboxsep \uboxsep=1ex
\def\uboxn#1{\vtop to 0pt{\hrule height 0pt depth 0pt\vskip\uboxsep
\hbox to 0pt{\hss #1\hss}\vss}}

\def\uboxs#1{\vbox to 0pt{\vss\hbox to 0pt{\hss #1\hss}
\vskip\uboxsep\hrule height 0pt depth 0pt}}

\newtheoremstyle{named}{}{}{\itshape}{}{\bfseries}{.}{.5em}{\thmnote{#3's }#1}
\theoremstyle{named}

\definecolor{mygray}{gray}{0.9}
\sethlcolor{mygray}

\makeatletter
\@namedef{subjclassname@2020}{\textup{2020} Mathematics Subject Classification}
\makeatother

\title{NCCRs of cones over del Pezzo surfaces}

\author{Anya Nordskova}
\author{Michel Van den Bergh}
\thanks{The first author is supported by World Premier International Research Center Initiative (WPI),  MEXT,  Japan.  The second author is a senior researcher at the Research Foundation Flanders (FWO).  This project has received funding from the European Research Council (ERC) under the European Union's Horizon 2020 research and innovation programme (grant agreement No 885203).}
\email{anya.nordskova@ipmu.jp}
\email{michel.vandenbergh@uhasselt.be}
\address[Michel Van den Bergh]{Vakgroep Wiskunde en Data Science, Vrije Universiteit Brussel, Pleinlaan 2, 1050 Brussel} 
\email{michel.van.den.bergh@vub.be}
\address[Michel Van den Bergh]{Vakgroep Wiskunde, Universiteit Hasselt, Universitaire Campus \\ B-3590 Diepenbeek}
\address[Anya Nordskova]{Kavli IPMU (WPI), UTIAS, The University of Tokyo, Kashiwa, Chiba 277-8583, Japan}
\begin{document}
\begin{abstract}
  Non-commutative crepant resolutions (NCCRs) are non-commutative versions of classical crepant resolutions in algebraic geometry. 
  For 3-dimensional terminal Gorenstein singularities Iyama and Wemyss proved that all NCCRs are connected by mutations,
  which may be viewed as a non-commutative analogue of Kawamata's result that all crepant resolutions are connected by flops.
  In this paper we prove the corresponding result for a class of canonical Gorenstein singularities
  which are not terminal, namely anticanonical cones over del Pezzo surfaces.  More precisely,  we first obtain a classification of NCCRs of anticanonical del Pezzo cones, showing that every NCCR arises from a geometric helix on the corresponding del Pezzo surface.  We then prove that all such geometric helices are connected to each other by mutations,  up to simple operations which include tensoring by line bundles and shifts. 

A crucial ingredient in our proofs is the polygons that can be associated to exceptional collections on del Pezzo surfaces following the works of Hille and Perling.  We obtain some interesting observations about these polygons which may be of independent interest.
\end{abstract}
\maketitle
 \tableofcontents 
\section{Introduction} 
\subsection{Main results}
Let $Z$ be a normal algebraic variety with Gorenstein
singularities.  A resolution of singularities (that is,  a proper birational map from a non-singular variety) ${\pi: Y \to Z}$ is said to be \emph{crepant} if ${\pi^*\omega_Z = \omega_Y}$. Such a $Y$ is often considered to be a nicest possible, ``minimal'' desingularisation of $Z$.  A non-commutative analogue of this classical notion was introduced in \cite{VdB} (see also \cite{VdBICM} for an overview).  More precisely,  if $R$ is a normal Gorenstein domain, then its \emph{non-commutative crepant resolution} (\emph{NCCR}) is an $R$-algebra $\Lambda = \End_R(M)$ with $M$ a finitely generated reflexive $R$-module, such that $\Lambda$ is Cohen-Macaulay as an $R$-module and has finite global dimension. 

\medskip

Now let $X$ be a smooth del Pezzo surface over $\CC$ and consider the graded ring of sections $R_X :=  \bigoplus_{k \geq 0} \Gamma(X, \omega_X^{-k})$.
The affine variety $Z := \Spec(R_X)$ is a cone over $X$ with the unique singularity at the origin which is a classical example of a canonical Gorenstein singularity that is not terminal.  The standard crepant resolution of $\Spec R_X$ is given by the non-compact Calabi-Yau variety $Y :=\Tot(\omega_X)$ and the map ${p: Y \to Z}$ contracting the zero section of the canonical bundle.  It was established in \cite[Proposition 7.3]{VdB} that cones over del Pezzo surfaces also have non-commutative crepant resolutions.  More precisely,  an NCCR of $R_X$ can be constructed as follows.  One shows that on every del Pezzo surface $X$ there exists a full exceptional collection
of vector bundles $\EE =(E_0, \dots,E_{n-1})$  such that $\Ext^l_X(E_i,E_j \otimes \omega_X^{-k}) = 0$  for all $k \geq 0$,  $l \neq 0$ and any $i,j$ (we refer to exceptional collections satisfying this condition as being \emph{very strong},  see Definition \ref{def:verystrong}).  Let $\pi: Y \to X$ be the canonical projection and $\Escr := \oplus_{i=0}^{n-1} E_0$.  If $\EE$ is a very strong exceptional collection on $X$,  then $\pi^*\Escr$ is a tilting bundle in $\Dscr^b(Y)$ and $B(\EE) := \End_{\Dscr^b(Y)}(\Escr)$ is an NCCR of $R_X$.  Our first result shows that all NCCRs of the completion\footnote{For simplicity we state our results in the introduction for the completion~$\widehat{R_X}$. However these results have uncompleted versions, valid for $R_X$,
provided we restrict to graded NCCRs.} $\widehat{R_X}$ in fact arise in this way:
\begin{theorem}[Theorem \ref{thm:nccrfromFEC}, Proposition \ref{prop:completion}]\label{thm:intro:classification}
Let $\Lambda$ be an NCCR of $\widehat{R_X}$.  Then there exists a very strong exceptional collection $\EE = (E_0, \dots, E_{n-1})$ on $X$ such that $\Lambda$ is Morita equivalent to the completion of $\End_{\Dscr^b(Y)}(\pi^*\Escr)$, where $\Escr = \oplus_{i=0}^{n-1} E_i$. 
\end{theorem}

Using the terminology introduced by Bridgeland and Stern in \cite{BridgelandStern},  a very strong exceptional collection $\EE = (E_0, \dots, E_{n-1})$ is the same as a \emph{thread} of a \emph{geometric helix} 
$ \HH = \{E_i \}_{i \in \ZZ}$,  where $E_{i-n} = E_i \otimes \omega_X$,  and an NCCR $B(\EE)$ is the \emph{rolled-up helix algebra}\footnote{The helix algebra was in fact already introduced in  \cite{BondalPolishchuk} as a $\ZZ$-algebra. The rolled-up helix algebra
  as defined in \cite{BridgelandStern} is a special case of a general construction which associates a graded algebra to a periodic $\ZZ$-algebra.} of $\HH$.
As we explain in \S \ref{sec:mutnccr},  Theorem \ref{thm:intro:classification} recovers in particular the main result of \cite{BridgelandStern}.

\medskip

As in the classical commutative setting, non-commutative crepant
resolutions are in general far from unique.  Any NCCR can be modified
to obtain different NCCRs using so-called \emph{mutations}
defined by Iyama-Reiten, Iyama-Wemyss
\cite{IyamaReiten,IyamaWemyss14_2}. While mutations of NCCRs may be viewed
as non-commutative counterparts of flops, the nature of our understanding
of them is not the same.
Whereas the
conjecture of Bondal and Orlov \cite{BondalOrlov95}, predicting derived
equivalence for varieties related by a flop, remains open in general,
mutations of NCCRs naturally induce derived equivalences.  On the
other hand, projective crepant resolutions are known to be connected
by sequences of flops \cite{Kawamata}, but for NCCRs of Gorenstein
rings the analogous statement is at most a conjecture.

However for 3-dimensional \emph{terminal}
Gorenstein (equivalently, compound Du Val singularities) Iyama and
Wemyss \cite[Theorem 9.6]{IW23}, \cite[Corollary 4.5]{Wemyss} showed
that all NCCRs are indeed related by sequences of mutations.
Our second main result extends this to a natural class of canonical Gorenstein singularities,
namely the completed anti-canonical del Pezzo
cones
\begin{theorem}[Theorem \ref{thm:nccrsrelated}]\label{thm:intro:main} All NCCRs of $\widehat{R_X}$ for $X$ a del Pezzo surface are related by sequences of mutations.  \end{theorem}
NCCRs of $\widehat{R_X}$,  or,  equivalently,  by Theorem \ref{thm:intro:classification},  rolled-up helix algebras of geometric helices, are $3$-Calabi-Yau,  and hence admit presentations as Jacobian algebras associated to quivers with potentials.  Accordingly,  mutations of NCCRs may be interpreted as mutations of quivers with potential in the sense of Derksen-Weyman-Zelevinsky \cite{DWZ}.

\medskip

Now recall that there is yet another kind of ``mutation'' operation which  induces a braid group action on the set of full exceptional collections (\cite{MR992977},  see also Proposition \ref{pr:braidaction}). 
These ``braid'' mutations do not generally preserve the property of an exceptional collection being very strong.  However,  in \cite{BridgelandStern} it was shown that mutations of rolled-up helix algebras can nevertheless be realised as compositions of successive braid mutations of the underlying exceptional collections
(see Proposition \ref{prop:rem:cluster} for the precise recipe).
We will thus use the term \emph{quiver mutation} for three distinct but closely related concepts: (1)
mutations of rolled-up helix algebras in the sense of NCCRs; (2) DWZ mutations of quivers with potential and (3) the corresponding sequences of braid mutations of underlying very strong exceptional collections or geometric helices.

In light of Theorem \ref{thm:intro:classification} and the fact that mutations of NCCRs can be performed on the level of their underlying geometric helices,  we prove Theorem \ref{thm:intro:main} by showing that all such helices are related by mutations,  modulo elementary operations that do not affect the corresponding NCCR: 
\begin{theorem}[Theorem \ref{thm:helixrelated}, see also
  \cite{Bousseau}]\label{thm:intro:helices} All geometric helices on a
  del Pezzo surface $X$ are related by quiver mutations, simultaneous
  shifts in the derived category and tensoring by line bundles,
  reordering of objects in orthogonal blocks and rotations (shifting
  the labels of objects). \end{theorem}
See \S \ref{sec:intocontext} for some context on this result and the relation with \cite{Bousseau}.

\subsection{Method of proof}
\label{sec:method}
We say that a very strong exceptional collection is \emph{minimal} if
it does not admit a quiver mutation reducing the sum of the ranks of
its objects. To prove Theorem \ref{thm:intro:helices} we obtain a
complete classification of (Gram matrices of) minimal very strong
exceptional collections. After this it remains to relate (a small
finite number of) minimal collections by sequences of mutations for
every del Pezzo surface, which we do by an explicit computer search.

\medskip

A key ingredient which goes into the classification of minimal exceptional
collections is the beautiful toric framework introduced by Hille and
Perling \cite{HillePerling,Perling} (see also \cite{KuznetsovToric}), which we recall and develop further
in \S \ref{sec:toric}.  This machinery associates a star-shaped polygon in a $2$-dimensional lattice to an exceptional collection on a del
Pezzo surface~$X$ in such a way that very strong exceptional collections
correspond to convex polygons. In this paper we obtain some results about the HP-polygon which we think are new.
\begin{itemize}
\item
In Theorem \ref{th:gale}, 
we show that the HP-polygon is in fact closely related to the dual exceptional collection.
\item
Leveraging this, we also deduce that the quiver associated to the rolled-up helix algebra of
a very strong exceptional collection can be obtained very directly from the HP-polygon (see Theorem \ref{thm:achievement}).
\item If the HP-polygon does not have parallel edges (which is not a restriction for us as discussed in \S\ref{sec:block_complete})
  then the shape of the polygon and the associated quiver is very constrained. In fact in Proposition \ref{prop:herzog}
  we verify that an informal conjecture by Herzog \cite{Herzog} about the quiver shape holds in this case.
\end{itemize}

The
key point is now that quiver mutations of very strong exceptional collections can be interpreted as certain
simple geometric operations (see Figure \ref{fig:quiver_mutation}) on these convex polygons.
The minimal
very strong exceptional collections are then characterised
by the fact that the area of the polygon cannot be reduced by such operations.
This is equivalent to 
 the origin being  inside a certain explicit ``forbidden region'' inside the polygon (see Corollary \ref{cor:forbidden} and also Figure \ref{fig:forbidden} for some examples).
It is this geometric constraint that we leverage to obtain our classification.

\medskip Some of our results depend on computer computations. The accompanying scripts can be found at \cite{GHrepo}.

\subsection{Context for this paper}\label{sec:intocontext}
Theorem \ref{thm:intro:main} was first formulated as a conjecture in
the first named author's PhD thesis, where it was also proved in two
special cases using the method described in \S\ref{sec:method}.  Shorty
after we obtained a proof in the general case, using the same method, we were informed that
Pierrick Bousseau had also independently proved the conjecture, and later this
proof was published as the preprint~\cite{Bousseau}. We did prefer however to continue writing down our own proof, as it was at that stage, without consulting~\cite{Bousseau}, and this has resulted in the current paper. But now we
make a comparison of this paper with \cite{Bousseau}.

  As outlined in \S\ref{sec:method} the main issue in the proof of Theorem \ref{thm:intro:main} is the classification of the mutation classes of HP-polygons. We were unaware however that this classification was in fact already known
  under the guise of the classification of so-called ``T-polygons'', up to mutation (see \cite[Theorem 1.2]{KNP} for a case by case proof and \cite[Theorem 4.8]{Lutz} for an intrinsic geometric proof, inspired by mirror symmetry).
  So the material of \S\ref{sec:minimal_collections} and \S\ref{sec:minimal} is in fact an independent reproof of this classification.
  
  The fact that HP-polygons are T-polygons was already noted
 in \cite[Theorem 11.3]{Perling}. In \cite[Definition 2.13]{Bousseau}
  a  T-polygon is constructed directly from the dual of a very strong exceptional collection. We obtain the same construction for the HP-polygon in
  \S\ref{sec:dual_ex}. So the two polygons are in fact the same.

  The classification of HP-polygons amounts to the classification of Gram matrices of very strong exceptional collections. Two exceptional collection with the same Gram matrix are related by the affine Weyl group. So it roughly remains to realise generators
  of the Weyl group as compositions of mutations. In this paper we simply use an explicit
computer search for this because we again overlooked that this part is known as well (see \cite[Appendix A]{Mizuno}).
The treatment in loc.\ cit.\ depends on a case by case analysis but
Bousseau replaces it by a very elegant geometric argument in the proof of \cite[Theorem 2.12]{Bousseau}. 

\section{Acknowledgement} 
The authors are very grateful to Sasha Kuznetsov and Tom Bridgeland for useful discussions and for their numerous helpful comments on the relevant part of the first-named author's PhD thesis.

All computer computations in this paper were carried out using SageMath \cite{sagemath}, a free open-source mathematics software system. The authors warmly thank the SageMath community and its contributors for developing and maintaining this resource.
\section{Notation and Conventions}
\label{sec:notation}
Unless otherwise specified modules are left modules.

\medskip

In principle our ground field is $\CC$ although we sometimes formulate algebraic statements over a ground field $k$, which we assume
to be algebraically closed.

\medskip
A quiver $Q$ with be a 4-tuple $(Q_0,Q_1,s,t)$ where the $Q_0$ is the set of vertices, $Q_1$ is the set of arrows and $s,t:Q_1\r Q_0$ are
the maps which relate an arrow to its source a target. We use functional composition in the path algebra $kQ$. In other words representations of $Q$
correspond to left $kQ$-modules.  
Sometimes we use the notation $Q(i,j)$ for the set of arrows starting in vertex $i$ and ending in vertex $j$.
Unless otherwise specified, a quiver is assumed to be finite.

\medskip

A \emph{polygon} is  defined as a non-self-intersecting cycle of a finite number of non-degenerate intervals which we call \emph{edges}.
The endpoints of the intervals are the \emph{vertices} of the
polygon. A polygon $P$ is a Jordan curve and we write $|P|$ for the closure of the bounded connected component of its complement.
If the adjacent edges of a vertex are collinear then we call the vertex \emph{straight}, otherwise we call it \emph{non-straight}.
Sometimes we will refer to the line segments connecting the non-straight edges as \emph{long edges}. The long edges are chains of collinear edges.
The vertices of $|P|$ (in the usual sense) are the non-straight vertices of $P$ and similarly the edges of $|P|$ (also in the usual sense)
correspond to the long edges of $P$.
\emph{Deleting a vertex} $y$ from a polygon is defined as replacing the adjacent edges $[xy]$ and $[yz]$ by a single edge $[xz]$.

\medskip

We denote the area of a region in a plane, equipped with a volume form, by $\area(A)$. If $A=\conv(a_1,\ldots,a_n)$ then we will also write $\Delta(a_1\cdots a_n)$
for $\Delta(A)$.

\medskip

If $A$ is a graded algebra and the $(e_i)_{i=1}^n$ is a complete system of idempotents of degree zero then a \emph{regrading} of $A$ is a graded
algebra of the form $\tilde{A}=\End_A(\oplus_{i=1}^n (e_iA)(n_i))$ for $n_i\in \ZZ$. It is obvious that $\Gr(A)$ and $\Gr(\tilde{A})$ are equivalent where $\Gr(?)$ denotes
the category of graded modules. In fact regrading is a particular case of \emph{graded Morita equivalence} which is defined in a similar way as ordinary Morita equivalence.

We have $A=\oplus_{i,j} e_j A e_i$ and $\tilde{A}=\oplus_{i,j} (e_j A e_i) (n_j-n_i)$. In other words, regrading amounts to shifting the pieces $e_jAe_i$ by $n_j-n_i$ degrees
to the left.
If $A$ has a graded presentation $kQ/(R)$  and $(e_i)_i$ are the standard idempotents corresponding to vertices then regrading amounts to changing the
degree of an arrow $e:i\r j$ to $\deg e-n_i+n_j$.
\section{Preliminaries on exceptional collections}
\subsection{Exceptional collections}
\label{sec:ex_col}
Let $X$ be a smooth projective variety.
\begin{definition} \label{def:exceptional} 
An object $E \in \Dscr^b(X)$ is \emph{exceptional} if $\Hom^i(E,E) = 0$ for any $i \neq 0$ and $\Hom^0(E,E) \cong \CC$.  An ordered collection $\EE = (E_0, \dots,E_{n-1})$ of objects in $\Dscr^b(X)$ is called an \emph{exceptional collection} if each $E_i$ is exceptional and $\Hom^\bullet(E_j,E_k) = 0$ for all $j>k$.  An exceptional collection $\EE$ is \emph{full} if $\Dscr^b(X)$ is the smallest triangulated subcategory of $\Dscr^b(X)$ containing $E_0, \dots,E_{n-1}$. 
\end{definition} 

\begin{definition} Let $(E,F)$ be an exceptional pair.  We define the left \emph{mutation} of $(E,F)$ as $(L_E F,  E)$,  where the object $L_E F$ is given by the distinguished triangle
\begin{align*}
\Hom^\bullet(E,F) \otimes E \to F\to L_E F.  
\end{align*}
 Dually,  the right mutation of $(E,F)$ is defined to be $(F, R_F E)$, where $R_F E$ is given by the distinguished triangle
 \begin{align*}
 R_F E \to E \to \Hom^\bullet(E,F)^\ast \otimes F.
\end{align*}

\end{definition} 
\begin{proposition}[{Bondal \cite[Assertion 2.3]{MR992977}}]\label{pr:braidaction} The braid group $B_n$ on $n$ strands acts on the set of exceptional collections in $\Dscr^b(X)$ of length $n$ by mutations.  More precisely,  the standard generator $\sigma_i \in B_n$, $1 \leq i \leq n-1$ acts by
\begin{align*} \sigma_i(E_0, \dots,E_{n-1}) = (E_0, \dots, E_{i-2}, L_{E_{i-1}}E_{i}, E_{i-1},E_{i+1}, \dots, E_{n-1})
\end{align*}
\end{proposition} 

\begin{proposition}[Bondal {\cite[Lemma 5.6]{MR992977}}] \label{def:dualcollection} Let $\EE = (E_0, \dots, E_{n-1})$ be a full exceptional collection in $\Dscr^b(X)$.  Define 
\[
\widecheck{E_i} := R_{E_{n-1}} \dots R_{E_{i+1}}(E_i)
\]
for all $0 \leq i \leq n-1$.  Then $\widecheck{\EE}= (\widecheck{E_{n-1}}, \dots, \widecheck{E_0})$ is a full exceptional collection such that $\Hom_{\Dscr^b(X)}(\widecheck{E_j},E_i[l]) \cong \CC$ if $l = 0,  i = j$ and $0$ otherwise.  We say that the collection $\widecheck{\EE}$ is \emph{right dual} to $\EE$.  Likewise,  the  \emph{left dual} exceptional collection of $\EE$ is defined as $\widecheck{\EE}' = (\widecheck{\EE_{n-1}}',\ldots, \widecheck{\EE_0}')$ with $\widecheck{\EE_i}' := L_{\EE_{0}} \dots L_{\EE_{i-1}}(\EE_i)$.  One has $\Hom_{\Dscr^b(X)}(E_i,\widecheck{E_j}'[l]) \cong \CC$ if $l = 0,  i = j$ and $0$ otherwise.
\end{proposition}
\begin{remark} \label{rem:left_right_dual} It is an elementary exercise in Serre duality that the left and right dual of an exceptional collection
  are related by the Serre functor.
  \end{remark}
\begin{remark} \label{rem:extreme_dual}
  Let the notations be as in Proposition \ref{def:dualcollection}. Then one may check that $\widecheck{E_{n-1}}=E_{n-1}$ and $\widecheck{E_0}=\omega_X^{-1}\otimes E_0[-\dim X]$.  Likewise,  for the left dual exceptional collection one has $\widecheck{E_{0}}' = E_0$ and $\widecheck{E_{n-1}}'=\omega_X \otimes E_{n-1}[\dim X]$. 
\end{remark}
For use below we make the following definition.
\begin{definition} \label{def:gram_matrix}
  The \emph{Gram matrix} $M$ of an exceptional collection $\EE$ is the $n\times n$ matrix such that $M_{ij}=\chi_X(E_i,E_j)$.
  \end{definition}
\subsection{Exceptional collections on del Pezzo surfaces}\label{sec:delPezzoprelim}
Here we recall some classical results on exceptional objects and exceptional collections on del Pezzo surfaces, mostly following the works of Gorodentsev and Kuleshov-Orlov \cite{KuleshovOrlov}.  Throughout $X$ denotes a smooth del Pezzo surface. 

\begin{definition} Let $\Fscr$ be a coherent sheaf.  The \emph{degree} of $\Fscr$ is defined by 
\[
d(\Fscr) = c_1(\Fscr)\cdot (-K_X)
\]
For a torsion-free sheaf $\Fscr \in \coh(X)$ we define the \emph{slope} of $\Fscr$ as 
\[
\mu(\Fscr) = \frac{d(\Fscr)}{r(\Fscr)}
\] and denote 
\[
s(\Fscr) = \frac{c_1(\Fscr)}{r(\Fscr)} \in \Pic(X) \otimes \QQ
\]
For $\Fscr$ a torsion sheaf we set $\mu(\Fscr) = +\infty$.  
 \end{definition}

For two torsion-free sheaves $\Escr, \Fscr$ we will also use the notations 
\begin{align*}
s(\Escr,\Fscr) := s(\Fscr) - s(\Escr)  \\
\mu(\Escr,\Fscr) := (-K_X)\cdot s(\Escr,\Fscr) = \mu(\Fscr) - \mu(\Escr).
\end{align*}

The following can be derived from the Riemann-Roch formula: 

\begin{proposition}\label{prop:chi} Let $\Escr, \Fscr$ be exceptional torsion-free sheaves on $X$.  Then 
  \begin{equation}
    \label{eq:riemann_roch}
\chi(\Escr, \Fscr) = \frac{ r(\Fscr) r(\Escr)}{2} \bigg(s(\Escr,\Fscr)^2 + \mu(\Escr,\Fscr) + \frac{1}{r(\Escr)^2} + \frac{1}{r(\Fscr)^2}  \bigg).
\end{equation}
 If in addition $(\Escr, \Fscr)$ form an exceptional pair, then 
 \begin{equation}
\label{eq:chi_slope}
\chi(\Escr, \Fscr) = r(\Escr)d(\Fscr) - r(\Fscr)d(\Escr) = r(\Escr)r(\Fscr)\mu(\Escr,\Fscr).
\end{equation}
\end{proposition} 

The following proposition claims that the only exceptional objects in derived categories of del Pezzo surfaces are shifts of exceptional vector bundles and torsion sheaves supported on $(-1)$-curves.  

\begin{proposition}[Kuleshov-Orlov, {\cite[Propositions 2.9, 2.10]{KuleshovOrlov}}]\label{pr:KOexobjs} Let $X$ be a del Pezzo surface.  
\begin{enumerate}
\item  An object $A \in \Dscr^b(X)$ is exceptional if and only if it is isomorphic to a shift $\Fscr[i]$ of some exceptional sheaf $\Fscr$ on $X$,  $i \in \ZZ$.
\item Let $\Fscr$ be an exceptional sheaf on $X$.  Then $\Fscr$ is either a vector bundle or a torsion sheaf of the form $\Oscr_C(d)$, where $C$ is a $(-1)$-curve, i.e. an irreducible rational curve $C$ with $C^2 = C\cdot K_X = -1$. 
\end{enumerate}
\end{proposition} 

Exceptional pairs of sheaves are classified into three types and the type can be determined by the order of their slopes. 

\begin{proposition}[Gorodentsev, {\cite[Propositions 5.3.1, 5.3.3]{Gorodentsev}}, see also Kuleshov-Orlov {\cite[Corollary 2.11]{KuleshovOrlov}}]\label{pr:slopes}
Let $(\Escr, \Fscr)$ be an exceptional pair of sheaves on a del Pezzo surface.  Then $\Ext^2(\Escr, \Fscr) = 0$ and at most one of the spaces $\Ext^i(\Escr, \Fscr)$ with $i=0,1$ is non-zero. Moreover, 
\begin{enumerate}
\item 
$\Hom(\Escr, \Fscr) \neq 0$ if and only if $\mu(\Escr) < \mu(\Fscr)$ (in this case we say that $(\Escr, \Fscr)$ is a  \emph{$\Hom$-pair}),
\item $\Ext^1(\Escr, \Fscr) \neq 0$ if and only if $\mu(\Escr) > \mu(\Fscr)$ (in this case we say that $(\Escr, \Fscr)$ is a \emph{$\Ext$-pair}),
\item $\Hom(\Escr, \Fscr) = \Ext^1(\Escr, \Fscr) = 0$ if and only if $\mu(\Escr) =  \mu(\Fscr)$ (a \emph{null-pair} or an \emph{orthogonal pair}).
\end{enumerate}
\end{proposition} 
\begin{theorem}[Gorodentsev \cite{Gorodentsev},  2.4]\label{thm:mustable} An exceptional vector bundle $\Fscr$ on a del Pezzo surface $X$ is $\mu$-stable (Mumford-Takemoto stable with respect to $-K_X$), i.e. $\mu(\Fscr') < \mu(\Fscr)$ for any proper non-zero subbundle $\Fscr' \subset \Fscr$.
\end{theorem}

\begin{theorem}[Kuleshov-Orlov \cite{KuleshovOrlov}, Theorem 5.2]\label{thm:KOrigidsumex} Any rigid bundle $\Tscr$ on a del Pezzo surface $X$ (i.e. $\Ext_X^1(\Tscr, \Tscr) = 0$) is a direct sum of exceptional bundles. \end{theorem}
\begin{conventions} \label{conv:tilde}
  We adopt the following notations. 
  If $(E,F)$ is an exceptional pair of sheaves on a del Pezzo surface $X$,  then $L_E F = F'[a]$ and $R_F E = E'[b]$, where $F', E'$ are sheaves and $a, b \in \ZZ$.  We denote $\widetilde{L_E}F = F'$ and $\widetilde{R_F} E = E'$.
  Hence if we define
\begin{align*}
\tilde{\sigma}_{i}(E_0, \dots,E_{n-1}) = (E_0, \dots, \widetilde{L_{E_{i-1}}}E_i,E_{i-1} \dots, E_{n-1}), \\ 
\tilde{\sigma}_{i}^{-1}(E_0, \dots,E_{n-1}) = (E_0, \dots, E_i, \widetilde{R_{E_i}}E_{i-1},\dots, E_{n-1}). 
\end{align*}
then we obtain a braid group action on the
set of exceptional collections of sheaves on $X$.
In addition, we define 
\begin{align*}
\tilde{\sigma}_n(E_0, \dots, E_{n-1}) = (E_{n-1} \otimes \omega_X,E_1, \dots, E_{n-2}, \widetilde{L_{E_{n-1}}} (E_0 \otimes \omega_X^{-1})), \\
\tilde{\sigma}_n^{-1}(E_0, \dots, E_{n-1}) = ((\widetilde{R_{E_0\otimes \omega_X^{-1}}} E_{n-1}) \otimes \omega_X, E_1,\dots, E_{n-2},  E_0 \otimes \omega_X^{-1}).
\end{align*}
Here $\tilde{\sigma}_n^{\pm 1}$ is the composition of the right rotation (see Definition \ref{def:rotation}), followed by $\tilde{\sigma}_{n-1}^{\pm 1}$ and the left rotation. 
\end{conventions}
\subsection{The symmetry group of a del Pezzo surface}
\label{sec:symmetrygroup}
It is natural to ask to what extent an exceptional collection is determined by its Gram matrix (see Definition \ref{def:gram_matrix}). In Proposition \ref{prop:symmetry}
below we answer this question for del Pezzo surfaces.
\begin{definition}For $X$ a del Pezzo surface,  denote the stabiliser of the canonical class $K_X$ in $O(\Pic(X))$ by 
\[ S(X) := \{ f \in O(\Pic(X) \, | \, f(K_X) = K_X \}. 
\]  A \emph{root} is an element $\rho$ of $\Pic(X)$ such that $K_X\cdot \rho=0$ and $\rho\cdot \rho=-2$. \end{definition} 

For $X_n=\Bl_{p_1,\dots,p_n}\PP^2$ with $0 \leq n \leq 8$ we use the standard basis
\begin{align}\label{eq:picbasis}
(H,E_1,\dots,E_n)
\end{align}
of $\Pic(X_n)$, where $H$ is the pullback of the line class $\Oscr_{\PP^2}(1)$ and $E_i$ are the classes of the exceptional divisors.  In this basis the intersection form is given by
\[
H^2=1,\qquad E_i^2=-1,\qquad H\cdot E_i=0,\qquad E_i\cdot E_j=0 \ \ (i\neq j),
\]
and the canonical class is
\[
K_X=-3H+E_1+\cdots+E_n,
\]
so that $K_X^2=9-n$.

For $X=\PP^1\times \PP^1$  we use the basis $(F_1,F_2)$ given by the classes of the two rulings.  In this basis one has $
F_1^2=F_2^2=0, F_1\cdot F_2=1$ and $K_X=-2F_1-2F_2$.

\begin{theorem}[{\cite[Theorem 23.9]{ManinCubicForms}}]\label{thm:manin} For $X = \Bl_{r}(\PP^2)$ with $3 \leq r \leq 8$ the set of roots $R_r$ forms a root system equal to $A_1 \times A_2,A_4,D_5,E_6,E_7,E_8$ respectively.
For $n\ge 3$, the simple roots may be taken to be
\[
\alpha_1=H-E_1-E_2-E_3,\qquad
\alpha_i=E_{i-1}-E_i \ \ (2\leq i\leq n).
\]Moreover  the group $S(X)$ coincides with the corresponding Weyl group $W(R_ r)$.
\end{theorem}
It turns out that Theorem \ref{thm:manin} also holds for the remaining del Pezzo surfaces. This is e.g.\ stated without proof in \cite[\S2.3]{LLR}.
\begin{lemma} 
For $X =  \Bl_{2}(\PP^2)$ and $X = \PP^1 \times \PP^1$ one has $S(X) = W(A_1) \cong \ZZ/2\ZZ$
and the simple roots may respectively be taken to be $F_1-F_2$ and $E_1-E_2$.
 For $X = \FF_1$ and $X = \PP^2$ the group $S(X)$ is trivial.
\end{lemma} 
\begin{proof} First consider the case $X = \Bl_2(\PP^2)$.  Consider the basis $H, E_1,E_2$ of $\Pic(X)$ as in \refeq{eq:picbasis}.  By definition $f \in S(X)$ is an automorphism of $\Pic(X)$ preserving the intersection form and the canonical class $K_X$,  thus it should also preserve the set $\Cscr := \{C \in \Pic(X) \,  | \,  C^2 = -1, C \cdot K_X = -1 \}$ of the classes of $(-1)$-curves.  Let $C = aH + bE_1 + cE_2$ be in $\Cscr$.  Since $K_X = -3H + E_1 + E_2$, one has $C^2 = a^2 - b^2 -c^2 = -1$ and $C \cdot K_X = -3a - b - c = -1$.  It is a straightforward exercise to show that the only integer solutions to this system of equations are $(a,b,c) = (0,1,0),(0,0,1)$ and $(1,-1,-1)$, i.e.  the set $\Cscr$ has $3$ elements $E_1$,$E_2$ and $L:= H-E_1-E_2$ (the strict transform of the line through the $2$ points of the blow-up).  Hence $f$ induces a permutation on this set of 3 elements.  Moreover,  since the $3$ elements  $E_1,E_2, L$ generate the whole $\Pic(X)$,  this permutation determines an automorphism uniquely.  We have $E_1 \cdot E_2 = 0$, $E_1 \cdot L = 1$ and $E_2 \cdot L = 1$,  so the permutation induced by $f$ can either fix both $E_1$ and $E_2$ or swap them.  Swapping $E_1$ and $E_2$ is exactly the reflection with respect to the root $E_1-E_2$.  For $X = \FF_1$ one has $\Pic(X) = \ZZ H \oplus \ZZ E$, where $E$ is the unique $(-1)$-class and $K_X = -3H = E$,  so a similar argument shows that any $f \in S(X)$ is forced to be the identity.  For $X = \PP^1 \times \PP^1$ the orthogonal complement to $K_X$ is generated by the simple root $F_1 - F_2$, where $F_1,F_2$ are the classes of the two ruling,  hence any $f \in S(X)$ can either be the identity or swap $F_1$ and $F_2$.  For $X = \PP^2$ the claim is trivial since $\Pic(X)$ has rank $1$. 
\end{proof}
\begin{definition}
  A \emph{numerical full exceptional collection} on $X$ is an ordered basis $(e_0,\dots,e_{n-1})$ in $K_0^{\num}(X)$ such that $\chi(e_i,e_i)=1$, $\chi(e_i,e_j)=0$ for $i>j$.
\end{definition}
\begin{proposition} \label{prop:numericalcol} Let $(e_0,\ldots,e_{n-1})$ be a full exceptional collection in  $K_0^{\num}(X)$ such that $r(e_i)\ge 0$ for all $i$. Then there exists a unique
 full exceptional collection $(E_0,\ldots,E_{n-1})$ of sheaves on $X$ such that $[E_i]=e_i$ for all $i$.
\end{proposition}
\begin{proof}
  By \cite[Theorem 10.9]{Perling} any numerical full exceptional collection can be transformed by mutations into a collection consisting of objects of rank $1$.  Such a collection is clearly realised by a collection of line bundles on $X$.
  
  Applying the inverse of the corresponding sequence of mutations on the level of derived categories,  we conclude that every numerical full exceptional collection on a del Pezzo surface is realised by an actual full exceptional collection.  Moreover, such a collection is unique up to even shifts (see e.g. \cite[\S 1.3]{KarpovNogin}). Since
  we require that the lift consist of sheaves, this fixes the shift.
\end{proof}
\begin{remark} It is not true that every numerical exceptional class can be lifted to an exceptional object. The class $18[O]-9[O(1)]+[O(2)]$ in $K_0^{\num}(\PP^2)$ is numerically exceptional but its rank is 10, which is not a Markov number. So it does not lift by \cite{Rudakov}.
\end{remark}
The following proposition shows that the exact sequence in the statement of \cite[Lemma 2.10]{Krah} is canonically split (this can also be extracted from the proof in loc.\ cit.).
\begin{proposition} \label{prop:SXaction}
  Let $S(X)$ act on $K_0^{\num}(X)$ in the following way. First identify
  \[
    h: K_0^{\num}(X)\r \ZZ \oplus \Pic(X) \oplus \ZZ:e \mapsto (r(e),c_1(e),\chi(e)) 
  \]
  and then for $s \in S(X)$, $e \in K_0^{\num}(X)$ put $s(e) = (r(e),s c_1(e), \chi(e))$. This action is compatible with the Euler form and moreover it preserves
  the rank and degree of a class and hence also its slope.
\end{proposition}
\begin{proof}
  That the action preserves the rank is straight from the definition. The degree of an object $e$ is $(-K_X)\cdot c_1(e)$ which is clearly also preserved.
For any $e,f \in K^{\num}_0(X)$ one obtains from Riemann-Roch 
\begin{equation}
  \label{eq:chi_rr}
\chi(e,f) = r(e)\chi(f) + r(f)\chi(e) - r(e)r(f) - c_1(e)\cdot c_1(f) + r(f)(K_X\cdot c_1(e)).
\end{equation}
Since $S(X)$ preserves the intersection form and $K_X$,  we get $\chi(e,f) = \chi(s(e),s(f))$ for any $s \in S(X)$.
\end{proof}
  
\begin{corollary} \label{cor:SXcoll}
  The action of $S(X)$ on $K_0^{\num}(X)$ induces an action on the set of full exceptional collections consisting of sheaves.
\end{corollary} 

\begin{proof}
  The corresponding numerical statement follows from Proposition \ref{prop:SXaction}.
  We then invoke  Proposition \ref{prop:numericalcol}.
\end{proof}
\begin{proposition} \label{prop:symmetry}
  Let $\EE$ and $\EE'$ be two full exceptional collections of sheaves on $X$ having the same Gram matrix.  Then $\EE$ and $\EE'$ are related by tensoring with a line bundle and the action of the symmetry group $S(X)$.
\end{proposition}
\begin{proof} Sending $\EE$ to $\EE'$ defines an $\alpha$ automorphism of $K_0^{\num}(X)$ compatible with the Euler form. By \cite[Lemma 2.10]{Krah} together with the splitting
  constructed in Proposition \ref{prop:SXaction}
  of the exact sequence given in loc.\ cit.\ we obtain that $\alpha$ can be written as a composition of an element of $S(X)$, tensoring with a line bundle and possibly
  a global sign change. However this last factor cannot occur since both $\EE$ and $\EE'$ consist of sheaves and hence they have total rank $>0$.
  This
  gives what we want.
  \end{proof}
 \subsection{Geometric helices on del Pezzo surfaces}
\label{sec:geom_helix}
\begin{definition}\label{def:verystrong} We say that a full exceptional collection of sheaves $\EE = (E_0, \dots,E_{n-1})$ is \emph{very strong} if it satisfies
\begin{equation}\label{eq:superstrong}
\Ext^l_X(E_i,E_j \otimes \omega_X^{-k}) = 0 \text{ for all } k \geq 0,\ l \neq 0,\ i,\ j. 
\end{equation} 
\end{definition} 

\begin{remark}\label{rem:verystrongslopes} An exceptional collection $(\Escr_0, \dots, \Escr_{n-1})$ is very strong if and only if
\[
\mu(\Escr_0) \leq \dots \leq \mu(\Escr_{n-1}) \leq \mu(\Escr_0 \otimes \omega_X^{-1}) = \mu(\Escr_0) + K_X^2.
\]
The ``only if'' direction follows immediately from Proposition \ref{pr:slopes}. 
For the ``if'' direction we first note that \eqref{eq:superstrong} for $l=2$ follows from Serre
duality and Theorem \ref{thm:mustable}.
For $l=1$ we observe that if $\mu(F)>\mu(E)$ then by restriction to a smooth anti-canonical divisor
as in the proof of \cite[Lemma 3.7]{KuleshovOrlov} we have a surjection $\Ext^1(E,F\otimes \omega_X)
\r \Ext^1(E,F)$. This allows us to go from $k$ to $k+1$ in \eqref{eq:superstrong}, which reduces us
to the case $k=0$ where we can invoke Proposition \ref{pr:slopes}.
 \end{remark}
 From this observation we obtain the following corollary.
 \begin{corollary} The action exhibited in Corollary \ref{cor:SXcoll} restrict to an action of $S(X)$ on very strong exceptional collections.
 \end{corollary}
 \begin{proof} This follows from Remark \ref{rem:verystrongslopes} together with the fact that as stated in Proposition \ref{prop:SXaction} the action
   of $S(X)$ on $K_0^{\num}(X)$ preserves slopes.
   \end{proof}
\begin{proposition}\label{pr:onlybundles} Let $\EE = (E_0, \dots,E_{n-1})$ be a very strong exceptional collection on a del Pezzo surface $X$. Then each $E_i$, $i=0, \dots, n-1$ is a vector bundle. 
\end{proposition} 

\begin{proof} By Proposition \ref{pr:KOexobjs} every exceptional sheaf is either a vector bundle or of the form $\Oscr_C(d)$,  where $C \subset X$ is a $(-1)$-curve.   Suppose $E_i = \Oscr_C(d)$ for some $i$.  By assumption the collection $\EE$ is full,  so it cannot consist entirely of torsion sheaves.  Hence without loss of generality we can assume that $E_{i+1} = E$ is a vector bundle.  By Lemma 3.1 of \cite{KuleshovOrlov} the restriction $E|_C$ of an exceptional vector bundle $E$ to a $(-1)$-curve $C$ has the form $n\Oscr_C(t) \oplus m \Oscr_C(t+1)$.  Then the condition (\ref{eq:superstrong}) implies that for all $k \geq 0$ we have
\begin{align*} 0 =  \Ext^{1}_X(E_i,E_{i+1} \otimes \omega_X^{-k}) = \Ext^1_X(\Oscr_C(d),E \otimes \omega_X^{-k}) \\
= \Hom_C(\Oscr_C(d-k-1),E|_C) = \Hom_S(\Oscr_C(d-k-1),n\Oscr_C(t) \oplus m \Oscr_C(t+1)) 
\end{align*}
But for $k$ large enough the last $\Hom$ is clearly non-zero, so we get a contradiction. 
\end{proof}
A full exceptional collection of sheaves $\EE = (E_0, \dots, E_{n-1})$ is very strong if and only if it generates a \emph{geometric helix} in the sense of \cite{BridgelandStern}: 
\begin{definition}\label{def:helix} Let $\EE = (E_0, \dots, E_{n-1})$ be an exceptional collection.  We define the corresponding infinite \emph{helix} $\HH = \{E_j\}_{j \in \ZZ}$ generated by $\EE$ by setting $E_{i-n} = E_i \otimes \omega_X$ for all $i \in \ZZ$.  An exceptional collection $(E_j, \dot,  E_{j+n-1})$ for any $j \in \ZZ$ is called a \emph{thread} of the helix $\HH$.  A helix $\HH = \{E_j\}_{j \in \ZZ}$ is said to be \emph{geometric} if $\Hom(E_i,E_j[k]) = 0$ for all $i, j$ and $k\neq 0$.  \end{definition}
\begin{remark}
  \label{rem:helixshift} 
  Definition \ref{def:helix} does not require a helix to consist of sheaves.  However,  a geometric helix on a del Pezzo surface automatically consists sheaves,  modulo a simultaneous shift.  Indeed,  suppose $E$ and $F[k]$,  with $E$ and $F$ sheaves and $k \neq 0$, are two exceptional objects in the same thread of a geometric helix. Since $\mu(E\otimes \omega^{-1}_X) = K^2_X + \mu(E) > \mu(E)$, one has either $\mu(F) > \mu(E)$ or $\mu(F) < \mu(E \otimes \omega^{-1}_X$ (note that $\mu(F) = \mu(E)$ automatically implies the second case and $\mu(F) = \mu(E \otimes \omega^{-1}_X)$ automatically implies the first).  If $\mu(F) > \mu(E)$,  Proposition \ref{pr:slopes} implies that $\Ext_X^{-k}(E,F[k]) = \Hom_X(E,F) \neq 0$ which contradicts the helix being geometric.  Similarly, if $\mu(F) < \mu(E \otimes \omega^{-1}_X)$, then $\Ext^{k}_X(F[k],E\otimes \omega^{-1}_X) = \Hom_X(F,E\otimes  \omega^{-1}_X) \neq 0$.  \end{remark} 

In what follows we will usually talk of exceptional collections rather than the helices they generate,  but nevertheless keep in mind that the resulting graded algebra is independent on the choice of a thread in the helix, up to regrading (see \S\ref{sec:notation}) which does not change the category of graded modules. Therefore we are always allowed to ``rotate'' an exceptional collection: 

\begin{definition}\label{def:rotation} The \emph{left rotation} of an exceptional collection $\EE = (E_0, \dots,E_{n-1})$ is defined to be $(E_{n-1}\otimes \omega_X, E_0, \dots,E_{n-2})$.  Similarly,  the \emph{right rotation} of $\EE$ is $(E_1, \dots,E_{n-1}, E_0 \otimes \omega_X^{-1})$.  By \cite[\S 4]{MR992977} the left rotation can be achieved by a composition of left mutations as $(\widetilde{L_{E_0}}\dots\widetilde{L_{E_{n-2}}} E_{n-1}, E_0, \dots, E_{n-2})$. Likewise, the the right rotation can be expressed as $(E_1, \dots,E_{n-1}, \widetilde{R_{E_{n-1}}}\dots \widetilde{R_{E_1}} E_0)$. 
\end{definition} 

 \section{NCCRs and very strong exceptional collections}\label{sec:allnccrs}
\subsection{Non-commutative crepant resolutions}
\label{sec:NCCRs}
\begin{definition}[Van den Bergh \cite{VdB}]\label{def:nccr}Let $R$ be a normal Gorenstein domain.  An $R$-algebra $\Lambda$ is a \emph{non-commutative crepant resolution} (NCCR) of $R$ if 
\begin{enumerate}
\item $\Lambda = \End_R(M)$ for some non-zero $M \in \reff(R)$,  where $\reff(R)$ denotes the category of finitely generated reflexive $R$-modules.
\item $\Lambda$ is Cohen-Macaulay as an $R$-module and has finite global dimension. 
\end{enumerate}
\end{definition}
We will mainly use this definition in two situations:
\begin{itemize}
\item \emph{``the complete context''}: $R$ is a complete local ring.
\item \emph{``the graded context''}:  $R$ is a graded ring of the form $k+R_1+R_2+\cdots$ and $M$ is a graded $R$-module.
\end{itemize}
\begin{convention}
If we are in the complete context then we implicitly assume that all operations are completed. This applies for example to the path algebra of a quiver.
If we are in the graded context then we implicitly assume that all constructions take the grading (which is extra data) into account.
\end{convention}
\begin{definition}
  \label{def:basic}
  Being in the graded or complete context we say that $\Lambda=\End_R(M)$ is \emph{basic} if all simples are one-dimensional, or equivalently, if the indecomposable summands of $M$
  occur with multiplicity one.
\end{definition}
\begin{remark}
  In the graded or complete context all NCCRs are Morita equivalent to a basic one.
\end{remark}
\begin{remark} If $\Lambda$, $\Gamma$ are basic graded NCCRs, then $\Lambda$, $\Gamma$ are graded Morita equivalent if and only if they can be obtained from each other
  by regrading (see \S\ref{sec:notation}).
\end{remark}
 \subsection{From very strong exceptional collections to NCCRs}
  Let $X$ be a del Pezzo surface,
 $Y=\Tot(\omega_X)$, and let $\pi:Y\to X$ be the projection.  Let $\EE=(E_0,\dots,E_{n-1})$ be a very strong exceptional collection of sheaves and set $\Escr=\bigoplus_{i=0}^d E_{n-1}$.  Define
 \begin{equation}
   \label{eq:tilting_algebra}
B(\EE):=\End_{\Dscr^b(Y)}(\pi^*(\Escr)).
\end{equation}
which we view as a graded algebra (the grading being inherited from the $G_m$-action on $Y$).

If $\HH$ is a geometric helix generated by $\EE$,  then by \cite[Theorem 3.6]{BridgelandStern} $B(\EE)$ is isomorphic to the \emph{rolled-up helix algebra} $B(\HH)$, which is defined to be the graded subalgebra of the graded algebra $\bigoplus_{k \geq 0} \prod_{j - i =k} \Hom_X(E_i, E_j)$ consisting of elements invariant under the $\ZZ$-action induced by $- \otimes \omega_X$.
Hence following Bridgeland and Stern we sometimes refer to $B(\EE)$ as defined in \eqref{eq:tilting_algebra} as a rolled-up helix algebra as well. 

\begin{theorem}[{\cite[\S 7]{VdB}},  see also the proof of Theorem \ref{thm:nccrfromFEC}] The object $\pi^*(\Escr)$ is a tilting bundle in $\Dscr^b(Y)$ and $B(\EE)$ is an NCCR of the affine cone $R_X$ given by $\omega_X$.  \end{theorem}

Hence,  we have an equivalence of categories
\begin{equation}
\label{eq:derived_equivalence}
\RHom_{\Dscr^b(Y)}(\pi^*(\Escr),-):\Dscr^b(Y)\to \Dscr^b(B(\EE)).
\end{equation}
In particular, $B(\EE)$ is $3$-Calabi--Yau (see \cite[Proposition 4.1]{Bridgeland}, \cite[Theorem 3.6]{BridgelandStern})
and can be encoded as a Jacobi algebra. We will discuss this in the next section.

\subsection{Relation with Jacobi algebras}
\label{sec:jacobi}
\begin{definition}\label{def:QPJacobi}
Let $Q$ be a quiver and let $kQ$ its path algebra.  A \emph{potential} on $Q$ is an element
\[
W \in kQ/[kQ,kQ].
\]
The \emph{Jacobi algebra} associated to $(Q,W)$ is
\[
\Jscr(Q,W):=kQ/(\partial W),
\]
where $(\partial W)$ denotes the two-sided ideal generated by the cyclic derivatives $\partial_a W$ for all arrows $a\in Q_1$.
\end{definition}
\begin{remark} In the complete context (see \S\ref{sec:NCCRs}) by our conventions $kQ$ is completed (this is the setting of \cite{DWZ}) and so is the space of commutators
  $[kQ,kQ]$. This implies that $W$ is represented by
  a (usually) infinite convergent linear combination of cyclic paths.
In the graded context we assume that $Q$ is graded (i.e.\ the arrows are equipped with a degree) and $W$ is homogeneous.
\end{remark}
Now let $\Lambda$ be a $3$-Calabi-Yau algebra and assume that we are in the complete or graded context.  Then  \cite[Theorem 3.1]{Bocklandt} (see also \cite[Theorems A,B]{VdBCY}) implies that $\Lambda$ can be represented as the Jacobi algebra $\Jscr(Q,W)$ associated to a quiver $Q$ with a homogeneous potential $W$ in the sense of Definition \ref{def:QPJacobi}.  
\begin{remark}\label{rem:Ext1graph}
  Note that $Q$ is the $\Ext^1$-graph of $\Lambda$. Let $(S_i)_{i=0}^{n-1}$ be the non-isomorphic simples of $\Lambda$, where in the graded
context,  $S_i$ is taken to be in degree zero. Then $Q=\{0,\ldots, n-1\}$ and the  arrows from $i$ to $j$ correspond to a basis of $\Ext^1_\Lambda(S_i,S_j)$.
\end{remark}

If $(\widecheck{E_j})_j$ denotes the left dual collection then the equivalence \eqref{eq:derived_equivalence} maps $\iota_*(\widecheck{E_j})$,  where $\iota: X \hookrightarrow Y$ is the inclusion of the zero-section,  to the simple $B(\EE)$-modules (see also \cite[\S 3]{Bridgeland}).  
Using some elementary algebra based on Remark \ref{rem:Ext1graph} this leads to the following description of $Q$:
\begin{proposition}
\label{prop:quiver_rule}
The vertices of $Q$ are in bijection with objects of the exceptional collection $\EE$, and the arrows are obtained from the left/right dual (see Remark \ref{rem:left_right_dual})
exceptional collection $\widecheck{\EE}=(\widecheck{E_{n-1}}, \ldots \widecheck{E_{0}})$ (see Proposition \ref{def:dualcollection}) as follows:
\begin{enumerate}
\item There are no loops, i.e.\ there are no arrows $i\r j$ if $i=j$.
\item If $i<j$ then the arrows correspond to a basis for $\Ext^1(\widecheck{E_j},\widecheck{E_i})^*$. Those arrows have degree zero.
\item if $i>j$ then the arrows correspond to a basis for $\Ext^2(\widecheck{E_i},\widecheck{E_j})$. Those arrows have degree one.
\end{enumerate}
Moreover the superpotential $W$ has degree one (for the construction of the superpotential see \cite[Theorem 6.3]{KellerDeformed}).
\end{proposition}
For use below we note that $Q$ can also be obtained from $\widecheck{\EE}$ in the following ``numerical'' way.
\begin{lemma}\label{lm:arrowschiY}
Let $Y\subset X$ be a smooth anti-canonical divisor.  Then the number of arrows $i\r j$ in $Q$ is equal to $\chi_Y(\widecheck{E_{i,Y}},\widecheck{E_{j,Y}})$, with
the convention that a negative number corresponds to arrows from $j$ to $i$.
\end{lemma}
\begin{proof} 
Since $A(\EE):=\End(\oplus_i E_i)$ has global dimension two (see e.g. \cite[Lemma A.3]{dTdVVdB}) the only possible $\Ext$'s between the simples
of $A(\EE)$ are in degree $\le 2$. The simples of $A(\EE)$ correspond to the elements of $\widecheck{\EE}$. Hence we have $\Ext^k(\widecheck{E_i},\widecheck{E_j})=0$
for $k=0$ and $k\ge 3$. By Proposition \ref{pr:slopes} there is at most a single non-zero $\Ext$ between any $\widecheck{E_i}$, $\widecheck{E_j}$. It follows
that the number of arrows $i\r j$ in $Q$ (counting arrows from $j\r i$ as negative) is equal to
\[
\begin{cases}
0 &\text{if $i=j,$}\\
-\chi(\widecheck{E_j},\widecheck{E_i})&\text{if $i<j,$}\\
\chi(\widecheck{E_i},\widecheck{E_j})&\text{if $i>j$.}
\end{cases}
\]
By \cite[Lemma 3.12]{NVdB2} this may be rewritten as
\[
\begin{cases}
0 &\text{if $i=j$}\\
-\chi_Y(\widecheck{E_{j,Y}},\widecheck{E_{i,Y}})&\text{if $i<j$}\\
\chi_Y(\widecheck{E_{i,Y}},\widecheck{E_{j,Y}})&\text{if $i>j$}
\end{cases}
\]
We finish with the observation that $\chi_Y(\widecheck{E_{i,Y}},\widecheck{E_{i,Y}}) = 0$ and $\chi_Y(\widecheck{E_{i,Y}},\widecheck{E_{j,Y}})=-\chi_Y(\widecheck{E_{j,Y}},\widecheck{E_{i,Y}})$.
\end{proof}
\begin{remark} \label{rem:no_loops_or_two_cycles}
  It follows from the explicit description of $Q$ that it has no
  loops or oriented 2-cycles. This was first observed in \cite[Proposition 7.5]{BridgelandStern}.  An independent proof was given in \cite[Theorem 4.2.2]{dTdVVdB}.
  \end{remark} 
 \subsection{From NCCRs to very strong exceptional collections}
 
 In this section we prove the following result:
 \begin{theorem}\label{thm:nccrfromFEC} Let $X$ be a del Pezzo surface and $R_X :=  \bigoplus_{k \leq 0} \Gamma(X, \omega_X^{\otimes k})$.
   If $\Lambda$ is a basic (see Definition \ref{def:basic})
   graded NCCR of $R_X$ then   there exists a
  very strong exceptional collection $\EE = (E_0, \dots, E_{n-1})$ on $X$ such that $\Lambda$ is obtained by regrading (see \S\ref{sec:notation})
  from $\End_{\Dscr^b(Y)}(\pi^*(\Escr))$, where $\Escr = \oplus_{i=0}^{n-1} E_i$ and $\pi: Y \to X$ the canonical projection from $Y = \Tot(\omega_X)$.
\end{theorem}
Using the following proposition one may also use this result to describe NCCRs of $\widehat{R_X}$.
 \begin{proposition} \label{prop:completion}
   Let $R=k+R_1+R_2+\cdots$ be a three-dimensional graded normal Gorenstein domain with an isolated singularity. Then a NCCR of the completion $\widehat{R}$ is the completion of a graded NCCR of $R$.
\end{proposition}
 First we require some preparation.
 Let $X$ be a smooth projective variety and $\Lscr \in \Pic(X)$ an ample line bundle.  Let $R := R_X :=  \Gamma^*(X,\Lscr) =\bigoplus_{k \in \ZZ} \Gamma(X, \Lscr^{\otimes k})$ be the corresponding graded ring of sections.   It is known that $R$ is a finitely generated normal $k$-algebra.  The affine variety $Z := \Spec(R_X)$ is a cone over $X$  with the unique singularity at the origin.  For $M \in \gr(R_X)$ denote by $\Mscr = \widetilde{M} \in \coh(X)$ the corresponding coherent sheaf on $X$.  

\begin{proposition}\label{pr:strigid} Let $\Mscr \in \coh(X)$ be such that $\Mscr = \widetilde{M}$,  where $M$ is a reflexive $R$-module with $\End_R(M)$ Cohen-Macaulay.  Assume that  $\Mscr$ is a vector bundle.  Then 
\[
\Ext^i_X(\Mscr,\Mscr \otimes_X \Lscr^{\otimes k})	= 0
\]
 for any $0 < i < \dim(X)$, $k \in \ZZ$.  In particular,  if $\dim X \geq 2$,  $\Mscr$ is rigid. 
\end{proposition} 

\begin{proof} 
For any $0 < i < \dim(X)$ we have 
\begin{align*}
\bigoplus_{k \in \ZZ} \Ext^i_X(\Mscr,\Mscr \otimes_X \Lscr^{\otimes k})	= \bigoplus_{k \in \ZZ} H^i(X, \HHom_X(\Mscr,\Mscr)\otimes_X \Lscr^{\otimes k})  \\
 = H^{i+1}_{R_{\geq 1}} (\Gamma^*(X,\HHom_X(\Mscr,\Mscr)) 
 = H^{i+1}_{R_{\geq 1}} (\Hom_R(M,M)) = 0,
 \end{align*}
where in the first equality we used that $\Mscr$ is a vector bundle and in the last equality we used $i+1 < \dim(X) + 1 = \dim(R)$ and the fact that $\Hom_R(M,M)$ is Cohen-Macaulay (also see \cite[Theorem A4.1]{Eisenbud} for the statement relating global and local cohomology).\footnote{Alternatively to prove Proposition \ref{pr:strigid} one may use that $\widetilde{\End_R(M)} \cong {\mathcal End}_X(\Mscr)$ is arithmetically Cohen–Macaulay by \cite[Proposition 1.2]{Beauville}.  We thank Alexander Kuznetsov for pointing this out to us.}
\end{proof} 

\begin{remark} If $\dim X = 2$,  the condition that $\Mscr$ is a vector bundle is automatic,  because in this case reflexive sheaves are locally free and $\widetilde{(-)}$ preserves reflexivity.  \end{remark} 

Now assume that $X$ is a del Pezzo surface and let $\Lscr = \omega_X^{-1}$.  Then $R$ is Gorenstein.  Let ${M \in \reff(R)}$ be a reflexive module such that $\Lambda := \End_R(M)$ is  an NCCR of $R$.  The following theorem implies that $M$ is rigid.   

\begin{theorem}[Iyama-Reiten, {\cite[Theorem 8.15]{IyamaReiten}}] \label{thm:IR}
  Let $R$ be a  three-dimensional normal Gorenstein domain with isolated singularities and let $M \in \reff(R)$ a reflexive module such that $\End_R(M)$ is Cohen-Macaulay.  Then $M$ is rigid,  i.e.  $\Ext_R^1(M,M) = 0$. 
\end{theorem}
For convenience of the reader we give a direct proof. 
\begin{proof} Consider an exact sequence of $R$-modules with $P$ finitely generated projective
\begin{align*}
0 \to N \to P \to M \to 0.
\end{align*} 
Applying $\Hom_R(-,M)$ we get an exact sequence
\begin{align*}
0 \to \End_R(M) \xrightarrow{\varphi} \Hom_R(P,M) \to \Hom_R(N, M) \to \Ext^1_R(M,M) \to 0.
\end{align*} 
By assumption the module $\End_R(M)$ is Cohen-Macaulay,  hence it has depth $3$.  Since $M$ is reflexive,  so are $\Hom_R(P,M)$ and $\Hom_R(N, M)$.  Therefore these modules have depth at least $2$.  Let $L = \coker(\varphi)$. Then 
\begin{align*}
\depth(L) \geq \min \{\depth(\End_R(M))-1,\depth ( \Hom_R(P,M)) \} \geq 2.
\end{align*} 
Assuming $\Ext^1_R(M,M)$ is non-zero,  the same argument then gives \[
\depth(\Ext^1_R(M,M)) \geq 1.
\]  But this is impossible since $\Ext^1_R(M,M)$ has finite length as an $R$-module.   Indeed,  this follows from the fact that over algebras of global dimension $2$ every reflexive module is projective,  hence $M_\pp$ is projective over $R_\pp$ for every $\pp$ in the regular part of $\Spec(R)$. 
\end{proof} 
It is well-known (see e.g. \cite[Exercise 8.8]{HuybrechtsFM}) that for $Z$ a smooth projective Calabi-Yau variety (i.e. $\omega_Z \simeq \Oscr_Z$),  its derived category $\Dscr^b(Z)$ admits no non-trivial semiorthogonal decompositions.  The reason is that $\Dscr^b(Z)$ has trivial Serre functor (up to shift),  so any semiorthogonal decomposition is automatically orthogonal.  We will require an analogous statement for $Y = \Tot(\omega_X)$ with $X$ a del Pezzo surface or,  more generally,  a smooth projective Fano variety.  As mentioned above,  $Y$ is a \emph{non-compact} Calabi-Yau variety,  hence its derived category does not admit a Serre functor unless we restrict to the subcategory of objects with compact support.  However,  $\Dscr^b(Y)$ has trivial \emph{Serre functor relative to $R_X$} which means that for every $F, G \in \Dscr^b(Y)$ there is a functorial isomorphism in $\Dscr^b(R_X)$:
\[
\RHom_{R_X}(\RHom_{Y}(F,G),R_X) \cong \RHom_Y(G,F)
\]
See e.g.  \cite[Lemma 4.13]{IyamaWemyss14}.  Hence we can conclude in the same manner:

\begin{lemma}\label{lm:noSOD} Let $X$ be a smooth projective Fano variety and $Y = \Tot(\omega_X)$.  Then $\Dscr^b(Y)$ admits no non-trivial $R$-linear semiorthogonal decompositions. \end{lemma}

\begin{remark} If $R$ is a normal Gorenstein domain and $\Lambda$ is an NCCR of $R$.  Then one can also show directly that $\Lambda$ has trivial Serre functor relative to $R$, i.e.  for every $M,N \in  \Dscr^b(\Lambda)$ there is a functorial isomorphism 
\[
\RHom_\Lambda(M,N) \cong \RHom_R(\RHom_\Lambda(N,M),R).
\]
Indeed,  it is sufficient to consider $M = N = \Lambda$ since $\Lambda$ has finite global dimension and thus $\Dscr^b(\Lambda) = \Perf(\Lambda)$.  Then the statement reduces to $\Lambda \cong \RHom_R(\Lambda,R)$ which in turn follows from the fact that $\Lambda$ is Cohen-Macaulay as an $R$-module.
 \end{remark} 
Now we are ready to prove Theorem \ref{thm:nccrfromFEC}:
\begin{proof}[Proof of Theorem \ref{thm:nccrfromFEC}]
  Assume that $ \Lambda = \End_{R_X}(M)$ is a basic graded NCCR of $R_X$. Denote $\Mscr = \widetilde{M} \in \coh(X)$  the corresponding coherent sheaf.
  Then Proposition \ref{pr:strigid} implies that $\Mscr$ is a rigid vector bundle. 

  Note also that $M$ and $\Gamma^*(X,\Mscr)$ are isomorphic in codimension one and reflexive,  hence $M \cong \Gamma^*(X,\Mscr)$. Moreover $M(t)=\Gamma^\ast(X,\Mscr\otimes \omega_X^{-t})$. Hence, since $M$ has no repeated indecomposable summands up to shifts, $\Mscr$ has no indecomposable summands related by tensoring with a a power of $\omega_X$.

  Recall that by Theorem \ref{thm:KOrigidsumex} every rigid vector bundle on a del Pezzo surface is a direct sum of exceptional vector bundles, i.e.
\[
\Mscr =  \Mscr_0 \oplus \dots \oplus \Mscr_{n-1},
\]
where $\Mscr_j \in \coh(X)$ are exceptional.
Let $\Mscr_j': = \Mscr_j \otimes \omega_X^{t_j}$ where $t_j \in \ZZ$
is such that $\mu(\Mscr_j') \in [0, K^2_X)$. Then the $(\Mscr'_j)_j$
are all distinct.  We further reorder the $(\Mscr_j')_j$ in such a way
that the slopes are non-decreasing, i.e.
$\mu(\Mscr_i') \leq \mu(\Mscr_j')$ if $i < j$.  Let
$\Mscr' = \bigoplus_{j=0}^{n-1} \Mscr_j'$.  Note that
$\Gamma^*(X,\Mscr_j) = \Gamma^*(X,\Mscr_j')(-t_j)$, so forgetting the
grading we can replace $\Mscr$ by $\Mscr'$.  Proposition
\ref{pr:strigid} implies that $\Mscr'$ is still rigid.  We claim that
$(\Mscr_0', \dots, \Mscr_{n-1}')$ is a very strong exceptional
collection on $X$.  Indeed, since $\Mscr'$ is rigid, we have
\[
\Ext^1_X(\Mscr_i', \Mscr_j') = \Ext^1_X(\Mscr_j', \Mscr_i') = 0 \text{ for any } 0 \leq i,j \leq n-1.
\]
Since exceptional bundles on del Pezzo surfaces are $\mu$-stable (Theorem \ref{thm:mustable}),  the condition on the order of slopes gives $\Hom_X(\Mscr_j', \Mscr_i') = 0$ whenever $i < j$.  It remains to check that $\Ext^2_X(\Mscr_j', \Mscr_i') = 0$ for $i < j$.  Indeed,  by Serre duality 
\begin{align*}\Ext^2_X(\Mscr_j', \Mscr_i')  = \Hom_X(\Mscr_i', \Mscr_j' \otimes \omega_X)^*
\end{align*}

But $\mu(\Mscr_j' \otimes \omega_X) = \mu(\Mscr_j') - K^2_X < \mu(\Mscr_i')$,  since $\mu(\Mscr_i'), \mu(\Mscr_j') \in [0,K^2_X)$.  Thus, $\Hom_X(\Mscr_i', \Mscr_j' \otimes \omega_X)$ vanishes and therefore so does $\Ext^2_X(\Mscr_j', \Mscr_i')$.  

The same argument as above shows that $\Ext^2_X(\Mscr', \Mscr'\otimes \omega_X^{-k}) = 0$ for all $k \geq 0$.   Hence,  combining this with Proposition \ref{pr:strigid} we have 
\begin{equation}\label{eq:tilting}\Ext^i_X(\Mscr', \Mscr'\otimes \omega_X^{-k}) = 0 \text{ for all }i > 0, k \geq 0.
\end{equation}

Now let $Y := \Tot(\omega_X)$ and denote by $\pi: Y \to X$ the canonical projection.  Then $\End_Y(\pi^*(\Mscr'))$ is a regrading of $\End_R(M) = \Lambda$ (since $\End_Y(\pi^*(\Mscr')) \to \End_R(M)$ is an isomorphism in codimension $1$ and both sides are reflexive) and $\pi^*(\Mscr')$ is a tilting object in $\Dscr^b(Y)$.  Indeed,  by (\ref{eq:tilting}) we have
\begin{align*} \Ext^i_Y(\pi^*(\Mscr'), \pi^*(\Mscr')) =  \Ext^i_X(\Mscr', \pi_*\pi^*(\Mscr')) \\ =  \bigoplus_{k \in \ZZ} \Ext^i_X(\Mscr', \Mscr' \otimes \omega_X^{-k}) = 0,
\end{align*}
 for all $i > 0$. 
 
Next we will show that $\pi^*(\Mscr')$ generates $\Dscr^b(Y)$.  To this aim we adopt an argument in \cite{SVdB}.  Denote by $\Dscr \cong \Dscr^b(\End_Y(\pi^*(\Mscr')))$ the thick subcategory of $\Dscr^b(Y)$ generated by $\pi^*(\Mscr')$.  Since $\End_Y(\pi^*(\Mscr'))$ has finite global dimension,  Lemma 1.1.1 in \cite{PVdB} implies that the subcategory $\Dscr$ is admissible in $\Dscr^b(Y)$.  But by Lemma \ref{lm:noSOD} $\Dscr^b(Y)$ does not admit any non-trivial $R$-linear semiorthogonal decompositions,  so we conclude that $\Dscr = \Dscr^b(Y)$.  Observe that this is equivalent to  $\pi^*(\Mscr')$ generating $\Dscr_{\QCoh}(Y)$ in the sense that the right orthogonal to $\pi^*(\Mscr')$  in $\Dscr_{\QCoh}(Y)$ is zero (due to a theorem by Ravenel and Neeman,  see \cite[Theorem 2.1.2]{BVdB}). 
 
Finally,  we show that the exceptional collection $(\Mscr_0', \dots, \Mscr_r')$ is full.  Suppose to the contrary that there exists an object $\Fscr \in \langle \Mscr_0', \dots, \Mscr_{n-1}' \rangle^\perp$.  For any $\Fscr_1,\Fscr_2 \in \Dscr^b(X)$ we have 
\[ \Hom_{\Dscr^b(X)}(\Fscr_1,\Fscr_2) \cong \Hom_{\Dscr^b(Y)}(\pi^*\Fscr_1,i_*\Fscr_2),
\]
where $i: X \to Y$ is the embedding of the zero section and we have used the adjunction $(\pi^*,\pi_*)$ together with the fact that $\pi\circ i = id_Y$.  Hence $i_* \Fscr \in \langle \pi^* \Mscr_1', \dots,  \pi^*\Mscr_{n-1}' \rangle^\perp = 0$ and therefore $\Fscr = 0$.\footnote{We thank Alexander Kuznetsov for drawing our attention to this simple and clear argument.  Alternatively,  one can argue as follows.  Note that it suffices to show that $n = \rk K_0 (X)$,  because on a del Pezzo surface every exceptional collection of maximal length is full \cite{KuleshovOrlov}.  Since it is already established that $\pi^*(\Mscr')$ is a tilting complex,  we have $\Dscr^b(\Lambda) \cong \Dscr^b(Y)$,  hence $\rk K_0(\Lambda) = \rk K_0( Y) = \rk K_0 (X)$.  On the other hand,  $K_0(\Lambda) = n$,  because both sides equal the number of non-isomorphic graded simple $\Lambda$-modules.}
\end{proof}
To prove Proposition \ref{prop:completion} we need the following fact relating rigidity to grading:
\begin{proposition}[Keller-Murfet-Van den Bergh, {\cite[Proposition 6.1]{KellerMurfetVdB}}]\label{pr:rigidgraded}
  Let $R=k+R_1+\dots$ be a Noetherian graded $k$-algebra with $char(k) = 0$ and $M$ a finitely generated $\widehat{R}$-module with $\Ext^1_{R}(M,M) = 0$.
  Then there exists a finitely generated graded $R-$module $N$ such that $M$ is the completion of $N$. 
\end{proposition} 
\begin{proof}[Proof of Proposition \ref{prop:completion}]
  Let  $ \Lambda = \End_{\widehat{R}}(M)$ be an NCCR of $\widehat{R}$. Then by Theorem \ref{thm:IR} $M$ is rigid.
  Let $N\in \gr(R)$ be the corresponding graded module
given by Proposition \ref{pr:rigidgraded}. Since completion preserves depth with respect to the irrelevant ideal in $R$ we obtain that $\Gamma=\End_{R}(N)$
is a graded NCCR of $R$. Clearly $\Lambda$ is the completion of $\Gamma$.
\end{proof}
\section{Mutations of NCCRs}
\label{sec:mutnccr}
\begin{definition}\label{def:NCCRmut} Let $R$ be a normal Gorenstein domain and $M = N \oplus L$ an $R$-module.  Let $f:N' \to M$ be a right $\add (N)$-approximation of $M$,  i.e. $N' \in \add (N)$ and the induced map 
\[
\Hom_R(N,N') \xrightarrow{} \Hom_R(N,M)
\]
 is surjective.  Then we define \emph{the right mutation} of $M$ with respect to $N$ as 
\[
\mu^+_N(M) := N\oplus \ker(f).
\]
\emph{Left mutations} are defined dually by 
\[
\mu^-_N(M) := N\oplus \ker(g)^*,
\]
where $g: N'' \to M^*$ is a right $\add (N^*)$-approximation of $M^*$. 
  \end{definition} 
\begin{remark} It is clear that mutations can be carried out in the graded context (see \S\ref{sec:NCCRs}).
  \end{remark}
  
 \begin{remark} In general,  one should keep in mind that mutations are defined only up to additive closure.  For instance, left and right mutations $\mu^{\pm}_N$ are mutually inverse only up to taking $\add(-)$.  However,  if we are in the complete or graded context,  mutations are unique up to isomorphism if we can consider approximations $f: N' \to M$ for which no non-trivial direct summands of $N'$ are contained in $\ker(f)$.  \end{remark} 
  
\begin{remark} If we are in the complete or graded context and if we fix a decomposition of $M = \oplus_{i=1}^n M_i$ into indecomposable direct summands,  we can write $\mu^{\pm 1}_j(M)$ for $\mu^{\pm 1}_{\oplus_{i \neq j} M_i}(M) = \bigoplus_{i \neq j} M_i \oplus M_j'$.  Note that this again gives a fixed enumeration of indecomposable direct summands of $\mu^{\pm 1}_j(M)$,  hence we can consider iterated mutations and talk of the ``$j$'th mutation'' $\mu_j^{\pm 1}$ at every step.   \end{remark} 
  \begin{theorem}[{\cite[Theorems 1.22, 1.23, 1.25]{IyamaWemyss14_2}}]\label{thm:IWmutations} Let $R, M$ and $N$ be as in Definition \ref{def:NCCRmut}.  Assume that $\Lambda := \End_R(M)$ is an NCCR of $R$. Then 
  \begin{enumerate}
  \item  $\End_R(\mu_N^+(M))$ and $\End_R(\mu_N^-(M))$ are also NCCRs of $R$. 
  \item The algebras $\Lambda$,  $\End_R(\mu_N^+(M))$ and $\End_R(\mu_N^-(M))$ are derived equivalent. 
  \item Assume in addition that we are in the complete or graded context and $\dim R = 3$.  Let $\Lambda = \End_R(M)$ be an NCCR of $R$, where $M = \bigoplus_{i=1}^n M_i$ with $M_i$ indecomposable and pairwise non-isomorphic (i.e.  $\Lambda$ is basic according to Definition \ref{def:basic}).  Then $\mu_i^+(M) \cong \mu_i^-(M)$ (up to a shift in grading of the components in the graded context).  In this case we can write $\mu_i: = \mu_i^+ = \mu_i^-$.  
  \end{enumerate}
  \end{theorem}
  
 In the case we will focus on, there is an alternative way to view mutations of NCCRs, namely as mutations of \emph{quivers with potentials} in the sense of Derksen, Weyman and Zelevinsky \cite{DWZ}.

\begin{definition}\label{def:DWZmutation}
Let $(Q,W)$ be a quiver with potential 
 and let $i\in Q_0$ be a vertex which is not contained in a loop or oriented 2-cycle.  
Define $\tilde{\mu}_i(Q,W)=(\widetilde{Q},\widetilde{W})$ to be obtained from $(Q,W)$ as follows:
\begin{enumerate}
\item For every pair of arrows $\beta,\alpha$ with $t(\beta)=i=s(\alpha)$, add a new arrow
\[
[\alpha\beta]: s(\beta)\to t(\alpha).
\]
\item Replace each arrow $\gamma$ with $s(\gamma)=i$ or $t(\gamma)=i$ by an arrow $\gamma^*$ in the opposite direction.
\item Define $\widetilde{W}$ to be the sum of
\begin{itemize}
\item the potential obtained from $W$ by replacing every subpath $\alpha\beta$ with $t(\beta)=i=s(\alpha)$ by the corresponding arrow $[\alpha\beta]$, and
\item the additional term $\sum_{\alpha,\beta} [\alpha\beta]\beta^*\alpha^*$, where the sum ranges over all pairs $(\beta,\alpha)$ with $t(\beta)=i=s(\alpha)$.
\end{itemize}
\end{enumerate}
\begin{remark} In the graded context we need to assign a degree to $\alpha^\ast$, $\beta^\ast$ such that $\deg \alpha +\deg \beta+\deg \alpha^\ast +\deg \beta^\ast=\deg W$.
The precise rule is given in \cite[\S2]{dTdVVdB}.
\end{remark}
If $W$ is unspecified then $\widetilde{Q}$ is still well-defined and we write
$
\tilde{\mu}_iQ=\widetilde{Q}
$.
Let $(Q,W)$ be a quiver with potential such that $Q$ has no loops or oriented $2$-cycles and assume that we are in the complete or graded context.
We let
\[
\mu_i(Q,W): =(\widetilde{Q},\widetilde{W})_{\redd}:=(\widetilde{Q}_{red},\widetilde{W}_{\redd}),
\]
to be the \emph{reduced part} of $(\widetilde{Q},\widetilde{W})$ (see \cite[Theorem 4.6]{DWZ}; in particular, $\widetilde{W}_{\redd}$ contains no 2-cycles). We write 
\[
\mu_i\Jscr(Q,W):=\Jscr(\mu_i(Q,W)).
\]
We will also define $\mu_iQ$ as the quiver obtained from $\widetilde{\mu}_i{Q}$ by deleting
all $2$-cycles.
\end{definition}

\begin{theorem}[Keller-Yang, {\cite[Theorem 3.2]{KellerYang}}]\label{thm:KYIW=DWZ}
Let $\Lambda$ be a 3-Calabi-Yau algebra with presentation $\Lambda\cong \Jscr(Q,W)$ with $Q$ having no loops or $2$-cycles.  
For $j\in Q_0$ denote by $P_j = \Lambda e_j$ the corresponding indecomposable projective module.  Then there is an algebra isomorphism
\[
\End(\mu_{\oplus_{j\neq i}{P_j}}(\Lambda))^\circ \cong \mu_i\Jscr(Q,W)=\Jscr(\mu_i(Q,W)).
\]
\end{theorem}
In particular, in this setting the mutation of NCCRs defined in Definition \ref{def:NCCRmut} agrees with DWZ mutations of quivers with potential.

Since the quiver $Q$ for $B(\EE)$ contains no loops or oriented two-cycles (see Remark \ref{rem:no_loops_or_two_cycles}) Theorem \ref{thm:KYIW=DWZ} allows us to describe mutations of the NCCR $B(\EE)$ by the combinatorial DWZ mutation procedure in Definition \ref{def:DWZmutation}.  For this reason in our situation we will often refer to mutations of NCCRs as \emph{quiver mutations}. 
This will also help us distinguish these mutations from mutations of exceptional collections on $X$ inducing the braid group action.  

Note that it is not guaranteed in general that the quiver of $\mu_i\Jscr(Q,W)$ again contains no loops or two-cycles (there is only such a guarantee for $W$).  However,  Bridgeland and Stern \cite[Theorem 1.7]{BridgelandStern} showed that this is indeed the case for mutations of $3$-Calabi--Yau algebras obtained from very strong exceptional collections on del Pezzo surfaces:
  
\begin{theorem}[Bridgeland-Stern, {\cite[Theorem 1.7]{BridgelandStern}}, see also \cite{dTdVVdB} for another proof]\label{thm:BS} Let $B(\EE) = \Jscr(Q,W)$ be as above with $\EE$ a very strong exceptional collection on a del Pezzo surface $X$.  Let $i \in Q_0$  be a vertex and $\mu_i B(\EE) = \Jscr(Q',W')$ the algebra obtained by mutation at $i$.  Then there exists a very strong exceptional collection $\EE'$ on $X$ such that $\mu_i B(\EE) = B(\EE')$.  In particular, $Q'$ does not contain loops or two-cycles. 
\end{theorem} 
This result may be obtained from Theorem \ref{thm:nccrfromFEC}: 
\begin{proof}[Proof of Theorem \ref{thm:BS}]
  By Theorem \ref{thm:IWmutations} $\mu_i B(\EE)$ is an NCCR of $R_X$.  Hence by Theorem \ref{thm:nccrfromFEC} $\mu_i B(\EE) \cong B(\EE')$ for some very strong exceptional collection $\EE'$ on $X$. 
\end{proof}

It is useful to understand explicitly how $\EE$ and $\EE'$ are related in Theorem \ref{thm:BS}.  First we introduce some notations.  Rotating the helix if necessary, we can always assume that $i = n-1$.  The following is proved in \cite{BridgelandStern}:

\begin{proposition}\label{prop:clusterVSregularmut} 
Let $\EE = (E_0, \dots, E_{n-1})$ be a very strong exceptional collection  on del Pezzo surface. Then there exists $j \in \{1, \dots, n-2\}$ such that 
\[
\EE'=  (E_{1},E_{2},\ldots, E_{j} ,\widetilde{R_{E_j}}\cdots \widetilde{R_{E_{1}}} E_0,E_{j+1},\ldots,E_{n-1}).
\]
is a very strong exceptional collection with
 $\mu_{0} B(\EE) \cong  B(\EE')$. 
\end{proposition}
To be consistent with the terminology introduced above, we will call $\EE'$ the quiver mutation of $\EE$ at $E_0$.
Quiver mutations at other $E_i$ are defined by rotation.
The following proposition strengthens Proposition \ref{prop:clusterVSregularmut} 
\begin{proposition}
  \label{prop:rem:cluster}
We have 
\[
\EE'=  (E_{1},E_{2},\ldots, E_{j} ,\widetilde{R_{E_j}}\cdots \widetilde{R_{E_{1}}} E_0,E_{j+1},\ldots,E_{n-1}).
\]
where $j$ is the smallest index such that $\EE'$
is very strong and $\Hom_X(E_0,E_j)
 \neq 0$.  
\end{proposition}
\begin{proof}
  This is a rephrasing of Proposition \ref{prop:lm:clustergeneralised} below.
  \end{proof}
In other words,  a quiver mutation amounts to applying several consequent braid group mutations to a very strong exceptional collection until it becomes very strong again for the first time.  The only exception to this rule is when the first mutation is a mutation in an orthogonal pair.  Such a mutation trivially keeps the collection very strong and does not affect the corresponding NCCR.  In this case one has to mutate ``through'' the whole orthogonal block.
In what follows,  we usually identify a quiver mutation (a mutation of an NCCR) and its corresponding sequence of mutations on the level of exceptional collections,  as described in Proposition \ref{prop:clusterVSregularmut}.
\begin{remark}
  \label{rem:left:quiver}
  If $\EE'$ is the quiver mutation of $\EE$ at $E_i$ 
  then using the Convention \ref{conv:tilde} we may write $\EE'$ as
  $
    \tilde{\sigma}_{j}^{-1}\cdots \tilde{\sigma}_{i+2}^{-1}\tilde{\sigma}_{i+1}^{-1}\EE
    $ for suitable $j$ where the indices are wrapped to remain in the interval $[1,n]$.
    Thus it makes sense to refer to $\EE'$ as the \emph{right} a quiver mutation of $\EE$ and $E_i$. The left quiver mutation
    at $E_i$ can be defined symmetrically and it will be equal to $\tilde{\sigma}_{j'+1}\cdots\tilde{\sigma}_{i-1}\tilde{\sigma}_i \EE$ for suitable $j'$.
    Using the analogue of Proposition \ref{prop:rem:cluster} one sees that left and right quiver mutations at $E_i$ are each other's inverse up to rotation
    and permuting the exceptional objects in a block.
\end{remark}
\section{Blocks and block mutations}
\subsection{Blocks and block quivers}
\label{sec:blocks_and_block_quivers}
\begin{definition}
If $Q$ is a quiver then a \emph{block} is  a set of vertices which are
not connected by edges. We say that a block is not incident to loops or oriented 2-cycles
if none of its members are.
\end{definition}
\begin{definition} A \emph{block quiver} $(Q,I)$ is a quiver $Q$ together with a decomposition
$Q_0=I_0\coprod \cdots \coprod I_{k-1}$ where the $I_i$ are blocks.
\end{definition}
In the sequel we will encounter a particular kind of block quiver. 
\begin{definition}
\label{def:vertex_mult}
A quiver with \emph{vertex multiplicities} is $(Q,(\alpha_i)_{i\in Q_0})$ is a pair
consisting of a quiver $Q$ and for each vertex $i\in Q_0$ a ``multiplicity'' $\alpha_i\ge 1$. 
\end{definition}
If $(Q,(\alpha_i)_{i\in Q_0})$ is a quiver with vertex multiplicities, without loops, then $Q^e$ is the block quiver obtained by replacing  each vertex~$i$ in $Q$ by a block 
$I_i$ of
size $\alpha_i$ and by letting for every $u\in I_i$, $v\in I_j$, $i\neq j$, $Q^e(u,v)$ be a distinct copy of $Q(i,j)$. If a block quiver $\Sigma$ is of the form $Q^e$ then we call $\Sigma$ \emph{regular} and we call $Q$ the \emph{reduced quiver} associated
to $\Sigma$. We denote it by $\Sigma_{\red}$.
\subsection{Block quiver mutations}
\label{sec:block_quiver_mut}
Let $(Q,W)$ be a quiver with potential.
Let
$I=\{i_0,\ldots,i_{\alpha-1}\}\subset Q_0$ be  a block not incident to loops or oriented 2-cycles. Then it is clear that
$\widetilde{\mu}_I (Q,W):=\widetilde{\mu}_{i_0}\cdots \widetilde{\mu}_{i_{\alpha-1}} (Q,W)$  (see Definition \ref{def:DWZmutation})
is well defined. In other words it does not depend on the chosen ordering on $I$.
If $Q$ has no loops or oriented $2$-cycles then we also put $\mu_I (Q,W)=\widetilde{\mu}_I(Q,W)_{\redd}$
and we call $\mu_I(Q,W)$ a \emph{block mutation} of $(Q,W)$.
One may show that $\mu_I$ is the composition of $(\mu_{i_j})_j$. Following the conventions
in Definition \ref{def:DWZmutation} we also use the notations $\widetilde{\mu}_I Q$ and $\mu_I Q$ when appropriate.

In case $(Q,I)$ is a block quiver,  we write $\mu_i$ for $\mu_{I_i}$.
\subsection{Block mutations of NCCRs}
Assume that we are in the graded or complete context.
Let $\Lambda$ be an NCCR of $R$ which is represented as $\Lambda$. Let $(S_i)_{i=0}^{n-1}$ be the non-isomorphic simples of $\Lambda$ (concentrated in degree zero in the graded
context), and let $Q$ be the associated
$\Ext^1$-graph as 
in Remark \ref{rem:Ext1graph}.

A \emph{block} is a collection
of simples $(S_i)_{i\in I}$ for $I=\{i_0,\ldots,i_{\alpha-1}\}\subset \{0,\ldots, n-1\}$ such that $\RHom_\Lambda(S_{i_k},S_{i_l})=0$ for $k\neq l$.
If this is the case and $I$ is not adjacent to any loops and oriented $2$-cycles in $Q$
then we claim that $\mu_I(\Lambda):=(\mu_{i_0}\cdots \mu_{i_{\alpha-1}})(\Lambda)$ is again well defined. This follows by representing $\Lambda$ as $\Jscr(Q,W)$ (see \S\ref{sec:mutnccr}) and invoking \S\ref{sec:block_quiver_mut}.
\begin{remark} Presumably this result will also hold without the restriction that $I$ is not adjacent to any loops or oriented $2$-cycles.
\end{remark}

\subsection{Block exceptional collections}
\label{sec:block_ex}
\begin{definition} A \emph{block} is a collection of exceptional objects $(E^{(0)},\ldots,E^{(\alpha-1)})$  which are pairwise orthogonal.
\end{definition}
\begin{definition}
A \emph{block exceptional collection} is a sequence of blocks
\begin{equation}
\label{eq:block_ex}
\EE=((E_0^{(0)},\ldots, E_0^{(\alpha_0-1)}) , 
\ldots,
(E_{k-1}^{(0)},\ldots, E_{k-1}^{(\alpha_{k-1}-1)} ))
\end{equation}
which becomes an exceptional collection when we forget the subdivision into blocks. 
If $k=2$ then we call $(\EE_0,\EE_1) = ((E_0^{(0)},\ldots, E_0^{(\alpha_0-1)}),  (E_{1}^{(0)},\ldots, E_{1}^{(\alpha_{1}-1)}))$ a \emph{block exceptional pair}. We say that a block exceptional collection $\EE$ is \emph{full} if the associated exceptional collection is full.
\end{definition}
\begin{definition} If $(\EE,\FF)=((E^{(0)},\ldots,E^{(\alpha-1)}),(F^{(0)},\ldots,F^{(\beta-1)}))$ is a block exceptional pair then we define
  \begin{equation}
    \label{eq:blockLR}
    \begin{aligned}
      L_\EE \FF&=(L_{E^{(0)}}L_{E^{(1)}}\cdots L_{E^{(\alpha-1)}}F^{(0)},\ldots,L_{E^{(0)}}L_{E^{(1)}}\cdots L_{E^{(\alpha-1)}} F^{(\beta-1)})\\
      R_\FF \EE&=(R_{F^{(0)}}R_{F^{(1)}}\cdots R_{F^{(\beta-1)}}E^{(0)},\ldots,R_{F^{(0)}}R_{F^{(1)}}\cdots R_{F^{(\beta-1)}} E^{(\alpha-1)})
    \end{aligned}
  \end{equation}
\end{definition}
The following result is easy to verify.
\begin{lemma} \label{lem:dual_block_exceptional}
  Let $\EE=(\EE_0,\cdots,\EE_{k-1})$ be a full block exceptional collection. Define the \emph{right dual} block exceptional collection 
  $\widecheck{\EE}  
=(\widecheck{\EE_{k-1}},\ldots, \widecheck{\EE_0})$ via
\[
\widecheck{\EE_i} := R_{\EE_{k-1}} \dots R_{\EE_{i+1}}(\EE_i)
\]
Then the underlying exceptional collections of $\widecheck{\EE}$, with the order of the exceptional objects in the blocks inverted, is the right dual (see Proposition \ref{def:dualcollection}) of the underlying exceptional collection of $\EE$. 
\end{lemma} 

\begin{definition}
\label{def:block_form}
We say that a  block exceptional collection has \emph{broken blocks} if $\EE_i$, $\EE_{i+1}$ are orthogonal for some $i$ with $\EE_{k}:=\omega_X^{-1}\otimes_X \EE_0$.
\end{definition}
\begin{definition}
Any full exceptional collection $\EE$ may be transformed using a suitable rotation into a block exceptional collection $\EE'$ without broken blocks. If $\EE'$ has block sizes $(\alpha_i)_{i=0}^{k-1}$ then we say that $\EE$ is $k$-block and
we call the $(\alpha_i)_i$ the \emph{block multiplicities} of $\EE$. The block multiplicities are well defined up to cyclic permutation.
\end{definition}
\begin{remark} The blocks related properties of an exceptional collection can be characterised solely in terms of slopes. See Proposition \ref{pr:slopes}.
\end{remark}
\begin{convention} \label{conv:terminology}
  Below terminology and notation, when applied to a block exceptional collection, are understood to apply
to the associated exceptional collection, unless defined otherwise.
\end{convention}
For a block exceptional pair $(\EE,\FF)$ we define $\widetilde{L_\EE}\FF$, $\widetilde{R_\FF \EE}$ by replacing  $L$ and $R$ in \eqref{eq:blockLR} by $\widetilde{L}$ and $\widetilde{R}$.
Proposition \ref{prop:clusterVSregularmut}  generalises in the following way:
\begin{proposition}\label{prop:clusterVSregularmut_block} 
Let $\EE = (\EE_0, \dots,\EE_{k-1})$ be a very strong block exceptional collection  on a del Pezzo surface. Then there exists $j \in {1, \dots, k-2}$ such that
\[
\EE' := (\EE_{1},\EE_{2},\ldots, \EE_{j} ,\widetilde{R_{\EE_j}}\cdots \widetilde{R_{\EE_{1}}} \EE_0,\EE_{j+1},\ldots,E_{k-1}).
\] is a very strong block exceptional collection with
 $\mu_{0} B(\EE) \cong  B(\EE')$. 
\end{proposition} 
We will call $\EE'$ a \emph{block quiver mutation} of $\EE$. The obvious analogue or Proposition \ref{prop:rem:cluster} holds.
\begin{remark} \label{rem:block_quiver}
  If $\EE_i=(E_i^{(0)},\ldots,E_i^{(\alpha_i-1)})$ then it is easy to see that a block quiver mutation as in Proposition \ref{prop:clusterVSregularmut_block}
is  the composition of quiver mutations at $E_i^{(\alpha_i-1)},\allowbreak E_i^{(\alpha_i-2)},\allowbreak\ldots,\allowbreak E_i^{(0)}$.
\end{remark}
 
\subsection{Block quivers associated to block exceptional collections}
\label{sec:blockquiver}
If $\EE$ is a block exceptional collection
\[
\EE=((E_0^{(0)},\ldots, E_0^{(\alpha_0-1)}) , 
\ldots,
(E_{k-1}^{(0)},\ldots, E_{k-1}^{(\alpha_{k-1}-1)} ))
\]
on a del Pezzo surface $X$, then the reduced Gram matrix $M_{\red}$ is the $k\times k$-matrix such that $M_{\red,ij}=\chi(E_i^{(?)}, E_j^{(?)})$.

If $\EE$ is in addition very strong
then
the quiver  associated to $\EE$, as constructed in Proposition \ref{prop:quiver_rule} becomes
naturally a block quiver $Q$ with blocks $(\{(i,0), \ldots, (i,\alpha_i-1)\})_{i=0}^{k-1}$. This block quiver is actually regular so it may equivalently be viewed as a quiver with  vertex multiplicities $(Q_{\red},(\alpha_i)_{i=0}^{k-1})$. 

\section{Toric machinery of exceptional collections}\label{sec:toric}
\subsection{Introduction}
This section is based on the framework developed by Perling in \cite{Perling} (see also \cite{KuznetsovToric},  \cite{HillePerling}) which allows to associate a toric system in $\Pic(X) \otimes \QQ$ and a fan in a two-dimensional lattice to an exceptional collection on a rational surface $X$.  Moreover,  very strong exceptional collections on del Pezzo surfaces correspond to fans generating convex polygons.  Via this correspondence mutations of exceptional collections as well as quiver mutations can be seen as operations on toric systems or fans.  For simplicity we will formulate everything for del Pezzo surfaces using the notations introduced in \S \ref{sec:delPezzoprelim}.
\subsection{Toric systems and their Gale duals}
\begin{proposition}[{\cite[\S 5]{Perling}, \cite[Proposition 4.9]{KuznetsovToric}}]\label{prop:Ts} 
 Let $\EE=(E_0,\dots,E_{n-1})$ be a full exceptional collection of vector bundles on a del Pezzo surface $X$ and denote $r_i := r(E_i)$. 
 
 Define $T_{i,i+1} \in \Pic(X) \otimes \QQ$ with $0 \leq i \leq n-2$ by
 \begin{equation}
   \label{eq:real_toric}
\begin{aligned}
T_{i,{i+1}}& := s(E_{i},E_{i+1}) = \frac{c_1(E_{i+1})}{r_{i+1}} -  \frac{c_1(E_i)}{r_i} ,  \,\,  i=0, \dots,n-2  \\
T_{n-1,n} &:= T_{-1,0} := T_{n-1,0} := s(E_{n-1},E_{0} \otimes \omega_X^{-1})
\end{aligned}
\end{equation}
Denote also $T_{ij} = T_{i,{i+1}} + \dots + \dots T_{{j-1},j}$ for $i < j$  Then $(T_{i,{i+1}})_{i=0}^{n-1}$ form a \emph{toric system},  which means that 
\begin{enumerate}
\item \label{it:prop:Ts_1} $T_{i-1,i} \cdot T_{i,i+1} = \frac{1}{r_{i}^2}$ for all $0 \leq i \leq n-1$
\item \label{it:prop:Ts_2} $T_{i-1,i}\cdot T_{j,j+1} = 0$ for $0 \leq i < j \leq n-1$.
\item \label{it:prop:Ts_3} $\sum_{i=0}^{n-1} T_{i,i+1} = -K_X$
\item \label{it:prop:Ts_4} $r_{i}r_{j}T_{ij}$ is integral for all $0 \leq i < j \leq n-1$ (hence ${a_{ij} := (r_{i}r_{j}T_{ij})^2 \in \ZZ}$). 
\end{enumerate}
Moreover,  
\begin{equation}
  \label{eq:chi_T}
\chi(E_{i},E_{j}) = \frac{a_{ij} + r_{i}^2 +r_{j}^2 }{r_{i}r_{j}}.
\end{equation}
\end{proposition}
\begin{definition}
  \label{def:gale_dual}
  Let $V$ be a vector space over a field $k$ and let $(v_i)_{i=1,\ldots,n}\in V$.  Assume that $(v_i)_i$ spans $V$. 
The \emph{Gale dual} of $(V,(v_i)_i)$ is a pair $(W,(w_i)_{i=1,\ldots,n})$ where $W$ is a $k$-vector space and $(w_i)_i \in W$ such that
\begin{enumerate}
\item  $\dim W=n-\dim V$;
\item  $(w_i)_i$ spans $W$;
\item \label{it:def:gale_dual_3} $\sum_{i=1}^n w_i \otimes v_i=0 \text{ in } W \otimes V$.
\end{enumerate}
\end{definition} 
\begin{remark}
One verifies that the Gale dual is unique up to unique isomorphism.
\end{remark}
It is easy to see that in the toric system   corresponding to $\EE$,  as defined in Proposition \ref{prop:Ts}, the vectors $(T_{i,i+1})_{i=0}^{n-1}$ generate
$\Pic(X)_\RR$ as $\RR$-vector space. Thus the Gale dual of $(\Pic(X)_\RR, (T_{i,i+1})_{i=0}^{n-1})$ is well defined (see \cite{KuznetsovToric,Perling}) and we write it as $(L,(l_{i,i+1})_{i=0}^{n-1})$
where indices are taken modulo $n$, i.e. $l_{n-1,0}:=l_{n-1,n}$. 
In Theorem \ref{th:gale} below we give an explicit description
of said Gale dual. 
\begin{remark} \label{rem:gale_setup}
Note that since $n-\dim \Pic(X)_\RR=\rk K_0(X)-(\rk K_0(X)-2)=2$ the Gale dual lives in a two dimensional space.
\end{remark}
\subsection{Gale duals of toric systems and dual exceptional collections}
\label{sec:dual_ex}
In this section we show how the Gale dual $(L,(l_{i,i+1})_i)$ of $(T_{i,i+1})_{i=0}^{n-1}$ can be directly obtained from the dual exceptional collection of $\EE$. This result appears to be new.
We use this to establish a direct connection between  $(L,(l_{i,i+1})_i)$ and the quiver associated to $B(\EE)$ (see \S\ref{sec:jacobi}). This appears to be new as well.
\begin{theorem}\label{th:gale} 
  Let $\EE = (E_0,\dots, E_{n-1})$ be a full exceptional collection of vector bundles on a del Pezzo surface $X$.  Let $(F_{n-1}, \dots, F_0)$ be the right dual collection of $\EE$
  (see Proposition \ref{def:dualcollection}).  For $Y\subset X$ a smooth anti-canonical divisor let $L := K^{\num}_0(Y)_{\RR}$,  $y \in Y$ and 
\begin{equation}
\label{eq:formula}
l_{i,i+1}=-[\Oscr_{y}]+\sum_{j=0}^i r_j[F_{j,Y}] \text{ for all } i \in \{0, \dotsm n-1\}. 
\end{equation}
 The Gale dual of the  toric system $(T_{i,i+1})_{i=0}^{n-1}$ corresponding to $\EE$,  as defined in Proposition \ref{prop:Ts},  can be taken to be $(L,(l_{i,i+1})_{i=0}^{n-1})$. 
 Moreover
\begin{equation}
\label{eq:cycle}
\sum_{j=0}^{n-1}r_j[F_{j,Y}]=0
\end{equation}
so that if we extend the formula \eqref{eq:formula} to all $i\in\ZZ$, we have $l_{i+n,i+n+1}=l_{i,i+1}$.
\end{theorem}
\begin{proof}  First we verify \eqref{eq:cycle}. 
Let $\alpha:Y\hookrightarrow X$ be the inclusion. We claim that 
\begin{equation}
\label{eq:cycle1}
\sum_{j=0}^{n-1} r_j [F_j]=[\alpha_\ast\Oscr_y] \text{ in } K_0(X). 
\end{equation}
Indeed,  it is sufficient to check that \eqref{eq:cycle1} holds after applying $\chi_X(-,[E_i])$ for every $i =0, \dots, n-1$,  but this gives just $r_i = r(E_i)$. 
Since $[\alpha^\ast\alpha_\ast\Oscr_y]=0$, applying $\alpha^\ast$ to \eqref{eq:cycle1} we get  \eqref{eq:cycle}. 

Now we prove the first part.  Note that $\dim K^{\num}_0(Y)_{\RR}=2=n-\dim \Pic(X)_\RR$ (see Remark \ref{rem:gale_setup}) and moreover it is easy to see that ${\alpha^\ast: K_0(X)_\RR \to K^{\num}_0(Y)_{\RR}}$ is surjective.  
So the $[F_{j,Y}]$ span $K_0^{\num}(Y)_{\RR}$.  Therefore $(l_{i,i+1})_i$ (as defined by \eqref{eq:formula}) also spans $K^{\num}_0(Y)_{\RR}$ and hence
it remains to prove that
\[
\sum_{i=0}^{n-1} l_{i,i+1}\otimes T_{i,i+1} = 0 \quad \text{ in } K^{\num}_0(Y)_{\RR}\otimes \Pic(X)_\RR.
\]
or concretely
\begin{equation}
\label{eq:gale_dual}
\sum_{i=0}^{n-1} l_{i,i+1} \otimes (c_1(E_{i+1})/r_{i+1}-c_1(E_i)/r_i)=0,
\end{equation}
where $E_{n}:= E_0 \otimes \omega_X^{-1}$. Let us write $E'_i=E_{i\mod n}$. 
Then \eqref{eq:gale_dual} can be
rewritten as 
\[
\sum_{i=0}^{n-1} l_{i,i+1} \otimes \bigg(\frac{c_1(E_{i+1}')}{r_{i+1}}-\frac{c_1(E_i')}{r_i} \bigg) = -l_{n-1,0}\otimes c_1(\omega_X^{-1})
\]
or equivalently 
\[
\sum_{i=0}^{n-1} \frac{(l_{i-1,i}-l_{i,i+1})}{r_i} \otimes c_1(E_{i})=-l_{n-1,0}\otimes c_1(\omega_X^{-1}).
\]
So we have to verify
\begin{equation}
\label{eq:rewritten}
\sum_{i=0}^{n-1} [F_{i,Y}] \otimes c_1(E_{i})=-[\Oscr_y]\otimes c_1(\omega_X^{-1}) \quad \text{ in } K^{\num}_0(Y)_{\RR}\otimes \Pic(X)_\RR.
\end{equation}
Consider the standard filtration on $K_0(X)_\RR$ by codimension of supports $F^\ast(X):=F^\ast K_0(X)_\RR$.
 We have 
\begin{align*}
F^0(X)/F^1(X)&=\RR[\Oscr_X]\\
F^2(X)&=\RR[\alpha_\ast\Oscr_y]
\end{align*}
for an arbitrary $y\in Y$.  Moreover there is a natural
identification
\[
F^1 K_0(X)_\RR/F^2 K_0(X)_\RR\cong \Pic(X)_\RR
\]
such that for $E\in K_0(X)_\RR$ we have
\begin{align*}
c_1(E)\cong [E]-\rk(E)[\Oscr_X] \qquad \mod F^2(X). 
\end{align*}
In particular
\begin{equation} \label{eq:c1}
c_1(\omega_X^{-1})\cong [\Oscr_X(Y)]-[\Oscr_X] \cong  [\alpha_\ast \Oscr_Y(Y)]\cong [\alpha_\ast \Oscr_Y]\qquad \mod F^2(X).
\end{equation}
By \eqref{eq:cycle} we have
\[
\sum_{i=0}^{n-1} [F_{i,Y}] \otimes r_i[\Oscr_X]=0. 
\]
Hence by \eqref{eq:c1} we find that \eqref{eq:rewritten} is equivalent to
\[
\sum_{i=0}^{n-1} [F_{i,Y}] \otimes [E_{i}]=-[\Oscr_y]\otimes [\alpha_\ast\Oscr_Y] \quad \mod K^{\num}_0(Y)_{\RR} \otimes \RR[\alpha_\ast\Oscr_y]
\]
in $K^{\num}_0(Y)_{\RR} \otimes K_0(X)_{\RR}$. In fact we have 
\[
\sum_{i=0}^{n-1} [F_{i,Y}] \otimes [E_{i}]=-[\Oscr_y]\otimes [\alpha_\ast\Oscr_Y] + [\Oscr_Y] \otimes [\alpha_\ast\Oscr_y].
\]
Indeed, after applying $\Id\otimes \chi_X([F_i], -)$ this is equivalent to
\[
[F_{i,Y}]=\chi(F_{i,Y})[\Oscr_y]+\rk(F_{i,Y})[\Oscr_Y]
\]
which is clear.
\end{proof}

\begin{lemma} \label{lem:hisone}
Let $(L,(l_{i,i+1})_{i=0}^{n-1})$ be as in Theorem \ref{th:gale}. Then for all $i$ we have
\[
\chi_Y(l_{i-1,i},l_{i,i+1})=r_i^2.
\]
\end{lemma}
\begin{proof} 
By \cite[Proposition 4.17]{KuznetsovToric} there exists a non-zero $h\in \RR$ such that
\begin{equation}
\label{eq:volume}
\chi_Y(l_{i-1,i},l_{i,i+1})=hr_i^2.
\end{equation}
for all $i$.  We only need to compute $h$ for a single $i$.  Let $i=n-1$.  It follows from \eqref{eq:cycle} and Remark \ref{rem:extreme_dual}
that 
\begin{align*}
l_{n-2,n-1} &= -[\Oscr_y]-r_{n-1} [F_{n-1,Y}]\\
&=-[\Oscr_y]-r_{n-1} [E_{n-1,Y}]
\end{align*}
so that
\begin{align*}
\chi_Y(l_{n-2,n-1},l_{n-1,0})&=\chi_Y(-[\Oscr_y]-r_{n-1} [E_{n-1,Y}], -[\Oscr_y])\\
&=r_{n-1}^2. \qedhere
\end{align*}
\end{proof}

\begin{corollary} \label{cor:choose}
Let $(L,(l_{i,i+1})_i)$ be the Gale dual of $(\Pic(X)_\RR, (T_{i,i+1})_{i=0,\ldots, n-1})$ (not necessarily the model of Theorem \ref{th:gale}).  Choose a volume form $\omega \in\Lambda^2 L \cong \RR$ such
that $\omega(l_{i-1,i},l_{i,i+1})>0$ for all $i$ (see  \cite[Proposition 4.17]{KuznetsovToric}). Then we have for $i>j$.
\begin{equation}
\label{eq:chi}
\chi_X(F_i,F_j)= \chi_Y([F_{i,Y}],[F_{j,Y}]) = \omega(\overline{m}_i,\overline{m}_{j})
\end{equation}
where
\[
\overline{m}_i=\frac{l_{i,i+1}-l_{i-1,i}}{\sqrt{\omega(l_{i-1,i},l_{i,i+1})}}.
\]
\end{corollary}
\begin{proof} We may use the concrete realisation of $(L,(l_{i,i+1})_i)$ constructed in Theorem \ref{th:gale}. Moreover,  $\omega$ is a scalar multiple of $\chi_Y$, since any two volume forms differ only by a scalar.  Since the right-hand side
of \eqref{eq:chi} is independent of $\omega$,  we may assume $\omega\in \{-\chi_Y,\chi_Y\}$ and by Lemma \ref{lem:hisone} we then must take $\omega=\chi_Y$.
Then we compute
\begin{align*}
\omega(\overline{m}_i,\overline{m}_j)&=\frac{\chi_Y(r_i[F_{i,Y}],r_j[F_{j,Y}])}{\sqrt{\omega(l_{i-1,i},l_{i,i+1})}\sqrt{\omega(l_{j-1,j},l_{j,j+1})}}= \frac{\chi_Y(r_i[F_{i,Y}],r_j[F_{j,Y}])}{r_ir_j} =\\&=\chi_Y([F_{i,Y}],[F_{j,Y}]),
\end{align*}
where we have used Lemma \ref{lem:hisone} again for the second equality.  Since $i>j$ and therefore $(F_i,F_j)$ is an exceptional pair, one has $\chi_Y([F_{i,Y}],[F_{j,Y}])=\chi_X(F_{i},F_{j})$ (see e.g. \cite[Lemma 3.12]{NVdB2}).
\end{proof}

\begin{lemma}
\label{lem:canonical}
Let $(L,(l_{i,i+1})_i)$ be as in Theorem \ref{th:gale}. 
Then the lattice $L_{\ZZ}$ in $L$ spanned by $(l_{i,i+1})_i$ is $K^{\num}_0(Y)$.
\end{lemma}
\begin{proof} Clearly $L_{\ZZ}\subset K^{\num}_0(Y)$.  For a point $p \in X$ the map ${\chi_X(-,[\Oscr_p]):K_0(X)\r \ZZ}$ is surjective since $\chi_X([\Oscr_X],[\Oscr_p])=1$.  One has $r_i = \chi_X([E_i],[\Oscr_p])$,  hence
we conclude $\gcd(\{r_i\}_{i=0}^{n-1})=1$.  If $[K^{\num}_0(Y):L_{\ZZ}]:=d$,  then $d$ divides $\chi_Y(-,-)$ when evaluated on $L_{\ZZ}$.
It follows from Lemma \ref{lem:hisone} that $d$ divides $r_i^2$ for all $i$ which contradicts the fact that the $r_i$ are coprime, unless $d=1$.
\end{proof}

\begin{theorem} \label{thm:achievement}
Let $(L,(l_{i,i+1})_{i=0}^{n-1})$ be the Gale dual of $(\Pic(X)_\RR, (T_{i,i+1})_{i=0}^{n-1})$ and let $L_{\ZZ}$ be the lattice generated by $(l_{i,i+1})_i$.
Let $\omega$ be the volume form on $L$ such that the fundamental volume of the lattice $L_{\ZZ}$ is $1$, satisfying in addition $\omega(l_{i-1,i},l_{i,i+1})>0$ for some or,  equivalently,  all $i$.
Then
\begin{enumerate}
\item \label{it:thm2_1} $m_i = l_{i,i+1}-l_{i-1,i}$ is $r_i$ times a primitive element of $L_{\ZZ}$ (hence $(r_i)_i$ are determined by $(l_{i,i+1})_i$).
\item \label{it:thm2_2} $l_{i,i+1}$ is primitive in $L_{\ZZ}$.
\item \label{it:thm2_3} $r_i^2=\omega(l_{i-1,i}, l_{i,i+1})$.
\item  \label{it:thm2_4}  For $i<j$ put $\chi_{ij}=\chi_X(E_i,E_j)$. Then
\begin{equation}
\label{eq:adjacent_chi}
\chi_{i,i+1}=\frac{\omega(m_i,m_{i+1})}{r_ir_{i+1}}
\end{equation}
and 
\begin{equation}
\label{eq:additivity}
\chi_{i,j}= r_ir_j \sum_{k=i}^{j-1} \frac{\chi_{k,k+1}}{r_kr_{k+1}} = r_ir_j \sum_{k=i}^{j-1} \frac{\omega(m_k,m_{k+1})}{(r_kr_{k+1})^2}.
\end{equation}
\item  \label{it:thm2_5}  The number of arrows from $i$ to $j$ in the quiver (see Lemma \ref{lm:arrowschiY}) corresponding to $(E_i)_i$ is equal to
\[
\frac{\omega(m_i,m_{j})}{r_ir_{j}}
\]
with the convention that a negative number corresponds to arrows from $j$ to~$i$.
\end{enumerate}
\end{theorem}
\begin{proof} By Lemma \ref{lem:canonical} we may,  and we will, assume that $L_\ZZ=K^{\num}_0(Y)$ and $\omega=\chi_Y$.
Then \eqref{it:thm2_1} follows from \eqref{eq:formula} and the fact that $F_{j,Y}$ are spherical and hence have primitive $(\text{rank}, \text{degree})$-vectors.

To prove \eqref{it:thm2_2} we note that $(\Pic(X)_\RR, (T_{i,i+1})_{i=0}^{n-1})$ is compatible with rotation of $(E_i)_i$ and hence so is its Gale dual
$(L,(l_{i,i+1})_i)$. Hence we may assume $i=n-1$. It follows from \eqref{eq:formula} that $l_{n-1,0}=-[\Oscr_p]$ which is clearly primitive.

\eqref{it:thm2_3} is just Lemma \ref{lem:hisone}.

\eqref{it:thm2_5} follows from Corollary \ref{cor:choose} and Lemma \ref{lm:arrowschiY} taking \eqref{it:thm2_3} into account.

\eqref{eq:adjacent_chi} in \eqref{it:thm2_4} is a special case of \eqref{it:thm2_5}. \eqref{eq:additivity} in \eqref{it:thm2_4} follows from \eqref{eq:chi_slope}.
\end{proof}

\begin{definition} \label{def:normalized}
The volume form $\omega$ as introduced in Theorem \ref{thm:achievement} will be called the \emph{normalised volume form} on $L$.
\end{definition}
Below we will use the following geometric reinterpretation of Theorem \ref{thm:achievement}(\ref{it:thm2_3},\ref{it:thm2_4}).
\begin{lemma}\label{lm:chiareas} Let $\omega$ be an a volume form on $L$ such that $\omega(l_{i-1,i},l_{i,i+1})>0$.
Define the triangles $A = \conv(0,l_{i-1,i},l_{i,i+1})$, $B = \conv(0,l_{i,i+1},l_{i+1,i+2})$, $C= \conv(l_{i-1,i},l_{i,i+1},l_{i+1,i+2})$. Then 
\begin{equation}
  \label{eq:area_formula}
\chi(E_i,E_{i+1}) = \frac{\area(C)}{\sqrt{\area(A)\area(B)}}.
\end{equation}
where $\Delta(C)$ should be interpreted as a signed area.
\end{lemma}
Proposition 10.7 in \cite{Perling} shows that $(l_{i,i+1})_{i=0}^{n-1}$ defines a complete fan in $L$.
It follows that the intervals $[l_{i-1,i},l_{i,i+1}]$ are the edges of a polygon $P$ (see \S\ref{sec:notation})  that contains the origin.
We call $P$ the \emph{polygon associated to the full exceptional collection~$\EE$}.  
We think of the edges of $P$ as being indexed by the elements of the full exceptional collection (i.e.  the edge $[l_{i-1,i},l_{i,i+1}]$ corresponds to $E_i$).
We recall the following beautiful result:
\begin{proposition}[{\cite[\S 9]{Perling}}]\label{prop:convex_perling}
The collection $\EE$ is very strong if and only if $P$ is convex.
\end{proposition}

The following observation will very useful for us.
\begin{lemma}
  \label{lem:very_useful}
Let 
\[
\EE=((E_0^{(0)},\ldots, E_0^{(\alpha_0-1)}) , 
\ldots,
(E_{k-1}^{(0)},\ldots, E_{k-1}^{(\alpha_{k-1}-1)} ))
\]
be a very strong block exceptional collection (see \S\ref{sec:block_ex}) without broken blocks (see Definition \ref{def:block_form}).
Let $P$ be the polygon associated to $\EE$. Then the long edges of $P$ (see \S\ref{sec:notation})
correspond to the blocks in $\EE$. The edges are obtained by dividing every long edge $e_i$, corresponding to the block $\EE_i$, into $\alpha_i$ equal segments
corresponding to $E_i^{(0)},\ldots, E_i^{(\alpha_i-1)}$, respectively.
\end{lemma}
\begin{proof}
The right dual to the exceptional collection underlying $\EE$ can be made into a block exceptional collection as 
stated in  Lemma \ref{lem:dual_block_exceptional}.
  Let
 $
\FF=(\FF_{k-1},
\ldots,
\FF_0)
$
be the result.\footnote{Recall that $\FF$ is constructed by first performing appropriate block mutations on $\EE$ and then inverting the order of the exceptional objects
  in the blocks.}
We claim that $\FF$ has no broken blocks.
The blocks $\FF_{i},\FF_{i-1}$ for $i=k-1,\ldots,1$ are not
orthogonal, since if they were they could be joined into a single block to form a new full block exceptional collection $\FF'$. Then the left dual $\EE'$ of $\FF'$ would have the same underlying exceptional collection as~$\EE$, but fewer blocks, contradicting the hypothesis that $\EE$ has no broken blocks.

It remains to show that $\FF_{0}$ is not orthogonal to $\omega_X^{-1}\otimes \FF_{n-1}$. This follows from Lemma \ref{lem:dual_not_broken} below and Proposition \ref{pr:slopes}.

Using \eqref{eq:chi} applied to pairs $(F_{i+1},F_{i})$ we see that the straight vertices of $P$ are given by the intra block pairs and the non-straight vertices are given by pairs
joining two adjacent blocks. It remains to show that the straight vertices divide the long edges of $P$ into equal segments. This follows from
the explicit model \eqref{eq:formula} and the fact that if $E,F$ is an orthogonal exceptional pair then $[E_Y]=[F_Y]$ in $K^{\num}_0(Y)_{\RR}$.  Indeed by \eqref{eq:chi_slope}
$[E_Y]$ and $[F_Y]$ are proportional. But since $E_Y$, $F_Y$ are spherical objects on $Y$, their (rank,degree)-vectors are primitive. This finishes the proof.
\end{proof}
\begin{lemma} \label{lem:dual_not_broken}
  Let
  \[
    \EE=(E_0,\ldots,E_{n-1})
  \]
  be a full exceptional collection on $X$ and let
  \[
    \FF=(F_{n-1},\ldots,F_0)
  \]
  be its right dual. Then 
  \[
    \RHom_X(F_0,\omega_X^{-1}\otimes F_{n-1}[-\dim X])=\RHom_X(E_0,E_{n-1})
  \]
\end{lemma}
\begin{proof}
  According to Remark \ref{rem:extreme_dual} one has $F_{n-1}=E_{n-1}$ and $F_0=\omega_X^{-1}\otimes E_0[-\dim X]$. This yields what we want.
  \end{proof}
  \section{Mutations on the level of toric systems and polygons}
  \subsection{Description of mutations}
Let $\EE = (E_0, \dots,E_{n-1})$ be a full exceptional collection of vector bundles on a del Pezzo surface,  $(T_{0,1}, \dots,T_{n-1,0})$ the corresponding toric system in $\Pic(X)_\QQ$ and $(L,(l_{0,1},\dots,l_{n-1,0}))$ the Gale dual.
In this section we explain how mutations of $\EE$ act on $(l_{i,i+1})_0^{n-1}$. In \eqref{eq:rem:geommut} we do this for braid mutations (see \S\ref{sec:ex_col}), and in Proposition \ref{prop:lm:clustergeneralised}
we do this for quiver mutations (see Proposition \ref{prop:rem:cluster}). See also \cite[Proposition 8.5]{Perling} for related formulas. As before we put $m_i=l_{i+1,i}-l_{i,i-1}$.
\begin{lemma}\label{lm:ts_and_ls_mut} Denote $h=\chi(E_i,E_{i+1})$. 
\begin{enumerate}
\item\label{it:tsmutleft} Assume $hr(E_i)-r(E_{i+1}) \neq 0$, i.e.  $\tilde{\sigma}_{i+1}(\EE)$ consists of vector bundles. The toric system corresponding to 
\[
\tilde{\sigma}_{i+1}(\EE) =(E_0,\dots, E_{i-1},\widetilde{L_{E_i}} E_{i+1}, E_i, E_{i+2}, \dots, E_{n-1})
\]
 is given by 
\[ 
(T_{0,1},  \dots, T_{i-2,i-1}, T_{i-1,i}-\alpha T_{i,i+1}, \alpha T_{i,i+1},T_{i,i+1} + T_{i+1,i+2},T_{i+2,i+3},\dots,  T_{n-1,0}),
\] where 
\begin{align*}
  \alpha =\frac{T_{i-1,i} \cdot T_{i,i+1}}{(T_{i,i+1}+T_{i+1,i+2}) \cdot T_{i,i+1}} =  \frac{r(E_{i+1})}{hr(E_i)-r(E_{i+1})}
=  \frac{\omega(m_{i+1},l_{i+1,i+2})}{\omega(m_{i+1},l_{i-1,i})}
  .
\end{align*}
For the Gale dual corresponding to $\tilde{\sigma}_{i+1}(\EE)$ we may take the following vectors in $L$:
\begin{equation}
  \label{eq:gale_left}
(l_{0,1}, \dots,l_{i-1,i},l_{i-1,i}+\frac{\omega(l_{i-1,i},m_{i+1})}{\omega(l_{i,i+1},m_{i+1})} m_{i+1}, l_{i+1,i+2},  \dots,l_{n-1,0}).
\end{equation}
\item\label{it:tsmutright} Assume $hr(E_{i+1})-r(E_i) \neq 0$, i.e.  $\tilde{\sigma}_{i+1}^{-1}(\EE)$ consists of vector bundles.  The toric system corresponding to \[
\tilde{\sigma}_{i+1}^{-1}(\EE) = (E_0, \dots, E_{i-1},E_{i+1},\widetilde{R_{E_{i+1}}} E_i,  E_{i+2}, \dots, E_{n-1})
\]
is given by 
\[ (T_{0,1}, \dots,T_{i-2,i-1},T_{i-1,i}+T_{i,i+1},\alpha T_{i,i+1}, T_{1,2}-\alpha T_{i,i+1},T_{i+2,i+3}, \dots,  T_{n-1,0}),
\]
where 
\[
  \alpha=\frac{T_{i+1,i+2} \cdot T_{i,i+1}}{T_{i,i+1}\cdot (T_{i-1,i}+T_{i,i+1)}}=\frac{r(E_i)}{hr(E_{i+1})-r(E_i)}
  =-\frac{\omega(l_{i,i+1},m_i)}{\omega(l_{i+1,i+2},m_i)}
  .
\]
Then for the Gale dual corresponding to $\tilde{\sigma}_{i+1}^{-1}(\EE)$ we may take the following vectors in $L$:
\begin{equation}
  \label{eq:gale_right}
(l_{0,1},  \dots, l_{i-1,i}, l_{i+1,i+2}- \frac{\omega(l_{i+1,i+2},m_i)}{\omega(l_{i,i+1},m_i)} m_i , l_{i+1,i+2},\dots,l_{n-1,0}).
\end{equation}
\end{enumerate}
\end{lemma} 

\begin{proof} We will prove only the first part.  Directly from the definition of $T_{ij}$'s we get that the toric system corresponding to
  $\tilde{\sigma}_{i+1}(\EE) =(E_0,\dots, E_{i-1},\widetilde{L_{E_i}} E_{i+1}, E_i, E_{i+2}, \dots, E_{n-1})$ is given by 
\[ 
(T_{0,1},  \dots, T_{i-2,i-1}, T', \alpha T_{i,i+1},T_{i,i+1} + T_{i+1,i+2},T_{i+2,i+3},\dots,  T_{n-1,0}),
\]
for some $\alpha \in \QQ$ with $T' +  \alpha T_{i,i+1} = T_{i-1,i}$.  Moreover,  by Proposition \ref{prop:Ts}\eqref{it:prop:Ts_2}  we have $T' \perp T_{i,i+1} + T_{i+1,i+2}$.  Hence $\alpha =\frac{T_{i-1,i} \cdot T_{i,i+1}}{(T_{i,i+1}+T_{i+1,i+2}) \cdot T_{i,i+1}}$.  Using Proposition \ref{prop:Ts}(\ref{it:prop:Ts_1},\ref{it:prop:Ts_4}) combined with \eqref{eq:chi_T} we have 
\[
\alpha = \frac{1/r(E_i)^2}{1/r(E_{i+1})^2 + (-1/r(E_i)^2 + h/r(E_i)r(E_{i+1}) - 1/r(E_i)^2)} = \frac{r(E_{i+1})}{hr(E_i) - r(E_{i+1})},
\]
where $h=\chi(E_i,E_{i+1})$. Using Theorem \ref{thm:achievement}(\ref{it:thm2_3},\ref{it:thm2_4}) we compute
\begin{align*}
  \frac{r(E_{i+1})}{hr(E_i) - r(E_{i+1})}&=\frac{\sqrt{\omega(l_{i,i+1},l_{i+1,i+2})}}{
                                           \frac{\omega(m_i,m_{i+1})}{\sqrt{\omega(l_{i-1,i},l_{i,i+1})\omega(l_{i-1,i},l_{i,i+1})}}\sqrt{\omega(l_{i-1,i},l_{i,i+1})}-
                                           \sqrt{\omega(l_{i,i+1},l_{i+1,i+2})}}\\
                                         &=
                                           \frac{\omega(l_{i,i+1},l_{i+1,i+2})}{\omega(m_i,m_{i+1})-\omega(l_{i,i+1},l_{i+1,i+2})}\\
                                         &=\frac{\omega(l_{i,i+1},m_{i+1})}{\omega(m_i,m_{i+1})-\omega(l_{i,i+1},m_{i+1})}\\
  &=-\frac{\omega(l_{i,i+1},m_{i+1})}{\omega(l_{i-1,i},m_{i+1})}
  \end{align*}

By the definition of Gale dual (see Definition \ref{def:gale_dual}) we have
\[
l_{0,1}\otimes T_{0,1}+l_{1,2}\otimes T_{1,2}+\dots +l_{n-1,0}\otimes T_{n-1,0}=0 \text{ in } L \otimes \Pic(X)_\RR.
\]
Then 
\begin{align*} l_{0,1}\otimes T_{0,1} + \dots + l_{i-1,i}\otimes (T_{i-1,i}-\alpha T_{i,i+1}) +
(l_{i-1,i} + \frac{l_{i,i+1} - l_{i+1,i+2}}{\alpha})\otimes \alpha T_{i,i+1} + \\ + l_{i+1,i+2}\otimes (T_{i,i+1}+T_{i+1,i+2}) + \dots + 
 +l_{n-1,0}\otimes T_{n-1,0} =0  \text{ in } L \otimes \Pic(X)_\RR.
\end{align*}
Invoking once again the definition of the Gale dual, we get $l_{i,i+1}' = l_{i-1,i}-\frac{m_{i+1}}{\alpha}$. Substituting the expression for $\alpha$
yields \eqref{eq:gale_left}.
 \end{proof}
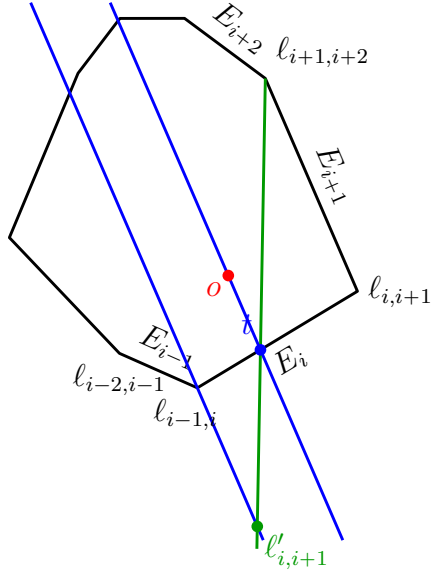
\begin{figure}[h!]\label{fig:braidmutleft}
\centering
\caption{Geometric interpretation of $\tilde{\sigma}_{i+1}$ on the level of $\{l_{i,i+1}\}_0^{n-1}$}

\begin{tikzpicture}[scale=1.1, line join=round,
  blu/.style={line width=1.1pt,blue},
  grn/.style={line width=1.1pt,green!60!black},
  poly/.style={line width=1.1pt,black}
]

\coordinate (l0) at ( 1.9281,-1.0401);
\coordinate (l1) at ( 0.8151, 1.5339);
\coordinate (l2) at (-0.1612, 2.2602);
\coordinate (l3) at (-0.9429, 2.2595);
\coordinate (l4) at (-1.4482, 1.5966);
\coordinate (l5) at (-2.2779,-0.3910);
\coordinate (l6) at (-0.9475,-1.7911);
\coordinate (l7) at (-0.0065,-2.2083);


\path[use as bounding box] (-2.85,-4.25) rectangle (2.20,2.75);

\pgfmathsetmacro{\m}{(-1.0401-1.5339)/(1.9281-0.8151)}   
\pgfmathsetmacro{\loa}{-2.2083 - \m*(-0.0065)}           

\def\ytop{2.45}
\def\ybot{-4.05}

\pgfmathsetmacro{\xUPtop}{\ytop/(\m)}
\pgfmathsetmacro{\xUPbot}{\ybot/(\m)}
\pgfmathsetmacro{\xLOtop}{(\ytop-\loa)/(\m)}
\pgfmathsetmacro{\xLObot}{(\ybot-\loa)/(\m)}

\coordinate (upA) at (\xUPtop,\ytop);
\coordinate (upB) at (\xUPbot,\ybot);
\coordinate (loA) at (\xLOtop,\ytop);
\coordinate (loB) at (\xLObot,\ybot);

\draw[poly]
  (l0)--(l1)--(l2)--(l3)--(l4)--(l5)--(l6)--(l7)--cycle;

\path (l6)--(l7) node[midway,sloped,above] (EN2) {$E_{i-1}$};
\path (l7)--(l0) node[midway,sloped,below] (EN1) {$E_{i}$};
\path (l0)--(l1) node[midway,sloped,above] (E0)  {$E_{i+1}$};
\path (l1)--(l2) node[midway,sloped,above] (E1)  {$E_{i+2}$};

\path[name path=UP] (upA) -- (upB);
\path[name path=LO] (loA) -- (loB);
\path[name path=Eedge] (l7) -- (l0);

\path[name intersections={of=Eedge and UP, by=t}];

\path[name path=green] (l1) -- ($(l1)!2.0!(t)$);
\path[name intersections={of=green and LO, by=p}];


\draw[blu] (upA) -- (upB);
\draw[blu] (loA) -- (loB);
\draw[grn] (l1) -- ($(l1)!1.05!(p)$);

\fill[blue] (t) circle (2pt);
\fill[green!60!black] (p) circle (2pt);

\pgfmathsetmacro{\yo}{-0.85}
\pgfmathsetmacro{\xo}{\yo/(\m)}
\coordinate (o) at (\xo,\yo);
\fill[red] (o) circle (2pt);

 ---------------------------
\node[blue,above=2pt,xshift=-5pt] at (t) {$t$};
\node[green!60!black,below right=-2pt] at (p) {$\ell'_{i,i+1}$};
\node[red,below left=-2pt] at (o) {$o$};

\node[right]       at (l0) {$\ell_{i,i+1}$};
\node[above right] at (l1) {$\ell_{i+1,i+2}$};
\node[below left,xshift=12pt] at (l7) {$\ell_{i-1,i}$}; 
\node[below]       at (l6) {$\ell_{i-2,i-1}$};

\end{tikzpicture}

\end{figure}

For $u,v\in L$ such that  $\omega(u,v)\neq 0$ let $A_{uv}$ be the unique affine transformation (shear mapping) of $\RR^2$ such that 
\begin{enumerate}
\item $A_{uv}(u)=v$.
\item $A_{uv}$ is the identity on the line through the origin parallel to $[u,v]$.
\end{enumerate}
Explicitly, $A_{uv}$ is given by
\begin{equation}
\label{eq:A}
A_{uv}(x) = x+\frac{\omega(x,v-u)}{\omega(u,v-u)}(v-u)
\end{equation}
Using this notation we may rewrite the left and right mutations of $(l_{01},\ldots,l_{n-1,0})$ as
\begin{equation}
  \label{eq:rem:geommut}
\begin{gathered}
  (l_{0,1}, \dots,l_{i-1,i},A_{l_{i,i+1},l_{i+1,i+2}}(l_{i-1,i}), l_{i+1,i+2},  \dots,l_{n-1,0})\\
  (l_{0,1},  \dots, l_{i-1,i}, A_{l_{i,i+1},l_{i-1,i}}(l_{i+1,i+2}) , l_{i+1,i+2},\dots,l_{n-1,0}).
\end{gathered}
\end{equation}
See Figure \ref{fig:braidmutleft} for an illustration of the first formula.
For use below we record the following trivial lemma.
\begin{lemma}
  \label{lem:switch}
  If $A_{uv}(x)=y$ with $\omega(x,y)\neq 0$ then $A_{uv}=A_{xy}$.
\end{lemma}
\begin{definition} Assume $l_{j,j+1}$ is not  on the line through the origin, parallel to $m_i$.
  The  \emph{generalised right mutation of
    $l_{i,i+1},\ldots,l_{i-1,i}$ at $m_i$ with respect to $l_{j,j+1}$} 
  is
  \begin{equation}
    \label{eq:cycle3}
    (l_{0,1}, \dots, l_{i-1,i}, l'_{i,i+1},\ldots,l'_{j-1,j},l_{j,j+1},\ldots,l_{n-1,0}),
    \end{equation}
    where $l'_{i,i+1},\ldots,l'_{j-1,j} :=A(l_{i+1,i+2}),\ldots,A(l_{j,j+1})$ with  $A:=A_{l_{i,i+1},l_{i-1,i}}$. 
    The expression \eqref{eq:cycle3} should be regarded as cyclic. I.e. the indices of the $(l')_{u,u+1}$
    wrap around modulo $n$.
\end{definition} 
\begin{remark}
  \label{rem:composition}
  Geometrically,  one can think of a generalised mutation as:
  \begin{itemize}
  \item
    collapsing edge $m_i$;
  \item
    displacing $m_{i+1},\ldots,m_j$ to obtain edges $m'_i,\ldots,m'_{j-1}$;
  \item
    creating a new edge $m'_j$ at the vertex $l_{j,j+1}$ (i.e.\ more explicitly $(l'_{j-1,j},l_{j,j+1})$),  parallel to the collapsed one.
  \end{itemize}
Note that for $j=i+1$ this operation coincides with the corresponding braid mutation described \eqref{eq:rem:geommut}. 
Moreover it follows from Lemma \ref{lem:switch} that  this operation is compatible with composition in the following sense:
the generalised right mutation at $m_i$ with respect to $l_{j,j+1}$ followed by the generalised right mutation at $l'_{j-1,j}-l_{j,j+1}$
(the newly created edge parallel to $m_i$) with respect to $l_{k,k+1}$ is the same as the generalised right mutation at $m_i$ with respect to $l_{k,k+1}$.
\end{remark}
\begin{definition}
  Let $\EE=(E_0,\ldots,E_{n-1})$ be a full exceptional collection. The \emph{generalised right mutation} of $\EE$ at $E_0$ with respect to the pair $(E_j,E_{j+1})$ is
  \begin{equation}
    \label{eq:iis0}
  \begin{aligned}
    \EE'&=\tilde{\sigma}_j^{-1}\cdots\tilde{\sigma}_2^{-1}\tilde{\sigma}_1^{-1}\EE\\
    &=(E_{1},E_{2},\ldots, E_{j} ,\widetilde{R_{E_j}}\cdots \widetilde{R_{E_{1}}} E_0,E_{j+1},\ldots,E_{n-1}).
  \end{aligned}
  \end{equation}
  The generalised right mutation of $\EE$ at $E_i$ with respect to the pair $(E_j,E_{j+1})$ is obtained by first right rotating $\EE$
  such that $E_i$ is at the zeroth position
  (and the $(E_j,E_{j+1})$ are at positions $(E_{j-i},E_{j+1-i})$),
  performing the mutation as defined above and then performing the inverse left mutation.
  \end{definition}
  \begin{lemma} \label{lem:generalised_right}
    Let $\EE'$ be a generalised right mutation of $\EE$ at $E_i$ with respect to the pair $(E_j,E_{j+1})$ and assume both $\EE$ and $\EE'$ consist of
    vector bundles.
Let $\ell:=(l_{01},\ldots,l_{n-1,0})$ and $\ell':=(l'_{01},\dots,l'_{n-1,0})$ be respectively the Gale duals of the toric systems
associated to $\EE$ and $\EE'$. Then $(l'_{01},\dots,l'_{n-1,0})$ is the generalised right mutation of $(l_{01},\ldots,l_{n-1,0})$
at $m_i$ with respect to $l_{j,j+1}$.
\end{lemma}
\begin{proof}
  By applying appropriate rotations we may assume that $i=0$ so that we can use the formula \eqref{eq:iis0}. If $j=1$ then the result follows from \eqref{eq:rem:geommut} (see also Figure \ref{fig:braidmutleft} for the dual picture involving a left mutation).
  If all $\tilde{\sigma}^{-1}_k\cdots \tilde{\sigma}_1^{-1}\EE$
  are exceptional collections of vector bundles, then the claim follows from the compatibility of generalised right mutations with composition (see Remark \ref{rem:composition}).

  Sadly if some intermediate mutation involves a torsion sheaf, then things are more complicated. One possibility is to extend the machinery of toric systems to this setting
  as is done in \cite[Definition 5.5]{Perling}. Here we choose a different route however.
  We will replace $\EE$, $\EE'$ by the corresponding (real) numerical
  exceptional collections $\ee$ and $\ee'$ in $K_0(X)_\RR$.  

  According to Lemma \ref{lem:families} below we may construct an embedded smooth curve $\ee_t:(-\epsilon,\epsilon)\r \Mscr$ such that $\ee_0=\ee$
  and such that for $t\neq 0$ none of the $\sigma_k^{-1}\cdots\sigma_1^{-1} (\ee_t)$ contains
  objects of rank $0$ for $t\neq 0$. Let $\ell_t$ be the Gale dual of the toric system (the latter defined via \eqref{eq:real_toric})  corresponding
  to\footnote{It is easy to see that $\ell_t$ can be constructed in such a way that it depends smoothly on $t$, e.g.\ using the formula \eqref{eq:formula}.} $\ee_t$.
  Now let $\ell'_t$ be obtained  by the generalised right mutation of $\ell_t$ at $(m_i)_t$ with respect to $(l_{j,j+1})_t$. For $t\neq 0$ it follows from
  compatibility with compositions as in the first paragraph of this proof, that $\ell'_t$ is the Gale dual of $\ee'_t:=\tilde{\sigma}_j^{-1}\cdots\tilde{\sigma}_1^{-1}\ee_t$.
  Taking the limit $t\r 0$ finishes the proof as it is clear that Definition \ref{def:gale_dual}\eqref{it:def:gale_dual_3} is compatible with limits.
\end{proof}
\begin{remark}
  \label{lem:block}
  A useful special case of this lemma is given when $l_{j,j+1}$ is on the same long edge as $m_i$. Then $E_i,\ldots,E_j$ are orthogonal and the generalised
  right mutation at $m_i$ with respect to $l_{j,j+1}$ corresponds to the permutation $E_i,\ldots,E_j\mapsto E_{j+1},\ldots,E_j,E_i$.
  \end{remark}
We now give the results that were used in the proof of Lemma \ref{lem:generalised_right}.
\begin{lemma}[{See also \cite[Theorem 3.2]{BondalSymplectic}}]
  Let $V$ be a real vector space of dimension~$n$ equipped with a bilinear form $\chi$.
  Let $\Mscr_u$ be the real algebraic subvariety of $V^u$ consisting of numerical exceptional collections of length $u$ in $V$ (with $\Mscr_{0}$ being just a point). Then $\Mscr_u$ is a submanifold
  of $V^u$. Moreover the map $\pi:\Mscr_{u+1}\r \Mscr_u$ which drops the last component is a fiber bundle over a Zariski
  open subset of $\Mscr_u$. 
\end{lemma}
\begin{proof}
  Let $\chi^+$ be the symmetrization of $\chi$.
  Clearly $\Mscr_{-1}$ is a smooth submanifold of $V^0$ since both sets are just single points. Assume that we have shown that $\Mscr_u$ is a submanifold of $V^u$.
  Then the image of $\pi:\Mscr_{u+1}\r \Mscr_u$ consists of the numerical exceptional collections $(e_0,\ldots,e_{u-1})$ which can be extended to a numerical
  exceptional collection $(e_0,\ldots,e_{u-1},e_u)$. This is possible when $L={}^{\perp_{\chi}}(e_0,\ldots,e_{u-1})$ is not isotropic for the quadratic form $\chi^+$, which
  is clearly a Zariski open condition. Moreover, the fiber $\pi^{-1}(e_0,\ldots,e_{u-1})$ corresponds to the $e_{u}$ such that $\chi^+(e_{u},e_{u})=1$,  which is a submanifold of $V$.
  From this one deduces that $\Mscr_{u+1}$ is a submanifold of $V^{u+1}$.
\end{proof}
\begin{corollary} If $v\ge u$ then the projection map $\Mscr_v\r \Mscr_u$ is a submersion.
  \end{corollary}
  \begin{lemma}
    \label{lem:families}
    Let $V=K_0(X)_\RR$  and let  $\ee\in \Mscr_n$. Choose a generic tangent vector $v$ at $\ee$ in $\Mscr_n$
    and let $\ee_t:(-\epsilon,\epsilon)\r \Mscr_n$ be an embedded smooth curve such that $\ee_0=\ee$
  and $(d\ee_t/dt)_{t=0}=v$. Then possibly after shrinking $\epsilon$, none of the numerical exceptional collections $\sigma_k^{-1}\cdots\sigma_1^{-1} (\ee_t)$ contains
  objects of rank $0$ for $t\neq 0$.
\end{lemma}
\begin{proof} It is sufficient to show that for any $k,l$ the object $(\sigma_k^{-1}\cdots\sigma_1^{-1} (\ee_t))_l$ does not have rank zero for $t\neq 0$ for $(\ee_t)_t$
  a particular embedded curve. Since $\sigma_k^{-1}\cdots\sigma_1^{-1}$ is a diffeomorphism $\Mscr_n\r \Mscr_n$, we may assume $k=0$. Using a suitable rotation we may assume $l=0$.
  If $\ee=(e_0,\ldots,e_{n-1})$, then we put $(e'_0)_t=e_0+t[\Oscr_X]$ and $(e_0)_t=(e'_0)_t/\sqrt{\chi((e_0')_t,(e_0')_t)}$. Then $(e_0)_0=e_0$ and for $t$ small, $(e_0)_t$ is exceptional and has non-zero rank if $t\neq 0$.
  Then, since $\Mscr_n\r \Mscr_1$ is a submersion, a corresponding lift $(\ee_t)_t$ with the required properties exists.
  \end{proof}
Recall that by Proposition \ref{prop:clusterVSregularmut} a quiver mutation can be expressed as a composition of a sequence of adjacent braid mutations.  Hence
quiver mutations of very strong exceptional collections can be viewed as generalised mutations of the corresponding convex polygons. Below we make this explicit.
\begin{definition} \label{def:opposing}
  Let $H_i$ denote the line parallel to $m_i$ and such that it is the furthest from $m_i$ but still intersects the polygon $\conv(l_{01},\ldots,l_{n-1,0})$. The vertices vertices $l_{j,j+1}$ on the line $H_i$ are called \emph{opposing vertices} for $m_i$. The earliest $l_{j,j+1}$ after $m_i$ in the cyclic ordering
  is called the \emph{earliest opposing vertex} for $m_i$.
\end{definition}
\begin{proposition}
  \label{prop:lm:clustergeneralised} Let $\EE=(E_0,\ldots,E_{n-1})$ be a very strong exceptional collection with corresponding Gale dual
  $(l_{01},\ldots,l_{n-1,0})$.
  Then the Gale dual of the toric
  system associated to the quiver mutation at $E_i$ is given by
  \[
    (l_{0,1}, \dots, l_{i-1,i}, A(l_{i+1,i+2}),\ldots,A(l_{j,j+1}),l_{j,j+1},\ldots,l_{n-1,0}),
  \]
with  $A:=A_{l_{i,i+1},l_{i-1,i}}$ (with cyclic indices) and $l_{j,j+1}$ the earliest opposing vertex for $m_i$.
\end{proposition}
\begin{proof}
  As usual we may assume $i=0$.
  By Proposition \ref{prop:clusterVSregularmut} the quiver mutation at $0$ of $\EE$ is of the form 
  \[
    \begin{aligned}
      \EE'&=\tilde{\sigma}_j^{-1}\cdots\tilde{\sigma}_2^{-1}\tilde{\sigma}_1^{-1}\EE\\
                     &=(E_{1},E_{2},\ldots, E_{j} ,\widetilde{R_{E_j}}\cdots \widetilde{R_{E_{1}}} E_0,E_{j+1},\ldots,E_{n-1}).
    \end{aligned}
  \]
  for a suitable $j\in \{1,\ldots,n-1\}$ such that  $\EE'$ is very strong. 

 Because both $\EE$ and $\EE'$ are very strong, the corresponding polygons $P$, $P'$ are convex.
 Moreover by Lemma \ref{lem:generalised_right} $P'$ is obtained from $P$ by the generalised right mutation at $m_0$ with respect to $l_{j,j+1}$.
 Direct inspection of Figure \ref{fig:quiver_mutation} shows that the set of $j$'s such that this generalised right mutation preserves convexity consists of
the opposing vertices for $m_i$.
Let $l_{j,j+1}$ be the earliest opposing vertex. By compatibility with composition, the generalised right mutation at $m_0$ with respect to another opposing vertex $l_{j',j'+1}$ is the generalised right mutation with respect to $l_{j,j+1}$ followed by a generalised right mutation within a long edge. By Remark \ref{lem:block} this corresponds to permuting the elements of a block, which has no effect on the resulting rolled-up helix algebra.
\end{proof}

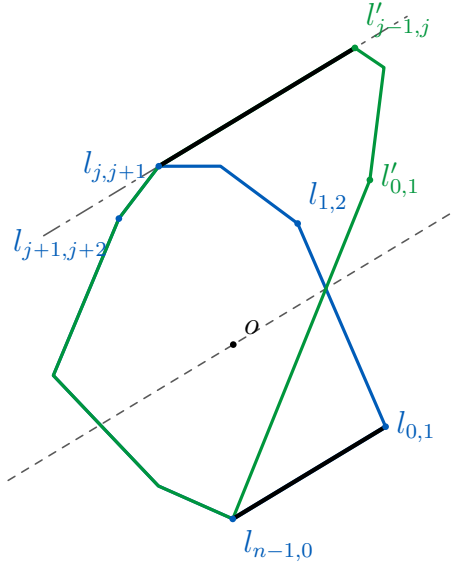
\begin{figure}[h!]\label{fig:clustermut}
\centering
\caption{The green polygon $(l_{0,1}',\dots,l_{j-1,j}',l_{j,j+1},\dots, l_{n-1,0})$ is obtained from the blue polygon $(l_{0,1},\dots,l_{n-1,0})$ by the generalised right mutation at $m_0$ with respect to $l_{j,j+1}$.  Here this is also the quiver mutation at $m_0$. }
\label{fig:quiver_mutation}
\begin{tikzpicture}[scale=1.05, line join=round, line cap=round]
  \definecolor{orig}{RGB}{0,92,184}
  \definecolor{mut}{RGB}{0,150,60} 

  \coordinate (o)  at (0,0);

  \coordinate (l0) at (1.9281,-1.0401);
  \coordinate (l1) at (0.8151, 1.5339);
  \coordinate (l2) at (-0.1612, 2.2602);
  \coordinate (l3) at (-0.9429, 2.2595);  
  \coordinate (l4) at (-1.4482, 1.5966);  
  \coordinate (l5) at (-2.2779,-0.3910);
  \coordinate (l6) at (-0.9475,-1.7911);
  \coordinate (l7) at (-0.0065,-2.2083);  

  \coordinate (lp0) at (1.7293, 2.0859);  
  \coordinate (lp1) at (1.9079, 3.5097);  
  \coordinate (lp2) at (1.5398, 3.7587);  

  \coordinate (axisA) at (-2.8500,-1.7210);
  \coordinate (axisB) at ( 2.8500, 1.7210);

  \coordinate (supA)  at (-2.4000, 1.3795);
  \coordinate (supB)  at ( 2.2000, 4.1574);

  \draw[black!65, dashed, line width=0.6pt] (axisA) -- (axisB);
  \draw[black!65, dashed, dash pattern=on 7pt off 3pt on 1.5pt off 3pt, line width=0.6pt] (supA) -- (supB);

  \draw[orig, very thick]
    (l0)--(l1)--(l2)--(l3)--(l4)--(l5)--(l6)--(l7)--cycle;

  \draw[mut, very thick]
    (lp0)--(lp1)--(lp2)--(l3)--(l4)--(l5)--(l6)--(l7)--cycle;

  \draw[orig, ultra thick] (l7) -- (l0);
  \draw[black, ultra thick] (l7) -- (l0);

  \draw[mut, ultra thick] (lp2) -- (l3);
  \draw[black, ultra thick] (lp2) -- (l3);

  \fill[black] (o) circle (1.2pt);

  \fill[orig] (l0) circle (1.2pt);
  \fill[orig] (l1) circle (1.2pt);
  \fill[orig] (l7) circle (1.2pt);
  \fill[orig] (l3) circle (1.2pt);
  \fill[orig] (l4) circle (1.2pt);

  \fill[mut] (lp0) circle (1.2pt);
  \fill[mut] (lp2) circle (1.2pt);

  \node[above right, text=black] at (o) {$o$};

  \node[right, text=orig]        at (l0) {$l_{0,1}$};
  \node[above right, text=orig]  at (l1) {$l_{1,2}$};
  \node[below right, text=orig]  at (l7) {$l_{n-1,0}$};

  \node[left, text=orig]         at (l3) {$l_{j,j+1}$};
  \node[below left, text=orig]   at (l4) {$l_{j+1,j+2}$};

  \node[right, text=mut]         at (lp0) {$l_{0,1}'$};
  \node[above right, text=mut]   at (lp2) {$l_{j-1,j}'$};
\end{tikzpicture}
\end{figure}

\subsection{Reduction of ranks and area change}

As before,  let $\EE = (E_0,\dots, E_{n-1})$ be a very strong exceptional collection and $P = \conv (l_{0,1},\dots,l_{n-1,0})$ the corresponding polygon.   
Assume that $\omega$ is normalised. By Theorem  \ref{thm:achievement}\eqref{it:thm2_3} (see also \cite[Proposition 4.17]{KuznetsovToric}) we have 
\[
2\area(P)=\sum_{i=0}^{n-1} \omega(l_{i-1,i},l_{i,i+1}) =  \sum_{i=0}^{n-1} r(E_i)^2,
\]
  Hence a quiver mutation reduces the sum of the ranks of an exceptional collection (\emph{the total rank} of an exceptional collection) if and only if it reduces the area of the corresponding polygon.

Let $P' = \conv (l_{0,1}, \dots, l_{i-1,i}, l_{i,i+1}',\dots,l_{j-1,j}',l_{j,j+1},\dots, l_{n-1,0})$ be the convex polygon obtained from $P$ by the quiver mutation at $i$. 
\begin{lemma} \label{lem:area_reduction}
  $\area(P')-\area(P) < 0$ if and only if
    \[
      \omega(m_i,l_{j,j+1}+l_{i,i+1}) < 0
      \]
\end{lemma}
\begin{proof}
To simplify the notations and without loss of generality we can assume $i=0$.  Divide the polygon $P$ into three parts 
 \[
 P=P_0\cup T \cup P_1, \]
where 
we put $P_0 := \conv(l_{j,j+1},l_{j+1,j+2},\ldots,l_{n-1,0})$, $P_1 := \conv(l_{0,1},l_{1,2},\dots, l_{j,j+1})$ and $T := \conv(l_{n-1,0},l_{0,1},l_{j,j+1})$ (see the left part of Figure \ref{fig:areachange}).
Then for the mutated polygon $P'$ we have 
\[
P'=P_0\cup T'\cup P'_1, 
\]
where $T' = \conv(l_{n-1,0},l_{j-1,j}',l_{j,j+1})$ and $P'_1 = \conv(l_{n-1,0},l'_{0,1},\ldots,l'_{j-1,j})$ (see the right part of Figure \ref{fig:areachange}).

We have $\area(P'_1)=\area(A(P_1)) = \area(P_0)$, since $A$ is a shear mapping and hence preserves the area.  Therefore $\area(P')-\area(P)=\area(T')-\area(T)$.  On the other hand, we have 
\[
2\area(T)=\omega(m_0, l_{j,j+1} - l_{n-1,0})
\]
and
\begin{align*}
2\area(T')&=\omega(l_{j,j+1} - l_{n-1,0},l_{j,j+1}-l'_{j-1,j})\\
&=\frac{\omega(l_{j,j+1},m_0)}{\omega(l_{0,1},m_0)}\omega(l_{j,j+1} - l_{n-1,0},m_0)
\end{align*}
Hence
\begin{align*}
2(\area(T')-\area(T))&=\left(\frac{\omega(l_{j,j+1},m_0)}{\omega(l_{0,1},m_0)}+1\right)\omega(l_{j,j+1} - l_{n-1,0},m_0)\\
&=\frac{\omega(l_{j,j+1}+l_{0,1},m_0)}{\omega(l_{0,1},m_0)}\omega(l_{j,j+1} - l_{n-1,0},m_0)
\end{align*}
One can see from the picture that $\omega(l_{j,j+1} - l_{n-1,0},m_0)<0$ and $\omega(l_{0,1},m_0)>0$.  Hence the sign of
$\area(T')-\area(T)$ is indeed the same as the sign of $\omega(m_0,l_{j,j+1}+l_{0,1})$.  
\end{proof}
\begin{corollary} No quiver mutation reduces the total rank of $\EE$ if and only if for all $i$
  \begin{equation}
    \label{eq:area_reduction}
    \omega(m_i,l_{j,j+1}+l_{i,i+1}) \ge 0
  \end{equation}
  where $j$ is any opposing vertex for $m_i$.
\end{corollary}
We will now rewrite \eqref{eq:area_reduction} in a way that is affinely invariant. In other words we will consider $L$ as an affine plane instead of as
a 2-dimensional vector space. Since an affine plane does not have a fixed origin, the origin becomes part of the data. Then \eqref{eq:area_reduction}
becomes
\[
      \omega(m_i,l_{j,j+1}+l_{i,i+1}-2o) \ge 0
\]  
where $o$ is the origin. This can be rewritten as
\begin{equation}
  \label{eq:half_plane}
  \omega(m_i,o-l_{i,i+1})\le \frac{1}{2}\omega(m_i,l_{j,j+1}-l_{i,i+1}),
\end{equation}
which gives an equation for the origin.

The line $H_i$ (see Definition \ref{def:opposing}) has equation
\[
  \omega(m_i,x-l_{i,i+1})=\omega(m_i,l_{j,j+1}-l_{i,i+1})
\]
and the corresponding half plane $H^+$ containing $l_{i,i+1}$ is given by
\[
    \omega(m_i,x-l_{i,i+1})\le \omega(m_i,l_{j,j+1}-l_{i,i+1}).
  \]
  Comparing this with \eqref{eq:half_plane} we see that \eqref{eq:half_plane} holds if and only if $o$ is in the half plane ${}^{1/2} H_i^+$ containing $l_{i,i+1}$, where ${}^{1/2}H_i$ is
  the line parallel to $m_i$ half way between $l_{i,i+1}$ and $H_i$. This leads to the following definition.
  \begin{definition} \label{def:forbidden_region} The \emph{forbidden region} of $P$ is defined as $\bigcap_i {}^{1/2} H^+_i$. 
  \end{definition}
  \begin{corollary}
    \label{cor:forbidden}
    No quiver mutation reduces the total rank of $\EE$ if and only if $o$ is in the forbidden region of $P$.
  \end{corollary}
  \begin{remark} \label{rem:straight}
    Clearly the notions we have defined above (such as the forbidden region) make sense for arbitrary convex polygons, and we will use them as such
    without further comment. Note also that the forbidden region of a polygon $P$ only depends on $|P|$ (see \S\ref{sec:notation}).
    \end{remark} 
    In Figure \ref{fig:forbidden} we illustrate the forbidden region for a polygon with respectively three and four non-straight vertices.
    \begin{definition} \label{def:admissible}
We say that $l_{j,j+1}$ is \emph{admissible} if it is an opposing vertex of some $m_i$.
In that case we will also call the  edge $m_i$ an \emph{opposing edge} for $l_{j,j+1}$.  
\end{definition}
The significance of admissibility lies in the following easy result which we leave as an exercise for the reader.
\begin{lemma}
  \label{lem:nonadm_bigger}
  Deleting  (see \S\ref{sec:notation}) a non-admissible vertex from a polygon enlarges the forbidden region. 
\end{lemma}
In Figure \ref{fig:forbidden}, $l_{3,0}$ is the only non-admissible vertex in the right polygon. One checks
that deleting it indeed enlarges the forbidden region.
\begin{figure}[h!]
\caption{}\label{fig:areachange}

\begin{minipage}{0.48\textwidth}
\centering
\begin{tikzpicture}[scale=1.05, line join=round, line cap=round]
  \definecolor{orig}{RGB}{0,92,184}
  \definecolor{mut}{RGB}{0,150,60}

  \colorlet{colP0}{yellow!55}
  \colorlet{colP1}{cyan!40}
  \colorlet{colT}{magenta!30}

  \path[use as bounding box] (-2.75,-2.55) rectangle (2.35,4.10);

  \coordinate (o)  at (0,0);

  \coordinate (l0) at ( 1.9281,-1.0401);
  \coordinate (l1) at ( 0.8151, 1.5339);
  \coordinate (l2) at (-0.1612, 2.2602);
  \coordinate (l3) at (-0.9429, 2.2595);
  \coordinate (l4) at (-1.4482, 1.5966);
  \coordinate (l5) at (-2.2779,-0.3910);
  \coordinate (l6) at (-0.9475,-1.7911);
  \coordinate (l7) at (-0.0065,-2.2083);

  \coordinate (lp0) at (1.7293, 2.0859);
  \coordinate (lp1) at (1.9079, 3.5097);
  \coordinate (lp2) at (1.5398, 3.7587);

  \fill[colP0, opacity=0.30] (l3)--(l4)--(l5)--(l6)--(l7)--cycle;
  \fill[colP1, opacity=0.30] (l0)--(l1)--(l2)--(l3)--cycle;
  \fill[colT,  opacity=0.30] (l7)--(l0)--(l3)--cycle;

  \node at (-1.55, 0.15) {$P_0$};
  \node at ( 0.50, 1.35) {$P_1$};   
  \node at ( 0.55,-0.35) {$T$};     

  \draw[orig, very thick]
    (l0)--(l1)--(l2)--(l3)--(l4)--(l5)--(l6)--(l7)--cycle;

  \draw[mut, very thick]
    (lp0)--(lp1)--(lp2)--(l3)--(l4)--(l5)--(l6)--(l7)--cycle;

  \draw[orig, ultra thick]  (l7) -- (l0);
  \draw[black, ultra thick] (l7) -- (l0);

  \draw[mut, ultra thick]   (lp2) -- (l3);
  \draw[black, ultra thick] (lp2) -- (l3);

  \fill[black] (o) circle (1.2pt);

  \fill[orig] (l0) circle (1.2pt);
  \fill[orig] (l1) circle (1.2pt);
  \fill[orig] (l7) circle (1.2pt);
  \fill[orig] (l3) circle (1.2pt);
  \fill[orig] (l4) circle (1.2pt);

  \fill[mut] (lp0) circle (1.2pt);
  \fill[mut] (lp2) circle (1.2pt);

  \node[above right, text=black] at (o) {$o$};

  \node[right, text=orig]        at (l0) {$l_{0,1}$};
  \node[above right, text=orig]  at (l1) {};
  \node[below right, text=orig]  at (l7) {$l_{n-1,0}$};

  \node[left, text=orig]         at (l3) {$l_{j,j+1}$};
  \node[below left, text=orig]   at (l4) {$l_{j+1,j+2}$};

  \node[right, text=mut]         at (lp0) {$l_{0,1}'$};
  \node[above right, text=mut]   at (lp2) {$l_{j-1,j}'$};
\end{tikzpicture}
\end{minipage}
\hfill
\begin{minipage}{0.48\textwidth}
\centering
\begin{tikzpicture}[scale=1.05, line join=round, line cap=round]
  \definecolor{orig}{RGB}{0,92,184}
  \definecolor{mut}{RGB}{0,150,60}

  \colorlet{colP0}{yellow!55}
  \colorlet{colP1}{cyan!40}
  \colorlet{colT}{magenta!30}

  \path[use as bounding box] (-2.75,-2.55) rectangle (2.35,4.10);

  \coordinate (o)  at (0,0);

  \coordinate (l0) at ( 1.9281,-1.0401);
  \coordinate (l1) at ( 0.8151, 1.5339);
  \coordinate (l2) at (-0.1612, 2.2602);
  \coordinate (l3) at (-0.9429, 2.2595);
  \coordinate (l4) at (-1.4482, 1.5966);
  \coordinate (l5) at (-2.2779,-0.3910);
  \coordinate (l6) at (-0.9475,-1.7911);
  \coordinate (l7) at (-0.0065,-2.2083);

  \coordinate (lp0) at (1.7293, 2.0859);
  \coordinate (lp1) at (1.9079, 3.5097);
  \coordinate (lp2) at (1.5398, 3.7587);

  \fill[colP0, opacity=0.30] (l3)--(l4)--(l5)--(l6)--(l7)--cycle;
  \fill[colT,  opacity=0.30] (l7)--(lp2)--(l3)--cycle;
  \fill[colP1, opacity=0.30] (l7)--(lp0)--(lp1)--(lp2)--cycle;

  \node at (-1.55, 0.15) {$P_0$};
  \node at ( 0.25, 1.10) {$T'$};
  \node at ( 1.25, 1.75) {$P_1'$}; 

  \draw[orig, very thick]
    (l0)--(l1)--(l2)--(l3)--(l4)--(l5)--(l6)--(l7)--cycle;

  \draw[mut, very thick]
    (lp0)--(lp1)--(lp2)--(l3)--(l4)--(l5)--(l6)--(l7)--cycle;

  \draw[orig, ultra thick]  (l7) -- (l0);
  \draw[black, ultra thick] (l7) -- (l0);

  \draw[mut, ultra thick]   (lp2) -- (l3);
  \draw[black, ultra thick] (lp2) -- (l3);

  \fill[black] (o) circle (1.2pt);

  \fill[orig] (l0) circle (1.2pt);
  \fill[orig] (l1) circle (1.2pt);
  \fill[orig] (l7) circle (1.2pt);
  \fill[orig] (l3) circle (1.2pt);
  \fill[orig] (l4) circle (1.2pt);

  \fill[mut] (lp0) circle (1.2pt);
  \fill[mut] (lp2) circle (1.2pt);

  \node[above right, text=black] at (o) {$o$};

  \node[right, text=orig]        at (l0) {$l_{0,1}$};
  \node[above right, text=orig]  at (l1) {};
  \node[below right, text=orig]  at (l7) {$l_{n-1,0}$};

  \node[left, text=orig]         at (l3) {$l_{j,j+1}$};
  \node[below left, text=orig]   at (l4) {$l_{j+1,j+2}$};

  \node[right, text=mut]         at (lp0) {$l_{0,1}'$};
  \node[above right, text=mut]   at (lp2) {$l_{j-1,j}'$};
\end{tikzpicture}
\end{minipage}

\end{figure}

\begin{figure}
\centering
\caption{}\label{fig:forbidden}
\begin{minipage}{0.48\textwidth}
\centering
\begin{tikzpicture}[scale=0.95,line cap=round,line join=round]

\coordinate (X) at (0.3,0.2);    
\coordinate (Y) at (2.8,4.1);    
\coordinate (Z) at (7.0,3.7);    

\coordinate (Mxy) at ($(X)!0.5!(Y)$);
\coordinate (Myz) at ($(Y)!0.5!(Z)$);
\coordinate (Mzx) at ($(Z)!0.5!(X)$);

\draw[black,very thick] (X)--(Y)--(Z)--cycle;

\fill[red,opacity=0.35] (Mxy)--(Myz)--(Mzx)--cycle;
\draw[red,dotted, very thick] (Mxy)--(Myz)--(Mzx)--cycle;

\newcommand{\tick}[4]{%
  \draw[red,very thick]
    let \p1 = ($(#1)!#3!(#2)$),
        \p2 = ($(#2)-(#1)$),
        \n1 = {veclen(\x2,\y2)}
    in
      ({\x1 - \y2/\n1*#4},{\y1 + \x2/\n1*#4}) --
      ({\x1 + \y2/\n1*#4},{\y1 - \x2/\n1*#4});
}

\node[below left] at (X) {$l_{2,0}$};
\node at ($(Y)+(-0.15,0.2)$) {$l_{1,2}$};
\node[right]      at (Z) {$l_{0,1}$};

\end{tikzpicture}
\end{minipage}
\begin{minipage}{0.48\textwidth}
\centering
\begin{tikzpicture}[scale=0.95,line cap=round,line join=round]

\coordinate (A) at (0,0);        
\coordinate (B) at (3.6,-0.6);   
\coordinate (C) at (6.2,4.8);    
\coordinate (D) at (-1.0,2.3);   

\coordinate (P) at (1.30,0.85);
\coordinate (Q) at (3.4494,1.5963);
\coordinate (R) at (3.10,2.40);
\coordinate (S) at (2.1246,2.5626);

\begin{scope}
  \clip (A)--(B)--(C)--(D)--cycle;
  \draw[red,dotted,very thick] ($(P)!-0.45!(Q)$) -- ($(P)!1.45!(Q)$);
  \draw[red,dotted,very thick] ($(Q)!-0.55!(R)$) -- ($(Q)!1.55!(R)$);
  \draw[red,dotted,very thick] ($(R)!-0.55!(S)$) -- ($(R)!1.55!(S)$);
  \draw[red,dotted,very thick] ($(S)!-0.45!(P)$) -- ($(S)!1.45!(P)$);
\end{scope}

\fill[red,opacity=0.35] (P)--(Q)--(R)--(S)--cycle;

\draw[black,very thick] (A)--(B)--(C)--(D)--cycle;

\node[below left]  at (A) {$l_{3,0}$};
\node[below right] at (B) {$l_{0,1}$};
\node at ($(C)+(0.18,0.16)$) {$l_{1,2}$};
\node[left]        at (D) {$l_{2,3}$};

\end{tikzpicture}
\end{minipage}\hfill

\end{figure}
\section{Polygon shapes and quiver shapes}
\subsection{Setting}
\label{sec:quiver_setting}
In this section $X$ is a del Pezzo surface
and $\EE$ is a very strong exceptional collection on $X$.
After performing a rotation we may, and we will, assume that $\EE$ is the underlying exceptional collection of a block exceptional collection
without broken blocks,
which we denote by the same symbol.
\[
  \EE=((E_0^{(0)},\ldots, E_0^{(\alpha_0-1)}) , 
\ldots,
(E_{k-1}^{(0)},\ldots, E_{k-1}^{(\alpha_{k-1}-1)} ))
\]
We will discuss some restrictions on the shape of the quiver $Q$ and the polygon $P$ associated
to $\EE$. Recall that the shape of $P$ was described in Lemma \ref{lem:very_useful}.
\begin{definition} 
We say that $\EE$  or $Q$ is \emph{block-complete} if $Q_{\red}$ is a complete quiver.\end{definition} 

Our results will be mainly for the block-complete case,  but this is not really a restriction as we show in the next section.
\subsection{Reduction to the block-complete case}
\label{sec:block_complete}
 \begin{lemma} \label{lem:no_parallel}
   $\EE$ is block-complete if and only if $P$ has no parallel long edges.
\end{lemma}
\begin{proof} This follows immediately from Theorem \ref{thm:achievement}\eqref{it:thm2_5}.
\end{proof}
\begin{corollary}
  \label{cor:block_complete}
  If   $\EE$ is not block complete then it can be tranformed, via a sequence of quiver mutations, into a block-complete very strong exceptional collection
  which reduces the number of blocks and does not increase the total rank.
\end{corollary}
\begin{proof}
  If $\EE$ is not block complete then by Lemma \ref{lem:no_parallel}
  this implies it has parallel long edges.
  We have to prove that $P$ can be mutated into a polygon without parallel long edges and fewer long edges, and this without increasing total rank.

  If $e$, $e'$ are parallel long edges
 respectively divided into subedges $e^{(0)},\ldots,e^{(\alpha-1)}$, $e^{\prime (0)},\ldots,e^{\prime (\alpha'-1)}$ then by performing quiver mutations at $e^{(\alpha-1)},\ldots, e^{(0)}$ as described
 in Proposition \ref{prop:lm:clustergeneralised} (see Figure \ref{fig:clustermut} for an illustration) we see that $e$ shrinks to a point and $e'$ is replaced by a long edge divided into $\alpha+\alpha'$ edges. In other words
 the resulting polygon has fewer long edges. Moreover it follows from Lemma \ref{lem:area_reduction} that the effect on the total rank of this procedure depends on the
 position of the origin with respect to the line half way between $e$ and $e'$. It follows that if the procedure increases the total rank then executing the same procedure but using $e'$ instead of $e$ will reduce the total rank. Hence we may assune that the prodcedure does not increase total rank.
 By repeating the procedure we end up with a polygon without parallel edges without increasing the total rank.
\end{proof}
\begin{remark} On the level of $\EE$, the sequence of quiver mutations that eliminates parallel edges $P$ is a suitable block quiver mutation as explained in Remark \ref{rem:block_quiver}. So the sequence of quiver mutations we have constructed in the proof of Corollary \ref{cor:block_complete} can be written as a sequence of block quiver mutations.
\end{remark}
\subsection{Quadratic quiver relations}
\begin{lemma} Assume $E_0,E_1,E_2,E_3\in \Dscr(Y)$ for a smooth projective surface $Y$. Let $\chi^-$ be the anti-symmetrization of $\chi$. Then
\begin{equation}
\label{eq:plucker}
\chi^-(E_0,E_1)\chi^-(E_2,E_3)-\chi^-(E_0,E_2)\chi^-(E_1,E_3)+\chi^-(E_0,E_3)\chi^-(E_1,E_2)=0
\end{equation}
\end{lemma}
\begin{proof} Let $s$ be the action of the Serre functor on $K_0(Y)$. It is easy to see that $\ker (s-1)$ is contained in the radical of $\chi^-(-,-)$.
By \cite[Proposition B]{dTdVVdB2}, $\rk(s-1)\le 2$. Hence $\chi^-(-,-)$ factors through a vector space of dimension $\le 2$. Then \eqref{eq:plucker} follows from the Pl{\"u}cker identity.
\end{proof}
\begin{corollary} Assume $(E_0,E_1,E_2,E_3)$ is an exceptional collection in $\Dscr^b(Y)$ where $Y$ is as above.  Then
\begin{equation}
\label{eq:plucker1}
\chi(E_0,E_1)\chi(E_2,E_3)-\chi(E_0,E_2)\chi(E_1,E_3)+\chi(E_0,E_3)\chi(E_1,E_2)=0
\end{equation}
\end{corollary}
\begin{proof} If $(E,F)$ is an exceptional pair then $\chi^-(E,F)=\chi(E,F)$.
\end{proof}
\begin{corollary} Assume that $\EE=(E_0,\ldots,E_{n-1})$ is a very strong exceptional collection on $X$. 
Let $c$ be the incidence matrix of the corresponding quiver $Q$, i.e. $c_{ij}=|Q(i,j)|$ if there are arrows $i \r j$ and $c_{ij}=-|Q(j,i)|$ if there are arrows $j\r i$.
Then for all 4-tuples of vertices $(i_0,i_1,i_2,i_3)$ of $Q$ we have
\begin{equation}
\label{eq:plucker2} 
c_{i_0i_1}c_{i_2i_3}-c_{i_0i_2}c_{i_1i_3}+c_{i_0i_4}c_{i_1i_2}=0
\end{equation}
\end{corollary}
\begin{proof} This follows by applying \eqref{eq:plucker1} to the dual exceptional collection of $\EE$.
\end{proof}
\begin{remark}
 \label{rem:impossible}
For use below we make this a bit more concrete. Assume a subquiver of $Q$ spanned by 4 vertices  is as in Figure \ref{fig:square}, with a negative multiplicity indicating that the arrow goes
in the opposite direction. 
\begin{figure}
\centering
\caption{}\label{fig:square}
\begin{tikzpicture}[scale=1.0,line cap=round,line join=round,>=stealth]

  \coordinate (TL) at (0,3);
  \coordinate (TR) at (5,3);
  \coordinate (BR) at (5,0);
  \coordinate (BL) at (0,0);

  \draw[black,very thick] (TL)--(TR)--(BR)--(BL)--cycle;
  \draw[black,very thick] (TL)--(BR);
  \draw[black,very thick] (BL)--(TR);

  \fill (TL) circle (3pt);
  \fill (TR) circle (3pt);
  \fill (BR) circle (3pt);
  \fill (BL) circle (3pt);

  \draw[->,very thick] ($(BL)!0.52!(TL)$) -- ($(BL)!0.68!(TL)$); 
  \draw[->,very thick] ($(TR)!0.48!(BR)$) -- ($(TR)!0.64!(BR)$); 
  \draw[->,very thick] ($(TL)!0.40!(TR)$) -- ($(TL)!0.56!(TR)$); 
  \draw[->,very thick] ($(BR)!0.40!(BL)$) -- ($(BR)!0.56!(BL)$); 
  \draw[->,very thick] ($(BL)!0.28!(TR)$) -- ($(BL)!0.40!(TR)$); 
  \draw[->,very thick] ($(TL)!0.32!(BR)$) -- ($(TL)!0.44!(BR)$); 

  \node[left]  at ($(BL)!0.50!(TL)+(-0.25,0)$) {$a$};
  \node[right] at ($(TR)!0.50!(BR)+(0.2,0)$) {$b$};
  \node[above] at ($(TL)!0.65!(TR)+(0,0.12)$) {$c$};
  \node[below] at ($(BL)!0.35!(BR)+(0,-0.15)$) {$d$};
  \node[right] at ($(BL)!0.60!(TR)+(0.3,0.12)$) {$e$};
  \node[right] at ($(TL)!0.62!(BR)+(0.2,-0.10)$) {$f$};

\end{tikzpicture}
\end{figure}

Then \eqref{eq:plucker2} becomes
\begin{equation}
  \label{eq:square2}
ab-ef-cd=0.
\end{equation}
This implies e.g. that if $a,b,c,d,e,f \geq 1$ then $a=b=1$ is not possible. Note that it does not matter how we order the vertices,  we always get the same equation. 

\end{remark}
\subsection{Quiver shapes in the block-complete case}
\begin{lemma} \label{lem:tool4}
  Let $(E,(F^{(0)},\ldots, F^{(\alpha-1)}),G)$, $\alpha\ge 1$ be a strong\footnote{According to Convention \ref{conv:terminology} this means that the underlying exceptional
    collection is \emph{strong}, i.e.\ there are no higher $\Ext$'s.}
  block exceptional collection of vector bundles on $X$ (in particular $r(F^{(0)})=\cdots=r(F^{(\alpha-1)})$) such that the 
quiver corresponding to $\End_X(E\oplus F^{(0)}\oplus\cdots\oplus F^{(\alpha-1)}\oplus G)$ is given by
\begin{equation}
  \label{eq:quiver_shape}
\begin{tikzcd}
& F^{(0)} \arrow[rddd, "b"]&\\
& \vdots&\\
& F^{(\alpha-1)} \arrow[rd, "b"]&\\
E\arrow[rr,"c"']\arrow[ru,"a"]\arrow[ruuu,"a"]& & G
\end{tikzcd}
\end{equation}
where $a,b,c$ are the multiplicities of the corresponding arrows. Assume that
${a,b,c\ge 1}$.
Then we have
\begin{equation}
\label{eq:max_inequality}
r(F^{(i)})\le \max(r(E),r(G))
\end{equation}
and moreover equality happens if and only if all of the following conditions hold
\begin{equation}
\label{eq:equality}
\begin{gathered}
r(E)=r(G)\\
\alpha=1\\
a=1\text{ or }b=1\\
c=1
\end{gathered}
\end{equation}
\end{lemma}
\begin{proof} From \eqref{eq:chi_slope} we have the identity
\[
\chi(E,G)r(F^{(i)})=\chi(E,F^{(i)})r(G)+\chi(F^{(i)},G)r(E)
\]
where
\begin{align*}
\chi(E,F^{(i)})&=a\\
\chi(F^{(i)},G)&=b\\
\chi(E,G)&=\alpha ab+c
\end{align*}
Hence
\begin{equation}
\label{eq:r_identity}
r(F^{(i)})=\frac{a}{\alpha ab+c}r(G)+\frac{b}{\alpha ab+c}r(E)
\end{equation}
We have
\begin{align*}
\frac{a}{\alpha ab+c}+\frac{b}{\alpha ab+c}&=\frac{a+b}{\alpha ab+c}\\
&=1-\frac{\alpha ab+c-a-b}{\alpha ab+c}\\
&=1-\frac{(\alpha-1)ab+(a-1)(b-1)+c-1}{\alpha ab+c}\\
&\le 1
\end{align*}
which implies \eqref{eq:max_inequality}. By revisiting the
computation
it is easy to see that equality in \eqref{eq:max_inequality}
can only happen under the conditions listed in \eqref{eq:equality}.
\end{proof}
We now revert to our general setting introduced in \S\ref{sec:quiver_setting}.   All indices are taken $\mod k$.
\begin{definition}
  A long edge in $P$ is \emph{narrowing} if the sum of its interior angles with the adjacent edges is $<\pi$.
\end{definition}
\begin{lemma} \label{lem:narrowing}
  The polygon $P$ cannot have two non-adjacent narrowing long edges.
\end{lemma}
\begin{proof}
  Let the long edges be $e_0,\cdots,e_{k-1}$. Then we have
  \begin{equation}
    \label{eq:exterior_angle}
    \sum_{i=0}^{k-1}\angle(e_i,e_{i+1})=2\pi
  \end{equation}
  (here $\angle(e_i,e_{i+1})$ is the exterior angle from $e_i$ to $e_{i+1}$).
  On the other hand if $e_j$ is a narrowing long edge then $\angle(e_{j-1},e_j)+\angle(e_j,e_{j+1})>\pi$. This finishes the proof.
\end{proof}
The following lemma is absolutely crucial for our classification results in \S\ref{sec:minimal}.
\begin{lemma}
  \label{lem:one_narrowing}
    Assume $\EE$ is block-complete. Then $P$ has at least one narrowing long edge.
\end{lemma}
\begin{proof}
  Assume first that there is no narrowing long edge.  By Theorem \ref{thm:achievement}\eqref{it:thm2_5}  
  this implies that the subquiver spanned by all $E_{i-1}^{(\alpha_{i-1}-1)},E_i^{(0)},\ldots,E_i^{(\alpha_i-1)}, E_{i+1}^{(0)}$ is as in \eqref{eq:quiver_shape}.

  By Lemma \ref{lem:tool4} this implies that
  $r(E_i^{(?)})\le \max (r(E_{i-1}^{(?)}),r(E_{i+1}^{(?)}))$ for all $i$ (indices $\mod k$), with equality if and only if $r(E_{i-1}^{(?)})=r(E_{i}^{(?)})=r(E_{i+1}^{(?)})$.
  This quickly implies that all $r(E_i^{(?)})$ are equal. 
  Assume there is an $i$ such that $|Q(i,i+1)|\neq 1$. Then again by Lemma \ref{lem:tool4} $|Q(i-1,i)|=1$, $|Q(i+1,i+2)|=1$, $|Q(i-1,i+1)|=1$, $|Q(i,i+2)|=1$.
  This is like in Figure \ref{fig:square} except that we do not know the sign of $d$. However \eqref{eq:square2} yields $d=0$ which contradicts block-completeness.
  Hence for all $i$ we have $|Q(i,i+1)|=1$ and also $|Q(i,i+2)|=1$ (again using Lemma \ref{lem:tool4}). But this is also impossible, for exactly the same reason.
This contradiction finishes the proof.
\end{proof}
  \begin{corollary}
    \label{cor:quiver_pos}
    Assume $\EE$ is block-complete.
    Then exactly one of the following is true.
  \begin{enumerate}
  \item \label{it:cor_quiver_pos1} $P$ has a single narrowing long edge.
  \item \label{it:cor_quiver_pos2} $P$ has exactly two narrowing long edges. Moreover these long edges are adjacent.
    \item $P$ is a triangle (i.e.\ it has exactly 3 long edges and they are all narrowing).
  \end{enumerate}
\end{corollary}
\begin{proof}
By Lemma \ref{lem:one_narrowing} $P$ has at least on narrowing long edge.
  Using Lemma \ref{lem:narrowing}  we deduce that this yields the possibilities listed in the corollary.
\end{proof}
\begin{lemma}
  \label{lem:three_adm}
  Assume $\EE$ is block complete and let $e$ be a narrowing long edge of $P$ with adjacent vertices $u$, $v$. Then $u$, $v$ are admissible (see Definition \ref{def:admissible}) and there is exactly one
  other admissible vertex. In particular the total number of admissible vertices is three.
\end{lemma}
\begin{proof} Note that since $P$ has no parallel edges by Lemma \ref{lem:no_parallel} a straight vertex can never be admissible. Also for the same reason the opposing
  vertex to an edge is unique.

  Let the vertices of $P$ be numbered as $l_{-1,0},\ldots, l_{\alpha_0-1,\alpha_0},\ldots $ with $u=l_{-1,0}$, $v=l_{\alpha_0-1,\alpha_0}$ (indices modulo $n:=\sum_i \alpha_i$).
  We first claim that if $[l_{i-1,i},l_{i,i+1}]$ is an edge not on $e$ then  the opposing vertex is either $u$ or $v$.
  Recall that by Definition \ref{def:admissible} the opposing vertex $l_{j,j+1}$ is characterised by the fact that $\delta_i(x) = \omega(m_i,x-l_{i,i+1})$ takes maximal
  value for $x=l_{j,j+1}$. By considering $x=l_{j-1,j}$ and $x=l_{j+1,j+2}$ we see that this implies $\omega(m_i,m_j)>0$, $\omega(m_i,m_{j+1})<0$. By \eqref{eq:exterior_angle}
  we have
  \begin{equation}
    \label{eq:positivity}
    \omega(m_p,m_q)\ge 0\quad \text{for $m_p$ not on $e$ and $q=p+1,\ldots, -1$}.
  \end{equation}
  Hence $\omega(m_i,m_{j+1})<0$ implies $j+1=0$ (since $l_{j,j+1}$ must be a straight vertex) or $j+1=\alpha_0,\ldots,i-1$. But if $j+1=\alpha_0+1,\ldots, i-1$ then by \eqref{eq:positivity} $\omega(m_j,m_i)\ge 0$, contradicting $\omega(m_i,m_j)>0$.

  The opposing vertex for $m_{-1}$ cannot be the adjacent vertex $l_{-1,0}$. Hence it must be $v=l_{\alpha_0-1,\alpha_0}$. So $v$ is admissible. With a similar argument we see that $u$ is admissible.

  We have determined that the opposing vertex for all edges except those on $e$ is either $u$ or $v$. On the other hand the opposing vertex for an edge on $e$ cannot be $u$ or $v$. Hence it is the third admissible vertex, whose existence is asserted in the statement of the lemma. This finishes the proof. 
  \end{proof}
\begin{example}
  The two possible polygon shapes and the corresponding reduced quiver shapes for $k=5$ (computed via Theorem \ref{thm:achievement}\eqref{it:thm2_5}) are given in Figure
  \ref{fig:red_5_blocks}. The picture on the left corresponds to the case that there are two narrowing long edges while in the picture on the right there is only a single
  one (i.e.\ these cases correspond to the possibilities \ref{it:cor_quiver_pos2} and \ref{it:cor_quiver_pos1} in Corollary \ref{cor:quiver_pos} respectively). The meaning of
  the flat blue dots will be explained below.
\end{example}
The examples in Figure \ref{fig:red_5_blocks} suggest the following result which follows easily from Lemma \ref{lem:three_adm}.
\begin{lemma}
  In the block complete case
  for $k\ge 4$ there are $k-3$ possible reduced quiver shapes,  namely, a single possibility with two narrowing long edges and $k-4$ possibilities with
  one narrowing long edge. These $k-4$ possibilities are further distinguished by the location of the unique admissible vertex which is not on the narrowing long edge.
\end{lemma}
We also obtain the following result.
\begin{proposition}
  \label{prop:herzog}
  Assume $\EE$ is block complete.
  Draw the reduced quiver $Q_{\red}$ inside $P$ as in Figure \ref{fig:red_5_blocks}. Then it is possible to find a polygon $\Sigma$ bounded by arrows of $Q_{\red}$ so that
  all arrows in $Q_{\red}$ travel in the counter clockwise direction around $\Sigma$. Moreover $\Sigma$ is adjacent to a vertex of $Q_{\red}$.
\end{proposition}
In Figure \ref{fig:red_5_blocks} the distinguished polygons $\Sigma$ are marked by fat blue dots.
\begin{remark}
In \cite{Herzog} Herzog informally conjectured that a polygon like $\Sigma$ should exist\footnote{In the middle of \cite[p16]{Herzog} is it stated: ``That well split implies the existence of such a polygon is not trivial, and
  we leave it as a conjecture.''} and Proposition \ref{prop:herzog} in fact proves this conjecture in the block complete case. Note however that, as stated in the proposition,
the location of $\Sigma$ is more constrained than suggested in loc.\ cit.
In particular the quiver
in the top left of Figure 1 in loc.\ cit.\ cannot
be the reduced quiver of a rolled-up helix algebra on a del Pezzo surface when all the arrow multiplicities are non-zero.
\end{remark}
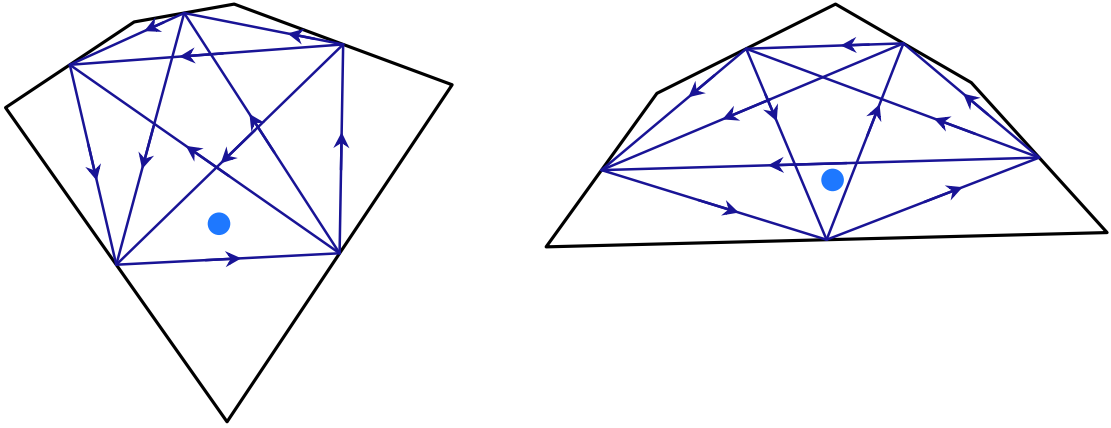
\begin{figure}[ht]
\centering
\caption{Reduced quivers for 5 blocks.}
\label{fig:red_5_blocks}
\vspace*{2mm}

\begin{minipage}[t]{0.48\textwidth}
  \strut\vspace*{-\baselineskip}\newline
\centering
\begin{tikzpicture}[scale=0.95,line cap=round,line join=round,>=Stealth]
  \definecolor{myblue}{RGB}{25,20,150}
  \definecolor{mydot}{RGB}{30,120,255}

  \newcommand{\arr}[4]{%
    \draw[myblue,line width=0.95pt,-{Stealth[length=6pt,width=6pt]}]
      ($(#1)!#3!(#2)$) -- ($(#1)!#4!(#2)$);
  }

  \coordinate (P0) at (0.0,2.5);
  \coordinate (P1) at (1.8,3.7);
  \coordinate (P2) at (3.2,3.95);
  \coordinate (P4) at (6.25,2.82);
  \coordinate (P5) at (3.10,-1.90);

  \draw[black,very thick] (P0)--(P1)--(P2)--(P4)--(P5)--cycle;

  \coordinate (A) at ($(P0)!0.5!(P1)$);
  \coordinate (B) at ($(P1)!0.5!(P2)$);
  \coordinate (C) at ($(P2)!0.5!(P4)$);
  \coordinate (D) at ($(P4)!0.5!(P5)$);
  \coordinate (E) at ($(P5)!0.5!(P0)$);

  \draw[myblue,line width=0.95pt] (A)--(B)--(C)--(D)--(E)--(A);
  \draw[myblue,line width=0.95pt] (C)--(E);
  \draw[myblue,line width=0.95pt] (D)--(B);
  \draw[myblue,line width=0.95pt] (B)--(E);
  \draw[myblue,line width=0.95pt] (D)--(A);
  \draw[myblue,line width=0.95pt] (C)--(A);

  \arr{E}{D}{0.40}{0.56} 
  \arr{D}{C}{0.40}{0.58} 
  \arr{C}{B}{0.18}{0.35} 
  \arr{B}{A}{0.18}{0.35} 
  \arr{A}{E}{0.40}{0.58} 

  \arr{C}{E}{0.39}{0.54} 
  \arr{D}{B}{0.42}{0.58} 
  \arr{D}{A}{0.41}{0.57} 
  \arr{C}{A}{0.44}{0.60} 
  \arr{B}{E}{0.45}{0.62} 

  \coordinate (Oleft) at ($(E)!0.46!(D)+(0,0.5)$);
  \fill[mydot] (Oleft) circle (4.5pt);
\end{tikzpicture}
\end{minipage}\hfill
\begin{minipage}[t]{0.48\textwidth}
  \strut\vspace*{-\baselineskip}\newline
\centering
\begin{tikzpicture}[scale=0.95,line cap=round,line join=round,>=Stealth]
  \definecolor{myblue}{RGB}{25,20,150}
  \definecolor{mydot}{RGB}{30,120,255}

  \newcommand{\arr}[4]{%
    \draw[myblue,line width=0.95pt,-{Stealth[length=6pt,width=6pt]}]
      ($(#1)!#3!(#2)$) -- ($(#1)!#4!(#2)$);
  }

  \coordinate (Q0) at (0.0,-1.45);   
  \coordinate (Q1) at (1.55,0.70);  
  \coordinate (Q2) at (4.05,1.95);   
  \coordinate (Q3) at (5.95,0.85);   
  \coordinate (Q4) at (7.85,-1.25);  

  \draw[black,very thick] (Q0)--(Q1)--(Q2)--(Q3)--(Q4)--cycle;

  \coordinate (A) at ($(Q0)!0.5!(Q1)$);
  \coordinate (B) at ($(Q1)!0.5!(Q2)$);
  \coordinate (C) at ($(Q2)!0.5!(Q3)$);
  \coordinate (D) at ($(Q3)!0.5!(Q4)$);
  \coordinate (E) at ($(Q4)!0.5!(Q0)$);

  \draw[myblue,line width=0.95pt] (A)--(B)--(C)--(D)--(E)--(A);
  \draw[myblue,line width=0.95pt] (D)--(A);
  \draw[myblue,line width=0.95pt] (E)--(C);
  \draw[myblue,line width=0.95pt] (C)--(A);
  \draw[myblue,line width=0.95pt] (B)--(E);
  \draw[myblue,line width=0.95pt] (D)--(B);

  \arr{B}{A}{0.22}{0.40} 
  \arr{C}{B}{0.22}{0.40} 
  \arr{D}{C}{0.38}{0.56} 
  \arr{E}{D}{0.46}{0.64} 
  \arr{A}{E}{0.43}{0.61} 

  \arr{D}{A}{0.44}{0.62} 
  \arr{E}{C}{0.53}{0.69} 
  \arr{C}{A}{0.44}{0.60} 
  \arr{B}{E}{0.24}{0.38} 
  \arr{D}{B}{0.22}{0.36} 

  \coordinate (Oright) at ($(E)!0.45!(C)+(-0.4,-0.4)$);
  \fill[mydot] (Oright) circle (4.5pt);
\end{tikzpicture}
\end{minipage}

\end{figure}

\section{All NCCRs are related by mutations}
\begin{definition}
  \label{def:equivalence}
  We say that two geometric helices (see \S\ref{sec:geom_helix}) on a del Pezzo surface are \emph{equivalent} if they can be transformed into each other by simultaneous shifts in the derived category, tensoring by line bundles, reordering of objects in orthogonal blocks and rotations (shifting the labels of objects). Two very strong exceptional collections are equivalent if the corresponding geometric helices are equivalent.
  \end{definition}
\begin{theorem}[See also \cite{Bousseau}]\label{thm:helixrelated} All geometric helices on a del Pezzo surface can be transformed into each other by quiver mutations, up to equivalence.
\end{theorem}

\begin{proof}
  According to Remark \ref{rem:helixshift} we may assume that a
  geometric helix consist of sheaves and hence that it is obtained from
  a very strong exceptional collection. By alternatingly applying Corollary \ref{cor:block_complete}
  and reduction of total rank by block quiver mutations,
every very strong exceptional collection can be reduced 
to  a block complete very strong exceptional collection whose total rank cannot be further reduced
by block quiver mutations. Section \S\ref{sec:minimal_collections} gives a complete list of the Gram
  matrices of such collections, up to 
  rotation. This list if obtained in \S\ref{sec:minimal}.

  Hence it remains to show that all the
  very strong exceptional  collections with these Gram matrices  can be transformed into each other by quiver mutations, up to equivalence.
  This explained in detail in \S\ref{sec:minimal_collections:intro}.
\end{proof} 

\begin{theorem}\label{thm:nccrsrelated} All basic graded NCCRs of $R_X =  \bigoplus_{k \leq 0} \Gamma(X, \omega_X^{\otimes k})$ (alternatively, NCCRs of the completion $\widehat{R_X}$) for $X$ a del Pezzo surface are related by sequences of mutations (up to regrading in the graded case).  \end{theorem}

\begin{proof} By Theorem \ref{thm:nccrfromFEC} every basic graded NCCR of $R_X$ is obtained from a very strong exceptional collection on $X$ (equivalently, a geometric helix of sheaves). Moreover, mutations of NCCRs can be viewed as quiver mutations of underlying helices (see \S \ref{sec:mutnccr}). The result now follows from Theorem \ref{thm:helixrelated}
  since equivalent geometric helices have isomorphic NCCRs.
\end{proof}
\section{Minimal block-complete very strong exceptional collections}
\label{sec:minimal_collections}
\subsection{Introduction}
\label{sec:minimal_collections:intro}
Below we provide the data backing the proof of Theorem \ref{thm:helixrelated}.
The data is grouped by surface where $X_i$ is the surface obtained by blowing up $\PP^2$ in $i$ points.

First of all, for each surface we provide the list of Gram matrices of the block-complete very strong exceptional collections of minimal total rank, up to rotation.
Each Gram matrix has a unique label of the form $(b,n)$ where
$b$ is the number of blocks (which must be 3 or 4) and $n$ is an integer.
The block sizes and the ranks are
respectively denoted by $(\alpha_i)_i$ and $(r_i)_i$. The matrix $M_/$ is the reduced Gram matrix and $Q_/$ is the associated reduced quiver, which depends only on $M_/$
and $(\alpha_i)_i$ (see \S\ref{sec:blockquiver}). For each listed Gram matrix we provide an example of a very strong exceptional collection, appropriately broken into blocks, that realises it. Each object is written in the form\footnote{It follows from the Riemann-Roch formula \eqref{eq:chi_rr} that an exceptional class is uniquely determined by
  $[r(E),c_1(E)]$, provided $r(E)\neq 0$.} $[r(E),c_1(E)]$, where 
$c_1(E) \in \Pic(X)$ is given by coordinates in the standard basis
(see \S\ref{sec:symmetrygroup}). The fact that the collection is numerically exceptional and has the given Gram matrix can be verified using the formula \eqref{eq:riemann_roch}.
The fact that it is uniquely realised by an actual exceptional collection follows from Proposition \ref{prop:numericalcol}.
The fact that our list is complete in the 3-block case follows directly from \cite{KarpovNogin}. The fact that it is complete in the 4-block case was verified by computer
as explained in \S\ref{sec:4block} below. The fact that no higher block counts occur is explained in \S\ref{sec:5blocks} and \S\ref{sec:6ormoreblocks}.

\medskip

For each surface, in a section entitled ``Relations'', we provide sequences of quiver mutations, applied from left to right,
indexed in such a way that $i$  stands for the left quiver mutation
that starts with $\tilde{\sigma}_i$  (see Remark \ref{rem:left:quiver}),
that connect all the different Gram matrices, up to rotation. A label $(3,\ast)$ refers to an unspecified $3$-block Gram matrix for that surface. Note
that this may not be the Gram matrix of a minimal collection, in which case it is not in the list. However thanks to \cite{KarpovNogin} it can be further reduced to a minimal one.

For each surface we also provide a proof (see below), for one of the listed exceptional
collections, that the action of every element of the symmetry group (see
\S\ref{sec:symmetrygroup}) yields an exceptional collection that is reachable by quiver
mutations, up to equivalence (see Definition \ref{def:equivalence}).

\medskip

So, since for every surface all Gram matrices are connected by quiver mutations, and since for one of the Gram matrices all collections with that Gram matrix are also connected by quiver mutations, up to equivalence,
this concludes the proof that for each surface all collections are connected by quiver mutations, up to equivalence.

\medskip

The proof that every exceptional collection in the symmetry group orbit of a particular very strong exceptional collection $\EE$ is reachable by quiver mutations, up to equivalence,
involves what we call a \emph{certificate}.
A certificate is a set of pairs $(w, m)$ where $w$ is a sequence
of simple reflections indexed as in SageMath (see below) and
$m$ is a sequence of left quiver mutations indexed in the same way as above.
The pairs $(w, m)$ satisfy $w(\EE) \sim m(\EE)$ where all sequences (both simple reflections and quiver mutations) are applied from left to right. 
Moreover the list of Weyl group element represented by the $w$'s
generates the whole group.  

Our certificates are generated by computer search and it is possible, and even likely, that much shorter ones exist. In particular if
the symmetry group for a surface is trivial
 or if the action of every reflection yields an equivalent collection (see Definition \ref{def:equivalence}) then there is a particularly trivial certificate
and therefore we flag these cases separately.

\medskip

Finally, as explained above, in the certificate we index the simple reflections as in \cite{sagemath}.
We will clarify this now. If $(s_i)_i$ are the simple reflections associated with the simple roots as listed in Theorem \ref{thm:manin} and $(s'_i)_i$ are the simple reflections
used by SageMath then $s'_i=s_i$ except for the following modifications:
\[
\begin{cases}
   s'_1,s'_3=s_3,s_1 &\text{for $X_4$,}\\
s'_0, s'_1, s'_2, s'_3 = s_1, s_2, s_3, s_0&\text{for $X_5$,}\\
 s'_0, s'_1= s_1,s_0 &\text{for $X_n$ for $n\ge 6$}.
  \end{cases}
\]
\newcommand{\Qthree}[3]{%
\vcenter{\hbox{%
\begin{tikzpicture}[
  baseline=(current bounding box.center),
  x=0.88em,y=0.88em,
  line width=.6pt,
  >=Latex,
  qvtx/.style={font=\fontsize{4.6}{4.6}\selectfont\bfseries\sffamily,inner sep=0pt},
  qelab/.style={font=\scriptsize,inner sep=.35pt}
]
\coordinate (v0) at (0,1.25);
\coordinate (v1) at (-1.35,0);
\coordinate (v2) at (1.35,0);
\fill (v0) circle (1.05pt);
\fill (v1) circle (1.05pt);
\fill (v2) circle (1.05pt);
\node[qvtx,above=1.2pt] at (v0) {0};
\node[qvtx,left=1.2pt] at (v1) {1};
\node[qvtx,right=1.2pt] at (v2) {2};
\ifnum#1>0\relax
  \draw[-{Latex[length=2.1mm,width=1.45mm]},shorten >=0.5pt,shorten <=0.5pt]
    (v0) -- (v1) node[pos=.54,above,xshift=-1.4pt,qelab] {$#1$};
\else\ifnum#1<0\relax
  \draw[-{Latex[length=2.1mm,width=1.45mm]},shorten >=0.5pt,shorten <=0.5pt]
    (v1) -- (v0) node[pos=.54,above,xshift=-2.2pt,qelab] {$\number\numexpr-#1\relax$};
\fi\fi
\ifnum#2>0\relax
  \draw[-{Latex[length=2.1mm,width=1.45mm]},shorten >=0.5pt,shorten <=0.5pt]
    (v0) -- (v2) node[pos=.54,above,xshift=2.1pt,qelab] {$#2$};
\else\ifnum#2<0\relax
  \draw[-{Latex[length=2.1mm,width=1.45mm]},shorten >=0.5pt,shorten <=0.5pt]
    (v2) -- (v0) node[pos=.54,above,xshift=2.1pt,qelab] {$\number\numexpr-#2\relax$};
\fi\fi
\ifnum#3>0\relax
  \draw[-{Latex[length=2.1mm,width=1.45mm]},shorten >=0.5pt,shorten <=0.5pt]
    (v1) -- (v2) node[pos=.48,below,qelab] {$#3$};
\else\ifnum#3<0\relax
  \draw[-{Latex[length=2.1mm,width=1.45mm]},shorten >=0.5pt,shorten <=0.5pt]
    (v2) -- (v1) node[pos=.48,below,qelab] {$\number\numexpr-#3\relax$};
\fi\fi
\end{tikzpicture}%
}}%
}

\newcommand{\Qfour}[6]{%
\vcenter{\hbox{%
\begin{tikzpicture}[
  baseline=(current bounding box.center),
  x=0.9em,y=0.9em,
  line width=.6pt,
  >=Latex,
  qvtx/.style={font=\fontsize{4.6}{4.6}\selectfont\bfseries\sffamily,inner sep=0pt},
  qelab/.style={font=\scriptsize,inner sep=.35pt}
]
\coordinate (v0) at (-1.35,1.15);
\coordinate (v1) at (-1.35,-1.15);
\coordinate (v2) at (1.35,-1.15);
\coordinate (v3) at (1.35,1.15);
\fill (v0) circle (1.05pt);
\fill (v1) circle (1.05pt);
\fill (v2) circle (1.05pt);
\fill (v3) circle (1.05pt);
\node[qvtx,above left=1.0pt and 1.0pt] at (v0) {0};
\node[qvtx,below left=1.0pt and 1.0pt] at (v1) {1};
\node[qvtx,below right=1.0pt and 1.0pt] at (v2) {2};
\node[qvtx,above right=1.0pt and 1.0pt] at (v3) {3};
\ifnum#1>0\relax
  \draw[-{Latex[length=2.1mm,width=1.45mm]},shorten >=0.5pt,shorten <=0.5pt]
    (v0) -- (v1) node[pos=.52,left,qelab] {$#1$};
\else\ifnum#1<0\relax
  \draw[-{Latex[length=2.1mm,width=1.45mm]},shorten >=0.5pt,shorten <=0.5pt]
    (v1) -- (v0) node[pos=.52,left,qelab] {$\number\numexpr-#1\relax$};
\fi\fi
\ifnum#2>0\relax
  \draw[-{Latex[length=2.1mm,width=1.45mm]},shorten >=0.5pt,shorten <=0.5pt]
    (v0) -- (v2) node[pos=.24,below,xshift=-1.0pt,yshift=-0.1pt,qelab] {$#2$};
\else\ifnum#2<0\relax
  \draw[-{Latex[length=2.1mm,width=1.45mm]},shorten >=0.5pt,shorten <=0.5pt]
    (v2) -- (v0) node[pos=.24,below,xshift=-1.0pt,yshift=-0.1pt,qelab] {$\number\numexpr-#2\relax$};
\fi\fi
\ifnum#3>0\relax
  \draw[-{Latex[length=2.1mm,width=1.45mm]},shorten >=0.5pt,shorten <=0.5pt]
    (v0) -- (v3) node[pos=.52,above,yshift=1.0pt,qelab] {$#3$};
\else\ifnum#3<0\relax
  \draw[-{Latex[length=2.1mm,width=1.45mm]},shorten >=0.5pt,shorten <=0.5pt]
    (v3) -- (v0) node[pos=.52,above,yshift=1.0pt,qelab] {$\number\numexpr-#3\relax$};
\fi\fi
\ifnum#4>0\relax
  \draw[-{Latex[length=2.1mm,width=1.45mm]},shorten >=0.5pt,shorten <=0.5pt]
    (v1) -- (v2) node[pos=.52,below,qelab] {$#4$};
\else\ifnum#4<0\relax
  \draw[-{Latex[length=2.1mm,width=1.45mm]},shorten >=0.5pt,shorten <=0.5pt]
    (v2) -- (v1) node[pos=.52,below,qelab] {$\number\numexpr-#4\relax$};
\fi\fi
\ifnum#5>0\relax
  \draw[-{Latex[length=2.1mm,width=1.45mm]},shorten >=0.5pt,shorten <=0.5pt]
    (v1) -- (v3) node[pos=.24,below,xshift=2.3pt,yshift=1.1pt,qelab] {$#5$};
\else\ifnum#5<0\relax
  \draw[-{Latex[length=2.1mm,width=1.45mm]},shorten >=0.5pt,shorten <=0.5pt]
    (v3) -- (v1) node[pos=.24,below,xshift=2.3pt,yshift=1.1pt,qelab] {$\number\numexpr-#5\relax$};
\fi\fi
\ifnum#6>0\relax
  \draw[-{Latex[length=2.1mm,width=1.45mm]},shorten >=0.5pt,shorten <=0.5pt]
    (v2) -- (v3) node[pos=.52,right,qelab] {$#6$};
\else\ifnum#6<0\relax
  \draw[-{Latex[length=2.1mm,width=1.45mm]},shorten >=0.5pt,shorten <=0.5pt]
    (v3) -- (v2) node[pos=.52,right,qelab] {$\number\numexpr-#6\relax$};
\fi\fi
\end{tikzpicture}%
}}%
}

\subsection{Minimal collections for \!{\boldmath$\mathbb{P}^2$}}
\subsubsection{Label (3, 1)}
\[
\begin{aligned}
\alpha_0,\alpha_1,\alpha_2&=1, 1, 1\\*[-1mm]
r_0,r_1,r_2&=1, 1, 1\\*[-1mm]
\end{aligned}\quad
M_/=\left(\begin{array}{rrr}
1 & 3 & 6 \\
0 & 1 & 3 \\
0 & 0 & 1
\end{array}\right)\quad
Q_/=\Qthree{3}{-3}{3}
\]
{\tiny \raggedright Collection: \hl{[1, (0)]},\ \ \hl{[1, (1)]},\ \ \hl{[1, (2)]}

\bigskip

\noindent The full orbit is reachable: the symmetry group is trivial.

}

\subsection{Minimal collections for \!{\boldmath$\mathbb{P}^1\times \mathbb{P}^1$}}
\subsubsection{Label (3, 2)}
\[
\begin{aligned}
\alpha_0,\alpha_1,\alpha_2&=1, 2, 1\\*[-1mm]
r_0,r_1,r_2&=1, 1, 1\\*[-1mm]
\end{aligned}\quad
M_/=\left(\begin{array}{rrr}
1 & 2 & 4 \\
0 & 1 & 2 \\
0 & 0 & 1
\end{array}\right)\quad
Q_/=\Qthree{2}{-4}{2}
\]
{\tiny \raggedright Collection: \hl{[1, (0, 0)]},\ \ \hl{[1, (0, 1)], [1, (1, 0)]},\ \ \hl{[1, (1, 1)]}

\bigskip

\noindent The full orbit is reachable: all reflections give equivalent collections.

}

\subsection{Minimal collections for \!{\boldmath$X_1$}}
\subsubsection{Label (4, 1)}
\[
\begin{aligned}
\alpha_0,\alpha_1,\alpha_2,\alpha_3&=1, 1, 1, 1\\*[-1mm]
r_0,r_1,r_2,r_3&=1, 1, 1, 1\\*[-1mm]
\end{aligned}\quad
M_/=\left(\begin{array}{rrrr}
1 & 2 & 3 & 5 \\
0 & 1 & 1 & 3 \\
0 & 0 & 1 & 2 \\
0 & 0 & 0 & 1
\end{array}\right)\quad
Q_/=\Qfour{2}{1}{-3}{1}{1}{2}
\]
{\tiny \raggedright Collection: \hl{[1, (1, 0)]},\ \ \hl{[1, (2, -1)]},\ \ \hl{[1, (2, 0)]},\ \ \hl{[1, (3, -1)]}

\bigskip

\noindent The full orbit is reachable: the symmetry group is trivial.

}

\subsection{Minimal collections for \!{\boldmath$X_2$}}
\subsubsection{Label (4, 2)}
\[
\begin{aligned}
\alpha_0,\alpha_1,\alpha_2,\alpha_3&=1, 1, 1, 2\\*[-1mm]
r_0,r_1,r_2,r_3&=1, 1, 1, 1\\*[-1mm]
\end{aligned}\quad
M_/=\left(\begin{array}{rrrr}
1 & 3 & 4 & 5 \\
0 & 1 & 1 & 2 \\
0 & 0 & 1 & 1 \\
0 & 0 & 0 & 1
\end{array}\right)\quad
Q_/=\Qfour{3}{1}{-2}{1}{1}{1}
\]
{\tiny \raggedright Collection: \hl{[1, (-1, -1, 0)]},\ \ \hl{[1, (0, -1, 0)]},\ \ \hl{[1, (1, -2, -1)]},\ \ \hl{[1, (1, -1, -1)], [1, (1, -2, 0)]}

\bigskip

\noindent The full orbit is reachable: all reflections give equivalent collections.

}

\subsubsection{Label (4, 3)}
\[
\begin{aligned}
\alpha_0,\alpha_1,\alpha_2,\alpha_3&=2, 1, 1, 1\\*[-1mm]
r_0,r_1,r_2,r_3&=1, 1, 1, 1\\*[-1mm]
\end{aligned}\quad
M_/=\left(\begin{array}{rrrr}
1 & 1 & 2 & 5 \\
0 & 1 & 1 & 4 \\
0 & 0 & 1 & 3 \\
0 & 0 & 0 & 1
\end{array}\right)\quad
Q_/=\Qfour{1}{1}{-2}{1}{1}{3}
\]
{\tiny \raggedright Collection: \hl{[1, (-1, -1, 0)], [1, (-1, 0, -1)]},\ \ \hl{[1, (-1, 0, 0)]},\ \ \hl{[1, (0, -1, -1)]},\ \ \hl{[1, (1, -1, -1)]}

}

\subsubsection{Relations}
\[\begin{array}{cc}%
(4, 2)\rightarrow(4, 3) & [2]
\end{array}\]\subsection{Minimal collections for \!{\boldmath$X_3$}}
\subsubsection{Label (3, 3)}
\[
\begin{aligned}
\alpha_0,\alpha_1,\alpha_2&=1, 2, 3\\*[-1mm]
r_0,r_1,r_2&=1, 1, 1\\*[-1mm]
\end{aligned}\quad
M_/=\left(\begin{array}{rrr}
1 & 3 & 4 \\
0 & 1 & 1 \\
0 & 0 & 1
\end{array}\right)\quad
Q_/=\Qthree{3}{-2}{1}
\]
{\tiny \raggedright Collection: \hl{[1, (0, 0, 0, 0)]},\ \ \hl{[1, (1, 0, 0, 0)], [1, (2, -1, -1, -1)]},\ \ \hl{[1, (2, 0, -1, -1)], [1, (2, -1, -1, 0)], [1, (2, -1, 0, -1)]}

\bigskip

\noindent The full orbit is reachable: all reflections give equivalent collections.

}

\subsubsection{Label (3, 4)}
\[
\begin{aligned}
\alpha_0,\alpha_1,\alpha_2&=2, 1, 3\\*[-1mm]
r_0,r_1,r_2&=1, 1, 1\\*[-1mm]
\end{aligned}\quad
M_/=\left(\begin{array}{rrr}
1 & 3 & 5 \\
0 & 1 & 2 \\
0 & 0 & 1
\end{array}\right)\quad
Q_/=\Qthree{3}{-1}{2}
\]
{\tiny \raggedright Collection: \hl{[1, (-1, 0, 1, 1)], [1, (0, -1, 0, 0)]},\ \ \hl{[1, (1, -1, 0, 0)]},\ \ \hl{[1, (2, -1, -1, 0)], [1, (2, -1, 0, -1)], [1, (2, -2, 0, 0)]}

}

\subsubsection{Relations}
\[\begin{array}{cc}%
(3, 3)\rightarrow(3, 4) & [3, 4, 5]
\end{array}\]\subsection{Minimal collections for \!{\boldmath$X_4$}}
\subsubsection{Label (3, 5)}
\[
\begin{aligned}
\alpha_0,\alpha_1,\alpha_2&=1, 1, 5\\*[-1mm]
r_0,r_1,r_2&=2, 1, 1\\*[-1mm]
\end{aligned}\quad
M_/=\left(\begin{array}{rrr}
1 & 5 & 9 \\
0 & 1 & 2 \\
0 & 0 & 1
\end{array}\right)\quad
Q_/=\Qthree{5}{-1}{2}
\]
{\tiny \raggedright Collection: \hl{[2, (-1, -1, 1, 1, 1)]},\ \ \hl{[1, (1, -1, 0, 0, 0)]},\ \ \hl{[1, (2, -1, -1, 0, 0)], [1, (2, -1, 0, -1, 0)], [1, (2, -1, 0, 0, -1)], [1, (3, -2, -1, -1, -1)], [1, (2, -2, 0, 0, 0)]}

\bigskip

\noindent The full orbit is reachable: all reflections give equivalent collections.

}

\subsubsection{Label (3, 6)}
\[
\begin{aligned}
\alpha_0,\alpha_1,\alpha_2&=1, 1, 5\\*[-1mm]
r_0,r_1,r_2&=1, 2, 1\\*[-1mm]
\end{aligned}\quad
M_/=\left(\begin{array}{rrr}
1 & 5 & 3 \\
0 & 1 & 1 \\
0 & 0 & 1
\end{array}\right)\quad
Q_/=\Qthree{5}{-2}{1}
\]
{\tiny \raggedright Collection: \hl{[1, (1, 0, 0, 0, 0)]},\ \ \hl{[2, (5, -1, -1, -1, -1)]},\ \ \hl{[1, (3, -1, -1, -1, 0)], [1, (3, -1, -1, 0, -1)], [1, (3, -1, 0, -1, -1)], [1, (3, 0, -1, -1, -1)], [1, (2, 0, 0, 0, 0)]}

}

\subsubsection{Label (4, 4)}
\[
\begin{aligned}
\alpha_0,\alpha_1,\alpha_2,\alpha_3&=1, 1, 1, 4\\*[-1mm]
r_0,r_1,r_2,r_3&=2, 1, 1, 1\\*[-1mm]
\end{aligned}\quad
M_/=\left(\begin{array}{rrrr}
1 & 3 & 7 & 9 \\
0 & 1 & 2 & 3 \\
0 & 0 & 1 & 1 \\
0 & 0 & 0 & 1
\end{array}\right)\quad
Q_/=\Qfour{3}{1}{-1}{2}{1}{1}
\]
{\tiny \raggedright Collection: \hl{[2, (-2, 1, 1, 1, 2)]},\ \ \hl{[1, (0, 0, 0, 0, 1)]},\ \ \hl{[1, (1, 0, 0, 0, 0)]},\ \ \hl{[1, (2, -1, 0, -1, 0)], [1, (2, 0, -1, -1, 0)], [1, (1, 0, 0, 0, 1)], [1, (2, -1, -1, 0, 0)]}

}

\subsubsection{Label (4, 5)}
\[
\begin{aligned}
\alpha_0,\alpha_1,\alpha_2,\alpha_3&=4, 1, 1, 1\\*[-1mm]
r_0,r_1,r_2,r_3&=1, 1, 1, 2\\*[-1mm]
\end{aligned}\quad
M_/=\left(\begin{array}{rrrr}
1 & 1 & 3 & 9 \\
0 & 1 & 2 & 7 \\
0 & 0 & 1 & 3 \\
0 & 0 & 0 & 1
\end{array}\right)\quad
Q_/=\Qfour{1}{1}{-1}{2}{1}{3}
\]
{\tiny \raggedright Collection: \hl{[1, (0, 0, 0, 0, 1)], [1, (0, -1, 0, 1, 1)], [1, (0, 0, 0, 1, 0)], [1, (1, -1, -1, 0, 0)]},\ \ \hl{[1, (1, -1, 0, 0, 0)]},\ \ \hl{[1, (2, -1, -1, 0, 0)]},\ \ \hl{[2, (6, -3, -2, -1, -1)]}

}

\subsubsection{Label (4, 6)}
\[
\begin{aligned}
\alpha_0,\alpha_1,\alpha_2,\alpha_3&=2, 1, 1, 3\\*[-1mm]
r_0,r_1,r_2,r_3&=1, 1, 1, 1\\*[-1mm]
\end{aligned}\quad
M_/=\left(\begin{array}{rrrr}
1 & 2 & 3 & 4 \\
0 & 1 & 1 & 2 \\
0 & 0 & 1 & 1 \\
0 & 0 & 0 & 1
\end{array}\right)\quad
Q_/=\Qfour{2}{1}{-1}{1}{1}{1}
\]
{\tiny \raggedright Collection: \hl{[1, (-1, 0, 1, 1, 1)], [1, (0, 0, 0, 0, 0)]},\ \ \hl{[1, (1, -1, 0, 0, 0)]},\ \ \hl{[1, (1, 0, 0, 0, 0)]},\ \ \hl{[1, (2, -1, 0, -1, 0)], [1, (2, -1, 0, 0, -1)], [1, (2, -1, -1, 0, 0)]}

}

\subsubsection{Label (4, 7)}
\[
\begin{aligned}
\alpha_0,\alpha_1,\alpha_2,\alpha_3&=3, 1, 1, 2\\*[-1mm]
r_0,r_1,r_2,r_3&=1, 1, 1, 1\\*[-1mm]
\end{aligned}\quad
M_/=\left(\begin{array}{rrrr}
1 & 1 & 2 & 4 \\
0 & 1 & 1 & 3 \\
0 & 0 & 1 & 2 \\
0 & 0 & 0 & 1
\end{array}\right)\quad
Q_/=\Qfour{1}{1}{-1}{1}{1}{2}
\]
{\tiny \raggedright Collection: \hl{[1, (-1, 1, 1, 1, 1)], [1, (0, 0, 0, 0, 1)], [1, (0, 0, 0, 1, 0)]},\ \ \hl{[1, (0, 0, 0, 1, 1)]},\ \ \hl{[1, (1, 0, 0, 0, 0)]},\ \ \hl{[1, (2, 0, -1, 0, 0)], [1, (2, -1, 0, 0, 0)]}

}

\subsubsection{Relations}
\[\begin{array}{cc}%
(3, 5)\rightarrow(3, 6) & [2, 3, 4, 5, 6] \\
(4, 4)\rightarrow(3, *) & [2, 1] \\
(4, 5)\rightarrow(3, *) & [4, 6] \\
(4, 6)\rightarrow(3, *) & [3, 1, 2] \\
(4, 7)\rightarrow(3, *) & [3, 5, 6]
\end{array}\]\subsection{Minimal collections for \!{\boldmath$X_5$}}
\subsubsection{Label (3, 7)}
\[
\begin{aligned}
\alpha_0,\alpha_1,\alpha_2&=2, 2, 4\\*[-1mm]
r_0,r_1,r_2&=1, 1, 1\\*[-1mm]
\end{aligned}\quad
M_/=\left(\begin{array}{rrr}
1 & 2 & 3 \\
0 & 1 & 1 \\
0 & 0 & 1
\end{array}\right)\quad
Q_/=\Qthree{2}{-1}{1}
\]
{\tiny \raggedright Collection: \hl{[1, (0, 0, 0, 0, 0, 1)], [1, (0, 0, 0, 0, 1, 0)]},\ \ \hl{[1, (1, 0, 0, 0, 0, 0)], [1, (2, -1, -1, -1, 0, 0)]},\ \ \hl{[1, (2, -1, 0, -1, 0, 0)], [1, (2, 0, -1, -1, 0, 0)], [1, (2, -1, -1, 0, 0, 0)], [1, (3, -1, -1, -1, -1, -1)]}

\bigskip

\noindent The full orbit is reachable: certificate = [([4], []), ([2, 3, 1, 2, 4, 2, 3, 1, 2], []), ([1, 2, 3, 0, 1, 2, 4, 2, 3, 1, 2, 0, 1], []), ([2, 3, 0, 1, 2, 4, 2, 3, 1, 2, 1, 0], []), ([3], []), ([4, 1, 2, 3, 0, 1, 2, 4, 0, 1, 2, 3, 0], [4, 5, 8, 1])]

}

\subsubsection{Label (4, 8)}
\[
\begin{aligned}
\alpha_0,\alpha_1,\alpha_2,\alpha_3&=1, 1, 1, 5\\*[-1mm]
r_0,r_1,r_2,r_3&=2, 1, 1, 1\\*[-1mm]
\end{aligned}\quad
M_/=\left(\begin{array}{rrrr}
1 & 3 & 5 & 7 \\
0 & 1 & 1 & 2 \\
0 & 0 & 1 & 1 \\
0 & 0 & 0 & 1
\end{array}\right)\quad
Q_/=\Qfour{3}{2}{-1}{1}{1}{1}
\]
{\tiny \raggedright Collection: \hl{[2, (-1, 0, 1, 1, 1, 1)]},\ \ \hl{[1, (1, -1, 0, 0, 0, 0)]},\ \ \hl{[1, (1, 0, 0, 0, 0, 0)]},\ \ \hl{[1, (2, -1, 0, 0, -1, 0)], [1, (2, -1, 0, 0, 0, -1)], [1, (3, -1, -1, -1, -1, -1)], [1, (2, -1, -1, 0, 0, 0)], [1, (2, -1, 0, -1, 0, 0)]}

}

\subsubsection{Label (4, 9)}
\[
\begin{aligned}
\alpha_0,\alpha_1,\alpha_2,\alpha_3&=5, 1, 1, 1\\*[-1mm]
r_0,r_1,r_2,r_3&=1, 1, 1, 2\\*[-1mm]
\end{aligned}\quad
M_/=\left(\begin{array}{rrrr}
1 & 1 & 2 & 7 \\
0 & 1 & 1 & 5 \\
0 & 0 & 1 & 3 \\
0 & 0 & 0 & 1
\end{array}\right)\quad
Q_/=\Qfour{1}{1}{-1}{1}{2}{3}
\]
{\tiny \raggedright Collection: \hl{[1, (-1, 0, 1, 1, 1, 1)], [1, (0, 0, 0, 0, 0, 1)], [1, (0, 0, 0, 0, 1, 0)], [1, (0, 0, 0, 1, 0, 0)], [1, (0, 0, 1, 0, 0, 0)]},\ \ \hl{[1, (1, -1, 0, 0, 0, 0)]},\ \ \hl{[1, (1, 0, 0, 0, 0, 0)]},\ \ \hl{[2, (5, -2, -1, -1, -1, -1)]}

}

\subsubsection{Relations}
\[\begin{array}{cc}%
(4, 8)\rightarrow(3, *) & [3, 4, 3, 1, 1] \\
(4, 9)\rightarrow(3, *) & [1, 1, 1, 2, 1, 2]
\end{array}\]\subsection{Minimal collections for \!{\boldmath$X_6$}}
\subsubsection{Label (3, 8)}
\[
\begin{aligned}
\alpha_0,\alpha_1,\alpha_2&=3, 3, 3\\*[-1mm]
r_0,r_1,r_2&=1, 1, 1\\*[-1mm]
\end{aligned}\quad
M_/=\left(\begin{array}{rrr}
1 & 1 & 2 \\
0 & 1 & 1 \\
0 & 0 & 1
\end{array}\right)\quad
Q_/=\Qthree{1}{-1}{1}
\]
{\tiny \raggedright Collection: \hl{[1, (1, 0, 0, 0, 0, 0, 1)], [1, (1, 0, 0, 0, 0, 1, 0)], [1, (1, 0, 0, 0, 1, 0, 0)]},\ \ \hl{[1, (2, -1, 0, 0, 0, 0, 0)], [1, (2, 0, -1, 0, 0, 0, 0)], [1, (2, 0, 0, -1, 0, 0, 0)]},\ \ \hl{[1, (4, -1, -1, -1, -1, -1, -1)], [1, (2, 0, 0, 0, 0, 0, 0)], [1, (3, -1, -1, -1, 0, 0, 0)]}

\bigskip

\noindent The full orbit is reachable: certificate = [([5, 4, 3, 0, 2, 1, 3, 4, 5, 0, 2, 3, 4, 1, 3], []), ([0], []), ([5], []), ([5, 3, 4, 2, 3, 0, 2, 1, 3, 4, 5, 4, 3, 2, 1, 3, 4, 0, 2, 3], []), ([2], []), ([1], []), ([4, 5, 4], []), ([3, 4, 5, 1, 3, 4, 2, 3, 0, 2, 1, 3, 4, 5, 2, 3, 0, 2, 1, 3, 4, 3, 2, 1, 3, 2], [9, 3, 7, 7])]

}

\subsubsection{Label (3, 9)}
\[
\begin{aligned}
\alpha_0,\alpha_1,\alpha_2&=2, 1, 6\\*[-1mm]
r_0,r_1,r_2&=1, 2, 1\\*[-1mm]
\end{aligned}\quad
M_/=\left(\begin{array}{rrr}
1 & 3 & 2 \\
0 & 1 & 1 \\
0 & 0 & 1
\end{array}\right)\quad
Q_/=\Qthree{3}{-1}{1}
\]
{\tiny \raggedright Collection: \hl{[1, (1, 0, 0, 0, 0, 0, 0)], [1, (-1, 1, 1, 1, 1, 1, 1)]},\ \ \hl{[2, (3, 0, 0, 0, 0, 0, 0)]},\ \ \hl{[1, (2, 0, -1, 0, 0, 0, 0)], [1, (2, 0, 0, -1, 0, 0, 0)], [1, (2, 0, 0, 0, -1, 0, 0)], [1, (2, 0, 0, 0, 0, -1, 0)], [1, (2, 0, 0, 0, 0, 0, -1)], [1, (2, -1, 0, 0, 0, 0, 0)]}

}

\subsubsection{Label (3, 10)}
\[
\begin{aligned}
\alpha_0,\alpha_1,\alpha_2&=1, 2, 6\\*[-1mm]
r_0,r_1,r_2&=2, 1, 1\\*[-1mm]
\end{aligned}\quad
M_/=\left(\begin{array}{rrr}
1 & 3 & 5 \\
0 & 1 & 1 \\
0 & 0 & 1
\end{array}\right)\quad
Q_/=\Qthree{3}{-1}{1}
\]
{\tiny \raggedright Collection: \hl{[2, (-1, 1, 1, 1, 1, 2, 2)]},\ \ \hl{[1, (1, 0, 0, 0, 0, 0, 1)], [1, (1, 0, 0, 0, 0, 1, 0)]},\ \ \hl{[1, (2, 0, -1, 0, 0, 0, 0)], [1, (2, 0, 0, -1, 0, 0, 0)], [1, (2, 0, 0, 0, -1, 0, 0)], [1, (1, 0, 0, 0, 0, 1, 1)], [1, (3, -1, -1, -1, -1, 0, 0)], [1, (2, -1, 0, 0, 0, 0, 0)]}

}

\subsubsection{Relations}
\[\begin{array}{cc}%
(3, 8)\rightarrow(3, 9) & [8, 1, 1, 2] \\
(3, 8)\rightarrow(3, 10) & [2, 7, 8, 9]
\end{array}\]\subsection{Minimal collections for \!{\boldmath$X_7$}}
\subsubsection{Label (3, 11)}
\[
\begin{aligned}
\alpha_0,\alpha_1,\alpha_2&=1, 1, 8\\*[-1mm]
r_0,r_1,r_2&=2, 2, 1\\*[-1mm]
\end{aligned}\quad
M_/=\left(\begin{array}{rrr}
1 & 4 & 3 \\
0 & 1 & 1 \\
0 & 0 & 1
\end{array}\right)\quad
Q_/=\Qthree{4}{-1}{1}
\]
{\tiny \raggedright Collection: \hl{[2, (0, 1, 1, 1, 1, 1, 1, 1)]},\ \ \hl{[2, (3, 0, 0, 0, 0, 0, 0, 0)]},\ \ \hl{[1, (2, 0, 0, 0, -1, 0, 0, 0)], [1, (2, 0, 0, 0, 0, -1, 0, 0)], [1, (2, 0, 0, 0, 0, 0, -1, 0)], [1, (2, 0, 0, 0, 0, 0, 0, -1)], [1, (4, -1, -1, -1, -1, -1, -1, -1)], [1, (2, -1, 0, 0, 0, 0, 0, 0)], [1, (2, 0, -1, 0, 0, 0, 0, 0)], [1, (2, 0, 0, -1, 0, 0, 0, 0)]}

}

\subsubsection{Label (3, 12)}
\[
\begin{aligned}
\alpha_0,\alpha_1,\alpha_2&=4, 2, 4\\*[-1mm]
r_0,r_1,r_2&=1, 2, 1\\*[-1mm]
\end{aligned}\quad
M_/=\left(\begin{array}{rrr}
1 & 1 & 1 \\
0 & 1 & 1 \\
0 & 0 & 1
\end{array}\right)\quad
Q_/=\Qthree{1}{-1}{1}
\]
{\tiny \raggedright Collection: \hl{[1, (1, 0, 0, 0, 0, 0, 0, 1)], [1, (1, 0, 0, 0, 0, 0, 1, 0)], [1, (1, 0, 0, 0, 0, 1, 0, 0)], [1, (1, 0, 0, 0, 1, 0, 0, 0)]},\ \ \hl{[2, (3, 0, 0, 0, 0, 0, 0, 0)], [2, (4, -1, -1, -1, 0, 0, 0, 0)]},\ \ \hl{[1, (2, 0, -1, 0, 0, 0, 0, 0)], [1, (2, 0, 0, -1, 0, 0, 0, 0)], [1, (4, -1, -1, -1, -1, -1, -1, -1)], [1, (2, -1, 0, 0, 0, 0, 0, 0)]}

\bigskip

\noindent The full orbit is reachable: certificate = [([6], []), ([6, 2, 3, 4, 5, 1, 3, 4, 2, 3, 0, 2, 1, 3, 4, 5, 6, 5, 4, 3, 2, 1, 3, 4, 5, 0, 2, 1, 3, 4, 1, 3, 0, 2, 0], []), ([5, 6, 5], []), ([0, 2], []), ([1], []), ([3, 4, 5, 1, 3, 4, 2, 3, 0, 2, 1, 3, 4, 5, 6, 5, 4, 3, 2, 1, 3, 4, 5, 0, 2, 1, 3, 4, 1, 3], []), ([4, 1], []), ([5, 4, 3, 2, 1, 3, 4, 5, 6, 0, 2, 3, 4, 5, 1, 3, 4, 2, 3, 0, 2, 1, 3, 4, 5, 4, 2], [4, 5]), ([3, 4, 5, 6, 0, 2, 3, 4, 5, 1, 3, 4, 2, 3, 0, 2, 1, 3, 4, 5, 6, 3, 4, 2, 3, 0, 2, 1, 3, 4, 5, 3, 0, 2, 1, 3, 4, 1, 3, 0, 2, 1, 0], [10, 4, 7, 7])]

}

\subsubsection{Label (3, 13)}
\[
\begin{aligned}
\alpha_0,\alpha_1,\alpha_2&=3, 6, 1\\*[-1mm]
r_0,r_1,r_2&=1, 1, 3\\*[-1mm]
\end{aligned}\quad
M_/=\left(\begin{array}{rrr}
1 & 1 & 4 \\
0 & 1 & 1 \\
0 & 0 & 1
\end{array}\right)\quad
Q_/=\Qthree{1}{-2}{1}
\]
{\tiny \raggedright Collection: \hl{[1, (-1, 1, 1, 1, 1, 1, 1, 1)], [1, (1, 0, 0, 0, 0, 0, 0, 1)], [1, (1, 0, 0, 0, 0, 0, 1, 0)]},\ \ \hl{[1, (1, 0, 0, 0, 0, 0, 1, 1)], [1, (2, -1, 0, 0, 0, 0, 0, 0)], [1, (2, 0, -1, 0, 0, 0, 0, 0)], [1, (2, 0, 0, -1, 0, 0, 0, 0)], [1, (2, 0, 0, 0, -1, 0, 0, 0)], [1, (2, 0, 0, 0, 0, -1, 0, 0)]},\ \ \hl{[3, (7, -1, -1, -1, -1, -1, 0, 0)]}

}

\subsubsection{Label (3, 14)}
\[
\begin{aligned}
\alpha_0,\alpha_1,\alpha_2&=1, 6, 3\\*[-1mm]
r_0,r_1,r_2&=3, 1, 1\\*[-1mm]
\end{aligned}\quad
M_/=\left(\begin{array}{rrr}
1 & 1 & 4 \\
0 & 1 & 1 \\
0 & 0 & 1
\end{array}\right)\quad
Q_/=\Qthree{1}{-2}{1}
\]
{\tiny \raggedright Collection: \hl{[3, (4, 0, 0, 0, 0, 0, 1, 1)]},\ \ \hl{[1, (2, -1, 0, 0, 0, 0, 0, 0)], [1, (1, 0, 0, 0, 0, 0, 1, 1)], [1, (2, 0, -1, 0, 0, 0, 0, 0)], [1, (2, 0, 0, -1, 0, 0, 0, 0)], [1, (2, 0, 0, 0, -1, 0, 0, 0)], [1, (2, 0, 0, 0, 0, -1, 0, 0)]},\ \ \hl{[1, (4, -1, -1, -1, -1, -1, 0, -1)], [1, (2, 0, 0, 0, 0, 0, 0, 0)], [1, (4, -1, -1, -1, -1, -1, -1, 0)]}

}

\subsubsection{Relations}
\[\begin{array}{cc}%
(3, 11)\rightarrow(3, 12) & [2, 3, 4, 5, 10] \\
(3, 11)\rightarrow(3, 13) & [4, 5, 6, 4, 1] \\
(3, 11)\rightarrow(3, 14) & [4, 5, 6, 10, 4]
\end{array}\]\subsection{Minimal collections for \!{\boldmath$X_8$}}
\subsubsection{Label (3, 15)}
\[
\begin{aligned}
\alpha_0,\alpha_1,\alpha_2&=8, 2, 1\\*[-1mm]
r_0,r_1,r_2&=1, 2, 4\\*[-1mm]
\end{aligned}\quad
M_/=\left(\begin{array}{rrr}
1 & 1 & 3 \\
0 & 1 & 2 \\
0 & 0 & 1
\end{array}\right)\quad
Q_/=\Qthree{1}{-1}{2}
\]
{\tiny \raggedright Collection: \hl{[1, (-1, 1, 1, 1, 1, 1, 1, 1, 1)], [1, (1, 0, 0, 0, 0, 0, 0, 1, 1)], [1, (2, -1, 0, 0, 0, 0, 0, 0, 0)], [1, (2, 0, -1, 0, 0, 0, 0, 0, 0)], [1, (2, 0, 0, -1, 0, 0, 0, 0, 0)], [1, (2, 0, 0, 0, -1, 0, 0, 0, 0)], [1, (2, 0, 0, 0, 0, -1, 0, 0, 0)], [1, (2, 0, 0, 0, 0, 0, -1, 0, 0)]},\ \ \hl{[2, (6, -1, -1, -1, -1, -1, -1, -1, 0)], [2, (6, -1, -1, -1, -1, -1, -1, 0, -1)]},\ \ \hl{[4, (15, -3, -3, -3, -3, -3, -3, -2, -2)]}

\bigskip

\noindent The full orbit is reachable: certificate = [([1, 3, 4, 2, 3, 0, 2, 1, 3, 4, 5, 4, 3, 2, 1, 3, 4, 0, 2, 3, 1], []), ([2], []), ([5], []), ([5, 1, 3, 4, 2, 3, 0, 2, 1, 3, 4, 5, 4, 2, 3, 0, 2, 1, 3, 4, 0, 2, 1, 3, 1], []), ([4, 3, 0, 2, 1, 3, 4, 5, 6, 7, 6, 5, 4, 3, 2, 1, 3, 4, 5, 6, 0, 2, 3, 4, 5, 1, 3, 4, 2, 3, 0, 2, 1, 3, 4, 5, 6, 7, 0, 2, 3, 4, 5, 1, 3, 4, 2, 3, 0, 2, 1, 3, 4, 5, 6, 4, 3, 2, 1, 3, 4, 5, 2, 1, 3, 4, 2, 3, 2, 1, 0], []), ([7], []), ([3, 2, 1, 3, 4, 5, 6, 0, 2, 3, 4, 5, 1, 3, 4, 2, 3, 0, 2, 1, 3, 4, 5, 6, 7, 6, 5, 4, 3, 2, 1, 3, 4, 5, 6, 3, 0, 2, 1, 3, 4, 5, 2, 1, 3, 4, 2, 3, 0, 2, 1, 3, 2], []), ([3, 4, 5, 2, 3, 4, 0, 2, 3], []), ([1, 3, 4, 5, 6, 2, 3, 4, 5, 1, 3, 4, 2, 3, 0, 2, 1, 3, 4, 5, 6, 7, 6, 5, 4, 3, 2, 1, 3, 4, 5, 6, 2, 1, 3, 4, 5, 4, 3, 2, 1, 3, 4, 2, 3, 0, 2, 1], []), ([7, 6, 5, 4, 3, 2, 1, 3, 4, 5, 6, 0, 2, 3, 4, 5, 1, 3, 4, 2, 3, 0, 2, 1, 3, 4, 5, 6, 7, 6, 5, 4, 3, 2, 1, 3, 4, 5, 6, 0, 2, 3, 4, 5, 1, 3, 4, 2, 3, 0, 2, 1, 3, 4, 5, 6, 7], [11, 1, 8, 1, 11, 1])]

}

\subsubsection{Label (3, 16)}
\[
\begin{aligned}
\alpha_0,\alpha_1,\alpha_2&=8, 1, 2\\*[-1mm]
r_0,r_1,r_2&=1, 4, 2\\*[-1mm]
\end{aligned}\quad
M_/=\left(\begin{array}{rrr}
1 & 1 & 1 \\
0 & 1 & 2 \\
0 & 0 & 1
\end{array}\right)\quad
Q_/=\Qthree{1}{-1}{2}
\]
{\tiny \raggedright Collection: \hl{[1, (-1, 1, 1, 1, 1, 1, 1, 1, 1)], [1, (1, 0, 0, 0, 0, 0, 0, 1, 1)], [1, (2, -1, 0, 0, 0, 0, 0, 0, 0)], [1, (2, 0, -1, 0, 0, 0, 0, 0, 0)], [1, (2, 0, 0, -1, 0, 0, 0, 0, 0)], [1, (2, 0, 0, 0, -1, 0, 0, 0, 0)], [1, (2, 0, 0, 0, 0, -1, 0, 0, 0)], [1, (2, 0, 0, 0, 0, 0, -1, 0, 0)]},\ \ \hl{[4, (9, -1, -1, -1, -1, -1, -1, 0, 0)]},\ \ \hl{[2, (6, -1, -1, -1, -1, -1, -1, -1, 0)], [2, (6, -1, -1, -1, -1, -1, -1, 0, -1)]}

}

\subsubsection{Label (3, 17)}
\[
\begin{aligned}
\alpha_0,\alpha_1,\alpha_2&=9, 1, 1\\*[-1mm]
r_0,r_1,r_2&=1, 3, 3\\*[-1mm]
\end{aligned}\quad
M_/=\left(\begin{array}{rrr}
1 & 1 & 2 \\
0 & 1 & 3 \\
0 & 0 & 1
\end{array}\right)\quad
Q_/=\Qthree{1}{-1}{3}
\]
{\tiny \raggedright Collection: \hl{[1, (-1, 1, 1, 1, 1, 1, 1, 1, 1)], [1, (2, -1, 0, 0, 0, 0, 0, 0, 0)], [1, (2, 0, -1, 0, 0, 0, 0, 0, 0)], [1, (2, 0, 0, -1, 0, 0, 0, 0, 0)], [1, (2, 0, 0, 0, -1, 0, 0, 0, 0)], [1, (2, 0, 0, 0, 0, -1, 0, 0, 0)], [1, (2, 0, 0, 0, 0, 0, -1, 0, 0)], [1, (2, 0, 0, 0, 0, 0, 0, -1, 0)], [1, (2, 0, 0, 0, 0, 0, 0, 0, -1)]},\ \ \hl{[3, (8, -1, -1, -1, -1, -1, -1, -1, -1)]},\ \ \hl{[3, (11, -2, -2, -2, -2, -2, -2, -2, -2)]}

}

\subsubsection{Label (3, 18)}
\[
\begin{aligned}
\alpha_0,\alpha_1,\alpha_2&=3, 6, 2\\*[-1mm]
r_0,r_1,r_2&=2, 1, 3\\*[-1mm]
\end{aligned}\quad
M_/=\left(\begin{array}{rrr}
1 & 1 & 5 \\
0 & 1 & 1 \\
0 & 0 & 1
\end{array}\right)\quad
Q_/=\Qthree{1}{-1}{1}
\]
{\tiny \raggedright Collection: \hl{[2, (0, 1, 1, 1, 1, 1, 1, 1, 2)], [2, (0, 1, 1, 1, 1, 1, 1, 2, 1)], [2, (3, 0, 0, 0, 0, 0, 0, 0, 0)]},\ \ \hl{[1, (2, -1, 0, 0, 0, 0, 0, 0, 0)], [1, (2, 0, -1, 0, 0, 0, 0, 0, 0)], [1, (2, 0, 0, -1, 0, 0, 0, 0, 0)], [1, (2, 0, 0, 0, -1, 0, 0, 0, 0)], [1, (2, 0, 0, 0, 0, -1, 0, 0, 0)], [1, (2, 0, 0, 0, 0, 0, -1, 0, 0)]},\ \ \hl{[3, (10, -2, -2, -2, -2, -2, -2, -1, -1)], [3, (8, -1, -1, -1, -1, -1, -1, -1, -1)]}

}

\subsubsection{Label (3, 19)}
\[
\begin{aligned}
\alpha_0,\alpha_1,\alpha_2&=3, 2, 6\\*[-1mm]
r_0,r_1,r_2&=2, 3, 1\\*[-1mm]
\end{aligned}\quad
M_/=\left(\begin{array}{rrr}
1 & 1 & 1 \\
0 & 1 & 1 \\
0 & 0 & 1
\end{array}\right)\quad
Q_/=\Qthree{1}{-1}{1}
\]
{\tiny \raggedright Collection: \hl{[2, (0, 1, 1, 1, 1, 1, 1, 1, 2)], [2, (2, 0, 0, 0, 0, 0, 1, 1, 1)], [2, (3, 0, 0, 0, 0, 0, 0, 0, 0)]},\ \ \hl{[3, (4, 0, 0, 0, 0, 0, 0, 1, 1)], [3, (4, 0, 0, 0, 0, 0, 1, 0, 1)]},\ \ \hl{[1, (2, -1, 0, 0, 0, 0, 0, 0, 0)], [1, (2, 0, -1, 0, 0, 0, 0, 0, 0)], [1, (2, 0, 0, -1, 0, 0, 0, 0, 0)], [1, (2, 0, 0, 0, -1, 0, 0, 0, 0)], [1, (2, 0, 0, 0, 0, -1, 0, 0, 0)], [1, (4, -1, -1, -1, -1, -1, -1, -1, 0)]}

}

\subsubsection{Label (3, 20)}
\[
\begin{aligned}
\alpha_0,\alpha_1,\alpha_2&=1, 5, 5\\*[-1mm]
r_0,r_1,r_2&=5, 1, 2\\*[-1mm]
\end{aligned}\quad
M_/=\left(\begin{array}{rrr}
1 & 2 & 9 \\
0 & 1 & 1 \\
0 & 0 & 1
\end{array}\right)\quad
Q_/=\Qthree{2}{-1}{1}
\]
{\tiny \raggedright Collection: \hl{[5, (4, 1, 1, 1, 1, 1, 2, 2, 2)]},\ \ \hl{[1, (2, -1, 0, 0, 0, 0, 0, 0, 0)], [1, (2, 0, -1, 0, 0, 0, 0, 0, 0)], [1, (2, 0, 0, -1, 0, 0, 0, 0, 0)], [1, (2, 0, 0, 0, -1, 0, 0, 0, 0)], [1, (2, 0, 0, 0, 0, -1, 0, 0, 0)]},\ \ \hl{[2, (6, -1, -1, -1, -1, -1, -1, -1, 0)], [2, (6, -1, -1, -1, -1, -1, -1, 0, -1)], [2, (6, -1, -1, -1, -1, -1, 0, -1, -1)], [2, (8, -2, -2, -2, -2, -2, -1, -1, -1)], [2, (9, -2, -2, -2, -2, -2, -2, -2, -2)]}

}

\subsubsection{Label (3, 21)}
\[
\begin{aligned}
\alpha_0,\alpha_1,\alpha_2&=5, 1, 5\\*[-1mm]
r_0,r_1,r_2&=1, 5, 2\\*[-1mm]
\end{aligned}\quad
M_/=\left(\begin{array}{rrr}
1 & 2 & 1 \\
0 & 1 & 1 \\
0 & 0 & 1
\end{array}\right)\quad
Q_/=\Qthree{2}{-1}{1}
\]
{\tiny \raggedright Collection: \hl{[1, (2, -1, 0, 0, 0, 0, 0, 0, 0)], [1, (2, 0, -1, 0, 0, 0, 0, 0, 0)], [1, (2, 0, 0, -1, 0, 0, 0, 0, 0)], [1, (2, 0, 0, 0, -1, 0, 0, 0, 0)], [1, (2, 0, 0, 0, 0, -1, 0, 0, 0)]},\ \ \hl{[5, (16, -3, -3, -3, -3, -3, -2, -2, -2)]},\ \ \hl{[2, (6, -1, -1, -1, -1, -1, -1, 0, -1)], [2, (8, -2, -2, -2, -2, -2, -1, -1, -1)], [2, (6, -1, -1, -1, -1, -1, 0, -1, -1)], [2, (6, -1, -1, -1, -1, -1, -1, -1, 0)], [2, (9, -2, -2, -2, -2, -2, -2, -2, -2)]}

}

\subsubsection{Relations}
\[\begin{array}{cc}%
(3, 15)\rightarrow(3, 16) & [10] \\
(3, 18)\rightarrow(3, 16) & [2, 10, 11] \\
(3, 18)\rightarrow(3, 17) & [9, 1, 1, 2] \\
(3, 21)\rightarrow(3, 18) & [7, 8, 7] \\
(3, 21)\rightarrow(3, 19) & [5, 6, 8, 11] \\
(3, 21)\rightarrow(3, 20) & [5]
\end{array}\]

\section{Derivation of the list of minimal collections}
\label{sec:minimal}
\subsection{Main result}
\begin{theorem}
  \label{thm:minimal_list}
  Let $X$ be a del Pezzo surface and let $\EE$ be a very strong block exceptional collection on $X$ without broken blocks (see Definition \ref{def:block_form}).
  Assume that $\EE$ is block complete (see \S\ref{sec:quiver_setting}) and ``minimal'' in the sense that its total rank cannot be reduced by a block quiver mutation
  \footnote{Since the forbidden region does not depend on the straight vertices (see Remark \ref{rem:straight}) this is equivalent to saying that no single quiver mutation reduces rank.} (see \S\ref{sec:mutnccr}). Then the number of blocks in $\EE$ is $\le 4$. All possible reduced Gram matrices and reduced quivers are listed
  in \S\ref{sec:minimal_collections}.
\end{theorem}
The fact that \S\ref{sec:minimal_collections} lists all the Gram matrices in the 3-block case follows directly from \cite{KarpovNogin}. In the subsequent sections with will treat the cases
of four blocks, five blocks and six or more blocks separately. We will start however with a discussion of the methods that we use.
In general they are a combination of plane geometry arguments, to establish various inequalities relating the areas of regions inside the polygon~$P$ associated on $\EE$,  and``integrality
constraints'', reflecting the fact that various
numbers computed from $P$ need to be integral.

\medskip Some of the discussions below must be completed by computer computations. The accompanying scripts can be found at \cite{GHrepo}.
\subsection{General approach}
\label{sec:approach}
Let
\[
\EE=((E_0^{(0)},\ldots, E_0^{(\alpha_0-1)}), 
\ldots,
(E_{k-1}^{(0)},\ldots, E_{k-1}^{(\alpha_{k-1}-1)}))
\]
be as in the statement of Theorem \ref{thm:minimal_list}.
As usual  $P$ is the polygon corresponding to $\EE$ and $L$ is its ambient space.
  We will equip $L$  with the lattice $L_{\ZZ}$ generated by the vertices of $P$ and with the normalised volume form $\omega$ (see Definition \ref{def:normalized}). The quiver associated to $B(\EE)$ is denoted by $Q$.
\begin{convention}
  Whenever something depends only on the block, but not on the choice of an object in that block,  we will index it by block number,  i.e.  throughout we write
  $\chi_{ij}=\chi(E_i^{(?)},E_j^{(?)})$ and $r_i=r(E_i^{(?)})$. The vectors $(l_{i,i+1})_i$ will
  denote the non-straight vertices of $P$ and $(m_i)_i$ will be the long edges corresponding to the blocks $(\EE_i)_i$.
\end{convention}
We state some relevant facts.
\begin{enumerate}
  \item
By Lemma \ref{lem:no_parallel} the polygon $P$ has no parallel edges.
Moreover by  Lemma \ref{lem:very_useful} its long edges correspond to the blocks in $\EE$ and they are subdivided into $\alpha_i$ equal segments
corresponding respectively to $E_i^{(j)}$ for $j=0,\ldots,\alpha_i-1$.
\item The fact that $\EE$ is assumed to be block complete means that the shape of $P$ is very constrained by Corollary \ref{cor:quiver_pos}.
\item
  The hypothesis that the total rank of $\EE$ cannot be a reduced by a quiver mutation implies that the origin $o$ is in the forbidden region (see Definition \ref{def:forbidden_region}) in $P$.
\item
  We frequently need to make assertions about the forbidden region and for this Lemma \ref{lem:nonadm_bigger} is very useful as it allows us to delete vertices from $|P|$ (recall that by Remark \ref{rem:straight} the forbidden region only depends on $|P|$).
  \item
We have
\begin{equation}
\label{eq:K_theory}
\alpha_0+\alpha_1+\cdots+\alpha_{k-1}+K^2_X=12
\end{equation}
To see this note that this equation is correct for $\PP^1 \times \PP^1$ and $\PP^2$ and it is preserved under blow up and blow down.
\item
From \eqref{eq:chi_slope} we obtain
\begin{equation}
\label{eq:work_horse}
\frac{\chi_{01}}{r_0r_1}+\frac{\chi_{12}}{r_1r_2}+\cdots+\frac{\chi_{k-1,0}}{r_{k-1}r_0}=K_X^2
\end{equation}
where for reasons of symmetry we  commit an abuse of notation by writing $\chi_{k-1,0}:=\chi(E_{k-1}^{(?)}, \omega_X^{-1}\otimes E_0^{(?)})$.
Our main approach will be to use plane geometry to carefully bound the various terms in \eqref{eq:work_horse} e.g.\ using the formula \eqref{eq:area_formula}.
\item
The data $(r_i)_i$, $(\chi_{i,i+1})_i$  determines the full reduced Gram matrix of $\EE$ via
\begin{equation}
  \label{eq:full_gram}
  \frac{\chi_{ij}}{r_ir_j}=\frac{\chi_{i,i+1}}{r_ir_{i+1}}+\frac{\chi_{i+1,i+2}}{r_{i+1}r_{i+2}}+
  \cdots+\frac{\chi_{j-1,j}}{r_{j-1}r_{j}}
\end{equation}
for $1\le i<j\le k-1$ which also follows from \eqref{eq:chi_slope}.
\item It will be convenient to write $R_i=\omega(l_{i-1,i},l_{i,i+1})$. Then from
  \ref{thm:achievement}\eqref{it:thm2_4} we obtain
  \begin{equation}
    \label{eq:Rformula}
    R_i=\alpha_i r_i^2
    \end{equation}
\item
Likewise  it will be convenient to write. 
  \begin{equation}
    \label{eq:Sformula}
S_{i,i+1}=2\Delta(l_{i-1,i},l_{i,i+1},l_{i+1,i+2})=\omega(m_i,m_{i+1})
\end{equation}
Then 
from \ref{thm:achievement}\eqref{it:thm2_4} we obtain
\begin{equation}
\label{eq:chi_formula}
\chi_{i,i+1}=\frac{S_{i,i+1}}{\alpha_i \alpha_{i+1} r_i r_{i+1}}
\end{equation}
\item Let $c_{ij}$ be the number of arrows from $i$ to $j$ in $Q_{\red}$ with the usual convention that a negative number denotes
  $-c_{ij}$ arrows in the reverse direction. Then with a short computation using Theorem \ref{thm:achievement}(\ref{it:thm2_3},\ref{it:thm2_5}) we find
  \begin{equation}
    \label{eq:quiver_arrows}
    c_{ij}^2\alpha_i\alpha_j  =
    \frac{\omega(m_i,m_j)^2}{\omega(l_{i-1,j},l_{i,i+1})\omega(l_{j-1,j},l_{j,j+1})}
  \end{equation}
  Note that $c_{i,i+1}=\chi_{i,i+1}$ and in that case \eqref{eq:quiver_arrows} is a rewrite of \eqref{eq:chi_formula}.
\item
Assuming we know $(\alpha_i)_i$, the reduced Gram matrix determines the Gram matrix~$M$. This can be used to further prune the solutions since
the ``Serre matrix'' $s=M^{-1}M^t$ satisfies some very strong conditions (see e.g.\ \cite[Definition 3.2.1]{dTdVVdB}). I.e.\
\begin{enumerate}
\item $s$ is unipotent.
\item $\rk(s-1)\le 2$.
\end{enumerate}
\item
To ultimately validate the existence of a particular Gram matrix one may supply an explicit exceptional collection realising that Gram matrix. 
\item
To manipulate sums of reciprocals as in \eqref{eq:work_horse} we may use the following lemma:
\end{enumerate}
\begin{lemma}
\label{lem:weird_lemma}
Let $a>0$ and assume that $(S_i)_{i=1}^{n}$ are subsets of $\RR_{>0}$ so that for all $i$ and for all $b>0$ we have that $S_i\cap [b,\infty)$ is finite.
Then the set
\[
S=\{s_1+\cdots+s_n\mid \forall i: s_i \in S_i, s_1+\cdots+s_n<a\}
\]
has a maximum for all $a>0$. Moreover the set
\[
T=\{s_1,s_2,\ldots, s_n\mid \forall i: s_i \in S_i, s_1+\cdots+s_n=a\}
\]
is finite.
\end{lemma}
\begin{proof} We will start with proving the first claim using induction on $n$. The case $n=1$ follows from the fact that by assumption $S_1\cap [s,\infty)$ is finite for any $s\in S_1$.

Now assume $n>1$. Let $m$ be the maximum of
\[
\{s_2+\cdots+s_n\mid \forall i\ge 2: s_i \in S_i, s_2+\cdots+s_n<a\}
\]

and for any $s\in S_1\cap [0,a)$ let $m_s$ be the maximum of
\[
\{s_2+\cdots+s_n\mid \forall i\ge 2: s_i \in S_i, s_2+\cdots+s_n<a-s\}
\]
We need to prove that $s+m_s$ has a maximum for $s\in S_1$.

Clearly if $a-s>m$ then $m_s=m$ and so $s+m_s=s+m$. Moreover $s<a-m$ and by the case $n=1$ the set $\{s<a-m\}$ has a maximum.

If $a-s\le m$ then $s\ge a-m$ and this can only happen for a finite number of $s\in S_1$. So we get
\[
\max S= \max\left(\max_{s\in S_1\cap [a-m,a)}(s+m_s), (\max_{s\in S_1\cap [0,a-m)} s)+m\right)
\]
Now we discuss the second claim. If $(s_i)_i\in T$ then with notations as above we have $s_1\ge a-m$. Hence there are only a finite number of possibilities for $s_1$.
We now use induction.
\end{proof}
\begin{remark}
The proof clearly leads to an algorithm for computing $\max S$. However it seems to be very slow. 
  \end{remark}
  \begin{remark}
    \label{rem:typical}
  A typical application of the lemma would be to take $S_i=\{a_i/k\mid k\in \NN_0\}$ for fixed $(a_i)_i$. This implies for example that \eqref{eq:work_horse} has a finite number of solutions for $(r_i)_i$ for given
  $(\chi_{i,i+1})_i$.
\end{remark}
\subsection{Restrictions on non-admissible vertices}
Lemma \ref{lem:three_adm} imposes strong restrictions on the non-admissible vertices of $P$.
The following preliminary result imposes further restrictions in the context of Theorem \ref{thm:minimal_list}.
\begin{lemma} \label{lem:prelim}
  Let $P$ be as above. Then $P$ can have at most two consecutive
  non-admissible non-straight  vertices. Moreover if $P$ has two consecutive
  non-admissible non-straight vertices
  $l_{i,i+1}$, $l_{i+1,l+2}$ then $r_i=r_{i+1}=r_{i+2}=1$, $\alpha_i=\alpha_{i+1}=\alpha_{i+2}=1$,
  \begin{equation}
    \label{eq:chi_dec}
    \chi_{i,i+1}=c_{i,i+1}=1, \quad
    \chi_{i+1,i+2}=c_{i+1,i+2}=1, \quad
    c_{i,i+2}=1,
  \end{equation}
  and with respect to suitable coordinates the configuration is as in Figure \ref{fig:5bl_basic_case_prelim},  where $o$ denotes the origin.
  
  \begin{figure}[h!]
  \centering
  \caption{} \label{fig:5bl_basic_case_prelim}
  \begin{tikzpicture}[line cap=round,line join=round,>=Stealth]
  \definecolor{myblue}{RGB}{150,180,220}
  \definecolor{myred}{RGB}{200,40,40}

  \def\r{2.0}
  \coordinate (A) at (0,{2*\r});    
  \coordinate (B) at (0,{\r});      
  \coordinate (C) at ({\r},0);      
  \coordinate (D) at ({2*\r},0);    
  \coordinate (O) at ({\r},{\r});   

  \draw[black,line width=0.95pt,->] (A) -- node[midway,left=3pt] {$m_i$} (B);
  \draw[black,line width=0.95pt,->] (B) -- node[midway,below left=2pt] {$m_{i+1}$} (C);
  \draw[black,line width=0.95pt,->] (C) -- node[midway,below=3pt] {$m_{i+2}$} (D);

  \draw[myblue,line width=0.70pt] (A)--(D);
  \draw[myblue,line width=0.70pt] (B)--(O);
  \draw[myblue,line width=0.70pt] (O)--(C);

  \fill[black] (O) circle (2.1pt);
  \node[below=2pt] at (O) {$o$};

  \node[above right=1pt] at (A) {$l_{i-1,i}$};
  \node[right=3pt]       at (B) {$l_{i,i+1}$};
  \node[above right=2pt]       at (C) {$l_{i+1,i+2}$};
  \node[below right=2pt] at (D) {$l_{i+2,i+3}$};

  \node[myred,left=6pt]  at (A) {$(0,2)$};
  \node[myred,left=6pt]  at (B) {$(0,1)$};
  \node[myred,below=6pt] at (C) {$(1,0)$};
  \node[myred,right=6pt] at (O) {$(1,1)$};
  \node[myred,right=12pt] at (D) {$(2,0)$};

\end{tikzpicture}
  \end{figure}
\end{lemma}
\begin{proof} Assume that $P$ has $\ge 2$ consecutive non-admissible  vertices $l_{i,i+1},\ldots, l_{j-1,j}$ (i.e.\ $j\ge i+2$) with $l_{i-1,i}$, $l_{j,j+1}$ being admissible.
  Then the situation is like in Figure \ref{fig:multiple_blocks_prelim}, where the origin is in the region marked $F$ bounded by two light blue lines, respectively parallel to $m_i$ and $m_j$. 
  
  \begin{figure}[h!]
  \centering
  \caption{}\label{fig:multiple_blocks_prelim}
 \begin{tikzpicture}[scale=0.8,line cap=round,line join=round,>=Stealth]
  \usetikzlibrary{decorations.pathmorphing,calc}
  \definecolor{myblue}{RGB}{150,180,220}

  \coordinate (L0) at (-4,0);   
  \coordinate (R0) at ( 4,0);   
  \coordinate (M)  at ( 0,0);   

  \coordinate (L1) at (-3.40,1.20);
  \coordinate (L2) at (-2.50,2.15);

  \coordinate (R1) at ( 3.32,1.45);  
  \coordinate (R2) at ( 2.45,2.30);

  \draw[myblue,line width=0.8pt] (L0)--(M);
  \draw[myblue,line width=0.8pt] (M)--(R0);

\draw[myblue,line width=0.8pt]
  ($(L0)!0.44!(M)$)
    .. controls ($(L0)!0.47!(M)+(0,0.16)$) and ($(L0)!0.50!(M)+(0,-0.16)$)
  .. ($(L0)!0.53!(M)$);

\draw[myblue,line width=0.8pt]
  ($(M)!0.47!(R0)$)
    .. controls ($(M)!0.50!(R0)+(0,0.16)$) and ($(M)!0.53!(R0)+(0,-0.16)$)
  .. ($(M)!0.56!(R0)$);

  \draw[black,line width=0.95pt] (L0)--(L1) node[midway,left=3pt] {$m_j$};
  \draw[black,line width=0.95pt] (L1)--(L2) node[midway,above left=0.8pt] {$m_{j-1}$};

  \draw[black,line width=0.95pt] (R2)--(R1) node[midway,above right=0.5pt] {$m_{i+1}$};
  \draw[black,line width=0.95pt] (R1)--(R0) node[midway,right=3pt] {$m_i$};

  \draw[black,dashed,line width=0.9pt]
    (L2) .. controls (-0.9,3.1) and (0.9,3.1) .. (R2);
    
  \def\k{2.2} 
  \coordinate (D1) at ($(M) + \k*(L0) - \k*(L1)$); 
  \coordinate (D2) at ($(M) + \k*(R0) - \k*(R1)$); 
  \draw[myblue,line width=0.8pt] (M)--(D1);
  \draw[myblue,line width=0.8pt] (M)--(D2);

  \node[myblue,scale=1.4] at (0,-1.7) {$F$};

  \node[below left=0.5pt]  at (L0) {$l_{j,j+1}$};
  \node[right=1pt] at (L1) {$l_{j-1,j}$};

  \node[left=1pt]  at (R1) {$l_{i,i+1}$};
  \node[below right=0.5pt] at (R0) {$l_{i-1,i}$};
\end{tikzpicture}
  \end{figure}
  
  We claim first that for $i\le u<v\le j$
we have
\begin{equation}
  \label{eq:omega_claim}
\frac{\omega(m_u,m_v)^2}{\omega(l_{u-1,u},l_{u,u+1})\omega(l_{v-1,v},l_{v,v+1})}<4.
\end{equation}
We may make an
affine coordinate change so that $m_u=(0,-1)$ and $m_v=(0,1)$. Then the picture becomes as in Figure \ref{fig:5bl_after_coordinates_prelim}.

\begin{figure}[h!]
\centering 
\caption{} \label{fig:5bl_after_coordinates_prelim}
\begin{tikzpicture}[line cap=round,line join=round,>=Stealth]
  \definecolor{myblue}{RGB}{150,180,220}
  \def\a{3.2}   
  \def\b{1.7}   

  \coordinate (U) at (0,{ \b+2});     
  \coordinate (L) at (0,{ \b});        
  \coordinate (A) at (\a,0);           
  \coordinate (B) at ({\a+2},0);     

  \draw[myblue,line width=0.75pt] (U)--(B);

  \coordinate (P) at ($(U)!0.46!(B)$);

  \draw[myblue,line width=0.75pt] (P) -- ++(0,1.55);
  \draw[myblue,line width=0.75pt] (P) -- ++(2.25,0);

  \node[myblue,scale=1.5] at ($(P)+(1.15,0.75)$) {$\widetilde{F}$};

  \draw[black,line width=0.95pt,->] (U) -- node[midway,left=4pt] {$m_u$} (L);

  \draw[black,line width=0.95pt,->] (A) -- node[midway,below=3pt] {$m_v$} (B);

  \node[left=4pt]  at (U) {$(0,b+1)$};
  \node[left=4pt]  at (L) {$(0,b)$};
  \node[below=4pt] at (A) {$(a,0)$};
  \node[below right=4pt] at (B) {$(a+1,0)$};
\end{tikzpicture}
\end{figure}
  Let $o$ be the origin. It is in the area marked by $\tilde{F}$, which contains the original $F$. I.e.
\begin{equation}
\label{eq:origin}
o=\left(\ge \frac{a+1}{2}, \ge \frac{b+1}{2}\right)
\end{equation}
Scaling $\omega$ such that $\omega(m_u,m_v)=1$ we then obtain
\begin{equation}
\label{eq:final_estimate}
\frac{\omega(m_u,m_v)^2}{\omega(l_{u-1,u},l_{u,u+1})\omega(l_{u-1,u},l_{u,u+1})}\le \frac{1}{(a+1)/2\cdot (b+1)/2}=\frac{4}{(a+1)(b+1)}\le 4.
\end{equation}
Let us now analyse when we have equality.  By \eqref{eq:final_estimate} this implies $a=b=0$, i.e. $v=u+1$.  Inspecting how we obtain the first inequality in \eqref{eq:final_estimate} we see the inequalities in \eqref{eq:origin} must also be equalities.
In other words $o$ is the midpoint of interval $[l_{u-1,u},l_{v,v+1}]$, which 
is not in $F$, unless $u=i$, $v=j$. But then there is only a single
non-admissible vertex, namely $l_{i,i+1}$, contradicting our assumption. Hence the inequality in \eqref{eq:final_estimate} is strict.

Since $c_{uv}$ is integral we obtain from \eqref{eq:omega_claim} together with \eqref{eq:quiver_arrows}
\begin{equation}
  \label{eq:c_bound}
  c_{uv}^2\alpha_u\alpha_v\le 3
\end{equation}
which implies
\begin{equation}
  \label{eq:c_bound2}
  c_{uv}=1\text{ for $i\le u<v\le j$}.
  \end{equation} If $j\ge i+3$ then this contradicts \eqref{eq:plucker2}. Hence $j=i+2$. In the sequel to simplify the notations let
us put $i=1$, $j=3$.  Then \eqref{eq:c_bound} combined with \eqref{eq:c_bound2} yields
\begin{equation}
  \label{eq:alpha_bound2}
  \alpha_1\alpha_2\le 3,\quad \alpha_2\alpha_3\le 3, \quad  \alpha_1\alpha_3\le 3.
  \end{equation}

Put
\[
\tilde{m}_i=\frac{m_i}{\alpha_i r_i}
\]
According to Theorem \ref{thm:achievement}\eqref{it:thm2_1} one has $\tilde{m}_i\in L_\ZZ$
and by Theorem \ref{thm:achievement}\eqref{it:thm2_5}
we have from $c_{uv}=1$:
\begin{align*}
\omega(\tilde{m}_1,\tilde{m}_2)&=1\\
\omega(\tilde{m}_2,\tilde{m}_3)&=1\\
\omega(\tilde{m}_1,\tilde{m}_3)&=1
\end{align*}
The only solution for this is (after choosing suitable coordinates for $L_\ZZ$)
\begin{align*}
\tilde{m}_1&=(0,-1)\\
\tilde{m}_2&=(1,-1)\\
\tilde{m}_3&=(1,0)
\end{align*}
So the situation is as in Figure \ref{fig:5bl_basic_case}. 

\begin{figure}
\centering
\caption{} \label{fig:5bl_basic_case}
\begin{tikzpicture}[scale=0.8,line cap=round,line join=round,>=Stealth]
 
  \coordinate (P0) at (0,4);    
  \coordinate (P1) at (0,2.2);  
  \coordinate (P2) at (2.2,0);  
  \coordinate (P3) at (6.6,0);  

  \draw[black,line width=0.95pt,->] (P0) -- node[midway,right=3pt] {$m_1$} (P1);
  \draw[black,line width=0.95pt,->] (P1) -- node[midway,above right=2pt] {$m_2$} (P2);
  \draw[black,line width=0.95pt,->] (P2) -- node[midway,above=3pt] {$m_3$} (P3);

  \node[left=8pt] at (P0) {$(0,\alpha_2 r_2+\alpha_1 r_1)$};
  \node[left=8pt] at (P1) {$(0,\alpha_2 r_2)$};
  \node[below=7pt] at (P2) {$(\alpha_2 r_2,0)$};
  \node[below=7pt] at (P3) {$(\alpha_2 r_2+\alpha_3 r_3,0)$};

  \coordinate (O) at (6.7,3.2);
  \fill[black] (O) circle (2.1pt);
  \node[above=2pt] at (O) {$o=(x,y)$};
\end{tikzpicture}
\end{figure}

It follows from Theorem \ref{thm:achievement}\eqref{it:thm2_3}, applied 
respectively to the pairs $(l_{01},l_{12})$ and $(l_{23}, l_{34})$, 
that $o$ lies on the lines
$x=r_1$, $y=r_3$. In other words $o=(r_1,r_3)$. Moreover Theorem \ref{thm:achievement}\eqref{it:thm2_3} applied to the pair $(l_{12}, l_{23})$
yields
\begin{equation}
\label{eq:r2expression}
r_2=\frac{r_1+r_3}{\alpha_2+1}
\end{equation}
The fact that $o$ is in the forbidden region implies ${x\ge (\alpha_2 r_2+\alpha_3 r_3)/2}$ and ${y\ge (\alpha_2 r_2+\alpha_3 r_1)/2}$.
From the first inequality we get
\begin{align*}
r_1&\ge \frac{1}{2}(\alpha_2 r_2+\alpha_3 r_3)\\
&=\frac{1}{2}\left(\frac{\alpha_2}{\alpha_2+1}(r_1+r_2)+\alpha_3 r_3          \right)\\
&=\frac{1}{2}\left(\frac{\alpha_2}{\alpha_2+1}r_1+\left(\frac{\alpha_2}{\alpha_2+1} +\alpha_3 \right)r_3          \right)
\end{align*}
which yields
\begin{equation}
\label{eq:bound1_}
\left(2-  \frac{\alpha_2}{\alpha_2+1}  \right)r_1 \ge
\left(     \frac{\alpha_2}{\alpha_2+1}+\alpha_3        \right)r_3
\end{equation}
Likewise
\begin{equation}
\label{eq:bound2_}
\left(2-  \frac{\alpha_2}{\alpha_2+1}  \right)r_3 \ge
\left(     \frac{\alpha_2}{\alpha_2+1}+\alpha_1        \right)r_1
\end{equation}
Multiplying both formulas and cancelling $r_1r_3$ yields
\begin{equation}
\label{eq:alpha_bound}
\left(2-  \frac{\alpha_2}{\alpha_2+1}  \right)^2
\ge \left(     \frac{\alpha_2}{\alpha_2+1}+\alpha_1        \right)\left(     \frac{\alpha_2}{\alpha_2+1}+\alpha_3        \right)
\end{equation}
From \eqref{eq:alpha_bound2} we get $\alpha_2\le 3$. Assume first $\alpha_2=1$. Then \eqref{eq:alpha_bound} becomes
\[
\frac{3}{2}\times \frac{3}{2} \ge \left(     \frac{1}{2}+\alpha_1        \right)\left(     \frac{1}{2}+\alpha_3        \right)
\]
whose only solution is $\alpha_1=\alpha_3=1$. From \eqref{eq:bound1_}\eqref{eq:bound2_}
we then get $r_1=r_3$ and from \eqref{eq:r2expression} we get $r_2=r_1=r_3$. From Theorem \ref{thm:achievement}\eqref{it:thm2_2} we
then obtain $r_1=r_2=r_3=1$.

If $\alpha_2= 2$ then from \eqref{eq:c_bound} we get $\alpha_1=\alpha_3=1$ and \eqref{eq:alpha_bound} becomes
\[
\frac{4}{3}\times \frac{4}{3}\ge \left(\frac{2}{3}+1\right)\left(\frac{2}{3}+1\right),
\]
which is false. Similarly,  if $\alpha_2=3$ then we get again $\alpha_1=\alpha_3=1$ and 
\[
\frac{5}{4}\times \frac{5}{4}\ge \left(\frac{3}{4}+1\right)\left(\frac{3}{4}+1\right),
\]
which is again false. This finishes the proof of the lemma.
\end{proof}
\begin{remark}
  The second part of Lemma \ref{lem:prelim} will ultimately turn out to be
  empty since, as asserted in Theorem \ref{thm:minimal_list}, $P$ can have at most 4 non-straight vertices.
  However Lemma \ref{lem:prelim} will be an intermediate step in establishing this.
\end{remark}
\subsection{Geometric lemmas for quadrangles}
\label{sec:geometric}
We consider some ad hoc geometric lemmas for convex quadrangles in the plane which will be used to bound the terms in \eqref{eq:work_horse}
during the treatment of 4- and 5-block collections.  First assume that we are in the situation of Figure \ref{fig:4bl_1},  where $P=\conv(abcd)$ is a convex polygon,  the vertex $d$ is non-admissible (see Definition \ref{def:admissible}) and the ``origin'' $o$ is in intersection of the forbidden region (see Definition \ref{def:forbidden_region}) and the interior of $P$. 

\begin{figure}[h!]
\centering
\caption{}\label{fig:4bl_1}
\begin{tikzpicture}[scale=0.5,line cap=round,line join=round]
  \definecolor{mygreen}{RGB}{0,128,72}

  \coordinate (c) at (0,0);
  \coordinate (d) at (1.15,-3.0);
  \coordinate (a) at (5.55,-3.35);
  \coordinate (b) at (7.85,4.15);
  \coordinate (o) at ($(d)!0.42!(b)$);

  \draw[black,line width=1pt] (c)--(b)--(a)--(d)--cycle;
  \draw[black,line width=1pt] (c)--(a);

  \draw[mygreen,line width=1pt] (d)--(o)--(b);
  \draw[mygreen,line width=1pt] (c)--(o)--(a);
  \fill[mygreen] (o) circle (3pt);

  \node at ($(c)+(-0.52,0.08)$) {$c$};
  \node at ($(d)+(-0.40,-0.52)$) {$d$};
  \node at ($(a)+(0.22,-0.35)$) {$a$};
  \node at ($(b)+(0.42,0.20)$) {$b$};
  \node[text=mygreen] at ($(o)+(0.42,-0.22)$) {$o$};
\end{tikzpicture}
\end{figure}

Note that the non-admissible vertex is unique and its existence implies that $P$ has no parallel edges.

\begin{lemma} \label{lem:forbidden_area}
The assumption that the origin is in the forbidden region is equivalent to the following conditions:
\begin{align*}
\frac{\area(oab)}{\area(abc)}&\le \frac{1}{2}&
\frac{\area(obc)}{\area(abc)}&\le \frac{1}{2}\\
\frac{\area(oda)}{\area(abd)}&\le \frac{1}{2}&
\frac{\area(ocd)}{\area(bcd)}&\le \frac{1}{2}
\end{align*}
\end{lemma}
\begin{lemma}
\label{lem:mainlem}
We have the following inequalities
\begin{align}
\label{eq:basic_eq}
\frac{\area(cda)}{\area(oda)}&\le \frac{2\area(oab)+2\area(obc)-\area(abc)}{\area(oab)}\\
\label{eq:basic_eq_sub}
\frac{\area(cda)}{\area(ocd)}&\le \frac{2\area(oab)+2\area(obc)-\area(abc)}{\area(obc)}
\end{align}
\end{lemma}
\begin{proof} The proofs of \eqref{eq:basic_eq} and \eqref{eq:basic_eq_sub} are similar, so we concentrate on \eqref{eq:basic_eq}.
We put:
\[
x=\frac{\area(obc)}{\area(abc)}, \qquad y =\frac{\area(oab)}{\area(abc)}
\]
After applying an affine transformation we may assume that $a = (0,1),b = (0,0),c=(1,0),d = (u,v),o=(x,y)$,  as
in Figure \ref{fig:coordinates}
\begin{figure}[!]
\centering
\caption{}\label{fig:coordinates}
\begin{tikzpicture}[x=6.2cm,y=6.2cm,line cap=round,line join=round]

\definecolor{mygreen}{RGB}{0,128,70}
\definecolor{myblue}{RGB}{150,180,210}
\definecolor{myred}{RGB}{245,55,35}

\coordinate (A) at (0,1);
\coordinate (B) at (0,0);
\coordinate (C) at (1,0);

\coordinate (O) at (0.34,0.36);

\coordinate (Op) at ($(B)!2!(O)$);

\coordinate (D) at (0.52,0.7);

\coordinate (Mab) at ($(A)!0.5!(B)$);
\coordinate (Mbc) at ($(B)!0.5!(C)$);
\coordinate (Mac) at ($(A)!0.5!(C)$);

\draw[black,line width=0.9pt] (B)--(A)--(C)--cycle;
\draw[black,line width=0.9pt] (A)--(D)--(C);
\draw[black,line width=0.9pt] (A)--(C);

\draw[myred,line width=0.9pt] (Mab)--(Mac)--(Mbc)--cycle;

\draw[myblue,line width=0.9pt] (A)--(Op)--(C);
\draw[myblue,line width=0.9pt] (Mab)--(O)--(Mbc);
\fill[myblue] (Op) circle (2.2pt);

\draw[mygreen,line width=0.9pt] (A)--(O)--(B);
\draw[mygreen,line width=0.9pt] (O)--(C);
\draw[mygreen,line width=0.9pt] (O)--(D);
\fill[mygreen] (O) circle (2.2pt);

\node at ($(A)+(-0.12,0.04)$) {$a(0,1)$};
\node at ($(B)+(-0.11,-0.06)$) {$b(0,0)$};
\node at ($(C)+(0.06,-0.05)$) {$c(1,0)$};
\node[text=mygreen] at ($(O)+(0.12,0.02)$) {$o(x,y)$};
\node at ($(Op)+(0.15,0.02)$) {$o'(2x,2y)$};
\node at ($(D)+(+0.1,-0.08)$) {$d(u,v)$};

\end{tikzpicture}
\end{figure}

The fact that $o$ is in the forbidden region implies by Lemma \ref{lem:nonadm_bigger} that it is in the red triangle, whose vertices are the midpoints of the edges of $\conv(abc)$. 
We construct an auxiliary point $o'=(2x,2y)$ as indicated by the light blue lines. The line $ao'$ is parallel to the line through $o$ and the midpoint of $[ab]$.  Likewise the line $co'$ is parallel with the line through $o$ and the midpoint of $[bc]$.
The fact that $o$ is in the forbidden region implies that $d$ is in the triangle $\conv(co'a)$. 

We claim first that
\begin{equation}
\label{eq:move_d}
\frac{\area(cda)}{\area(oda)}\le \frac{\area(co'a)}{\area(oo'a)}.
\end{equation}
The right hand side does not depend on $d$ so we must analyse how the expression on the left changes if we move $d$. To this end we consider the following picture in Figure \ref{fig:4bl_2},  where the line $cc'$ is parallel to the line $ad$.

\begin{figure}[h!]
\centering
\caption{}\label{fig:4bl_2}
\begin{tikzpicture}[scale=0.6,line cap=round,line join=round]
  \definecolor{mygreen}{RGB}{0,128,70}
  \definecolor{myblue}{RGB}{150,180,210}

  \coordinate (a) at (0,5.2);
  \coordinate (o) at (4.0,-0.1);
  \coordinate (d) at (6.25,3.7);

  \coordinate (cp) at ($(a)!0.58!(o)$);

  \coordinate (c) at ($(cp)+1.02*(d)-1.02*(a)$);

  \draw[black,line width=1.2pt] (a)--(d)--(c)--cycle;
  \draw[black,line width=1.2pt] (a)--(c);

  \draw[mygreen,line width=1.2pt] (a)--(o)--(d);
  \draw[mygreen,line width=1.2pt] (a)--(o);

  \draw[myblue,line width=1.2pt] (cp)--(c);
  \fill[myblue] (cp) circle (2.2pt);

  \fill[mygreen] (o) circle (2.4pt);

  \node at ($(a)+(-0.55,0.10)$) {$a$};
  \node at ($(d)+(0.35,0.30)$) {$d$};
  \node at ($(c)+(0.35,-0.15)$) {$c$};
  \node[text=mygreen] at ($(o)+(0.42,0.10)$) {$o$};
  \node[text=myblue] at ($(cp)+(-0.55,-0.10)$) {$c'$};
\end{tikzpicture}
\end{figure}

One has
\[
\frac{\area(cda)}{\area(oda)}=\frac{|ac'|}{|ao|}.
\]
Here the right hand side does not change if $d$ moves on the line $ad$ but it increases if $\angle dac$ increases. 
Since $d$ is constrained to the region $\conv(co'a)$ we see that the maximum of the left hand side of \eqref{eq:move_d} is achieved when $d=o'$.
This proves \eqref{eq:move_d}. Hence it is sufficient to prove \eqref{eq:basic_eq} with $d=o'$ since the right hand side of \eqref{eq:basic_eq} does not depend
on $d$.  We compute
\begin{align*}
\area(co'a)&=(2x+2y-1)/2\\
\area(oo'a) & = x/2\\
\area(oab)&=x/2\\
\area(obc)&=y/2\\
\area(abc) & = 1/2
\end{align*}
Then \eqref{eq:basic_eq} becomes
\[
\frac{(2x+2y-1)/2}{x/2}\le \frac{2(x/2)+2(y/2)-1/2}{x/2}
\]
and we see that it is in fact an equality.
\end{proof}
\begin{lemma}
\label{lem:basic_bounds}
One has
\begin{equation}
\label{eq:bound1}
\frac{\area(cda)^2}{\area(oda)\area(ocd)}\le 4
\end{equation}
with equality if and only if $o$ is the midpoint of $[ac]$. Furthermore, 
\begin{equation}
\label{eq:bound2}
 \frac{\area(abc)^2}{\area(obc)\area(oab)}\ge 4.
\end{equation}
with equality if and only if $o$ is the midpoint of $[ac]$. Finally,
\begin{equation}
\label{eq:chi_rel1}
\frac{\area(cda)^2}{\area(oda)\area(ocd)}
 \cdot \frac{\area(abc)^2}{\area(oab)\area(obc)}\le 16.
\end{equation}
\end{lemma}
\begin{proof} Both \eqref{eq:bound1} and \eqref{eq:bound2} are easy.  To prove \eqref{eq:bound1} we choose a coordinate system such that $c=(0,1)$, $a=(0,0)$, $d=(1,0)$
and we put $o=(p,q)$. Then the left hand side \eqref{eq:bound1} becomes $\frac{(1/4)}{(p/2)(q/2)}=\frac{1}{pq}$.
The fact that $o$ is in the forbidden region implies $p\ge 1/2$, $q\ge 1/2$ which yields what we want. To prove \eqref{eq:bound2} we use a coordinate system as in Figure \ref{fig:coordinates}.

Now we concentrate on \eqref{eq:chi_rel1}. Put
\[
u=\frac{\area(cda)^2}{\area(oda)\area(ocd)}.
\]
Hence $u\le 4$ by \eqref{eq:bound1}.
Multiplying \eqref{eq:basic_eq} and \eqref{eq:basic_eq_sub} yields:
\begin{equation}
\label{eq:lemma_eq1}
u\le \frac{(2\area(oab)+2\area(obc)-\area(abc))^2}{\area(oab)\area(obc)}
\end{equation}
We claim that this implies \eqref{eq:chi_rel1}.
To prove this we put ourselves in the setting of Figure \ref{fig:coordinates}. Then \eqref{eq:lemma_eq1}
is equivalent to
\[
f(x,y):=(2x+2y-1)^2-uxy\ge 0
\]
We have to prove that this implies $xy\ge u/16$ under the conditions $0<x\le 1/2$, $0<y\le 1/2$, $2x+2y-1\ge 0$, which express that $o$ is in the red
triangle in Figure \ref{fig:coordinates} as well as in the interior of $\conv(abcd)$.

The curve $C := \{f(x,y)=0\}$ is an ellipse and $f$ takes negative values on the interior of the region bounded by $C$.
Clearly,  $C$ is invariant under the reflection $x\leftrightarrow y$, moreover, it intersects the lines $x=1/2$, $y=1/2$ respectively
in the points $\{(1/2,0), (1/2,u/8)\}$ and $\{(0,1/2), (u/8,1/2)\}$.  So the situation is like depicted in Figure \ref{fig:4bl_extremal}. 

\begin{figure}[ht]
\centering
\caption{}\label{fig:4bl_extremal}
\begin{tikzpicture}
\begin{axis}[
    width=11cm,
    height=8.2cm,
    xmin=-0.01, xmax=0.51,
    ymin=-0.01, ymax=0.51,
    axis equal image,
    axis lines=box,
    xtick={0,0.1,...,0.5},
    ytick={0,0.1,...,0.5},
    minor tick num=4,
    tick label style={font=\small},
    clip=false,
    samples=300,
    domain=0:0.5,
]

\addplot[black, thick] {0.5 - x};

\addplot[blue, thick, domain=0:0.5]
    {(2 - 3*x - sqrt(x*(4 - 7*x)))/4};

\addplot[blue, thick, domain=0.25:0.5]
    {(2 - 3*x + sqrt(x*(4 - 7*x)))/4};

\addplot[red, thick, domain=0.25:0.5]
    {1/(8*x)};

\node[blue] at (axis cs:0.09,0.35) {$f<0$};
\node[blue] at (axis cs:0.42,0.43) {$f>0$};
\node[blue] at (axis cs:0.08,0.05) {$f>0$};
\node[red] at (axis cs:0.30,0.36) {$xy=u/16$};

\end{axis}
\end{tikzpicture}
\end{figure}

In blue we have drawn an implicit plot of the curve $C$ and we have indicated the sign of $f$ on the complement of this curve.  We are only interested in the region $f\ge 0$ to the right of the black line. We have to understand the curves of the
form $xy=\lambda$ which intersect this region.  As the picture shows, this happens
for $\lambda\ge u/16$ which gives what we want.
\end{proof}
\begin{lemma}
We have
\begin{equation}
\label{eq:inequalities}
2\le \frac{\area(bcd)}{\area(ocd)}\le \frac{1}{\frac{1}{2}-2 
\frac{\area(obc)\area(oab)}{\area(abc)^2}}
\end{equation}
\end{lemma}
\begin{proof}
Consider the following picture in Figure \ref{fig:second_bound}, where $oo'$ is parallel to $cd$ and $oo''$ is parallel to $ac$.
\begin{figure}[h!]
\centering
\caption{}\label{fig:second_bound}
\begin{tikzpicture}[scale=0.5,line cap=round,line join=round]
  \definecolor{mygreen}{RGB}{0,128,72}
  \definecolor{myblue}{RGB}{150,180,210}

  \coordinate (c) at (0,0);
  \coordinate (d) at (1.15,-3.0);
  \coordinate (a) at (5.55,-3.35);
  \coordinate (b) at (7.85,4.15);
  \coordinate (o) at ($(d)!0.42!(b)$);

  \coordinate (op) at ($(c)!0.42!(b)$);

  \coordinate (opp) at ($(c)!0.26953!(b)$);

  \draw[black,line width=1pt] (c)--(b)--(a)--(d)--cycle;
  \draw[black,line width=1pt] (c)--(a);

  \draw[mygreen,line width=1pt] (d)--(o)--(b);
  \draw[mygreen,line width=1pt] (c)--(o)--(a);
  \fill[mygreen] (o) circle (3pt);

  \draw[myblue,line width=1pt] (o)--(op);
  \draw[myblue,line width=1pt] (o)--(opp);
  \fill[myblue] (op) circle (2.4pt);
  \fill[myblue] (opp) circle (2.4pt);

  \node at ($(c)+(-0.52,0.08)$) {$c$};
  \node at ($(d)+(-0.40,-0.52)$) {$d$};
  \node at ($(a)+(0.22,-0.35)$) {$a$};
  \node at ($(b)+(0.42,0.20)$) {$b$};
  \node[text=mygreen] at ($(o)+(0.42,-0.22)$) {$o$};

  \node[text=myblue] at ($(op)+(0.02,0.45)$) {$o'$};
  \node[text=myblue] at ($(opp)+(-0.28,0.28)$) {$o''$};
\end{tikzpicture}
\end{figure}

 We find
\[
\frac{\area(bcd)}{\area(ocd)}=\frac{|cb|}{|co'|}\le \frac{|cb|}{|co''|}
\]
We first prove the upper bound in \eqref{eq:inequalities}.
Put
\begin{equation}
\label{eq:vdef}
v =\frac{\area(abc)^2}{\area(obc)\area(oab)}
\end{equation}
We need to find an upper bound for $|cb|/|co''|$ in terms of $v$. To this end we choose the coordinates as in Figure \ref{fig:coordinates}.  This gives us the picture in Figure \ref{fig:4bl_3}, where the purple curve expresses \eqref{eq:vdef} and the indicated $o''$ is the one such that $|co''|$ is minimal.
\begin{figure} 
\centering
\caption{}\label{fig:4bl_3}
\begin{tikzpicture}[x=5.5cm,y=5.5cm,line cap=round,line join=round]
  \definecolor{mypurple}{RGB}{140,85,210}
  \definecolor{myred}{RGB}{245,55,35}
  \definecolor{myblue}{RGB}{150,180,210}

  \pgfmathsetmacro{\vinv}{0.132361586}
  \pgfmathsetmacro{\twovinv}{2*\vinv}
  \pgfmathsetmacro{\xleft}{(1 - sqrt(1 - 4*\vinv))/2}
  \pgfmathsetmacro{\xright}{(1 + sqrt(1 - 4*\vinv))/2}

  \coordinate (A) at (0,1);
  \coordinate (B) at (0,0);
  \coordinate (C) at (1,0);

  \coordinate (Mab) at ($(A)!0.5!(B)$);
  \coordinate (Mac) at ($(A)!0.5!(C)$);
  \coordinate (Mbc) at ($(B)!0.5!(C)$);

  \coordinate (P) at ({2*\vinv},0.5);      
  \coordinate (Q) at (0.5,{2*\vinv});      
  \coordinate (R) at ({0.5+2*\vinv},0);    

  \draw[black,line width=1pt] (A)--(B)--(C)--cycle;

  \draw[myred,line width=1pt] (Mab)--(Mac)--(Mbc)--cycle;

  \draw[mypurple,line width=1pt,smooth,domain=\xleft:\xright,samples=200]
    plot (\x,{\vinv/\x});

  \draw[myblue,line width=1pt] (Q)--(R);
  \fill[myblue] (R) circle (2.5pt);

  \node at ($(A)+(-0.08,0.05)$) {$a(0,1)$};
  \node at ($(B)+(-0.02,-0.08)$) {$b(0,0)$};
  \node at ($(C)+(0.12,-0.03)$) {$c(1,0)$};

  \node[text=mypurple,font=\small] at ($(P)+(-0.03,0.10)$) {$xy=\dfrac{1}{v}$};

  \node[text=mypurple,font=\small] at ($(Q)+(0.11,0.09)$) {$\left(\dfrac12,\dfrac{2}{v}\right)$};
  \node at ($(R)+(0.04,0.05)$) {$o''$};
  \node[font=\small] at ($(R)+(0.00,-0.09)$) {$\left(\dfrac12+\dfrac{2}{v},0\right)$};
  
  \fill[mypurple] (P) circle (2.5pt);
  \fill[mypurple] (Q) circle (2.5pt);

\end{tikzpicture}
\end{figure} 

We see that
\[
|co''|\ge \frac{1}{2}-\frac{2}{v }
\]
so that
\[
\frac{|cb|}{|co''|}\le \frac{1}{\frac{1}{2}-\frac{2}{v }}
\]
which finishes the proof of the upper bound in \eqref{eq:inequalities}.
The lower bound follows from Lemma \ref{lem:forbidden_area}.
\end{proof}
\begin{lemma}
One has
\begin{equation}
\label{eq:r_red}
\area(oda)+\area(ocd)\le \area(oab)+\area(obc)
\end{equation}
\end{lemma}
\begin{proof}
We have 
\begin{align*}
\area(oda)+\area(oab)+\area(obc)+\area(ocd)&= \area(bcd)+\area(bda)\\
&\ge 2\area(ocd)+2\area(oda)\qquad (\text{by Lemma \ref{lem:forbidden_area})}.
\end{align*}
This implies \eqref{eq:r_red}.
\end{proof}
\subsection{Four blocks}
\label{sec:4block}
We let the setting be as in \S\ref{sec:approach},
i.e.\  up to rotation we are in the situation as in Figure \ref{fig:remainingcase}, where the long edge labelled by $E_i$ is subdivided into $\alpha_i$ equal segments
corresponding respectively to $E_i^{(j)}$ for $j=0,\ldots,\alpha_i-1$. We have fixed the rotation by assuming that the vertex $l_{12}$ corresponding to the intersection of the edges associated to $E_1^{(\alpha_1-1)}$, $E_2^{(0)}$ is not admissible (see Definition \ref{def:admissible}).  The directed graph in blue is the associated reduced quiver (see \S\ref{sec:blocks_and_block_quivers}).
\begin{figure}[h!]
\centering
\caption{}\label{fig:remainingcase}
\begin{tikzpicture}[scale=0.8,line cap=round,line join=round]
  \definecolor{myblue}{RGB}{30,80,200}

  \tikzset{
    midarrow/.style={
      draw=myblue,
      line width=1pt,
      postaction={decorate},
      decoration={
        markings,
        mark=at position 0.58 with {\arrow{Stealth[length=9pt,width=7pt]}}
      }
    }
  }
  \coordinate (L01) at (0,0.1);
  \coordinate (L12) at (0.85,-2.75);
  \coordinate (L23) at (5.0,-3.05);
  \coordinate (L30) at (7.1,4.15);

  \coordinate (E0) at ($(L01)!0.5!(L30)$);
  \coordinate (E1) at ($(L01)!0.5!(L12)$);
  \coordinate (E2) at ($(L12)!0.5!(L23)$);
  \coordinate (E3) at ($(L23)!0.5!(L30)$);

  \draw[black,very thick] (L01)--(L30)--(L23)--(L12)--cycle;
  
  \draw[black,dashed,thick] (L01)--(L23);

  \draw[midarrow] (E0)--(E1);
  \draw[midarrow] (E1)--(E2);
  \draw[midarrow] (E2)--(E3);
  \draw[midarrow] (E3)--(E0);
  \draw[midarrow] (E0)--(E2);
  \draw[midarrow] (E1)--(E3);

  \node at ($(L01)+(-0.5,-0.05)$) {$l_{0,1}$};
  \node at ($(L12)+(0.05,-0.42)$) {$l_{1,2}$};
  \node at ($(L23)+(0.38,-0.20)$) {$l_{2,3}$};
  \node at ($(L30)+(0.4,0.08)$) {$l_{3,0}$};

  \node at ($(E0)+(-0.45,0.55)$) {$E_0$};
  \node at ($(E1)+(-0.70,-0.20)$) {$E_1$};
  \node at ($(E2)+(0.00,-0.65)$) {$E_2$};
  \node at ($(E3)+(0.78,0.45)$) {$E_3$};

  \coordinate (O) at ($(E0)!0.34!(E2)!0.36!(E3)$);
\fill[black] (O) circle (2.2pt);
\node at ($(O)+(0.22,0.10)$) {$o$};

\end{tikzpicture}
\end{figure} 

We will frequently use the geometric results from \S\ref{sec:geometric}. To
apply them in our current situation we observe that by \eqref{eq:quiver_arrows} applied for $(i,j)=(i,i+1)$
we have
\begin{equation}
\label{eq:ratio_formula}
\chi^2_{i,i+1}\alpha_i\alpha_{i+1}=\frac{\area(abc)^2}{\area(oab)\area(obc)}
\end{equation}
with notations as Figure \ref{fig:bl4_4}. 
\begin{figure}[h!]
\centering
\caption{}\label{fig:bl4_4}
\begin{tikzpicture}[scale=0.6,line cap=round,line join=round]
  \definecolor{mygreen}{RGB}{0,128,72}

  \coordinate (O) at (0,0);
  \coordinate (C) at (1.9,4.8);
  \coordinate (B) at (6.0,5.1);
  \coordinate (A) at (6.6,2.7);

  \draw[black,very thick] (O)--(C)--(B)--(A)--cycle;

  \draw[black,very thick] (O)--(B);
  \draw[black,very thick] (C)--(A);

  \node at ($(O)+(-0.05,-0.45)$) {$o$};
  \node at ($(C)+(-0.35,0.15)$) {$c$};
  \node at ($(B)+(0.25,0.10)$) {$b$};
  \node at ($(A)+(0.35,-0.20)$) {$a$};

  \node[text=mygreen] at ($(C)!0.28!(B)+(0.00,0.65)$) {$E_{i+1}$};
  \node[text=mygreen] at ($(C)!0.68!(B)+(0.10,0.62)$) {$\alpha_{i+1}$};

  \node[text=mygreen] at ($(B)!0.40!(A)+(1.00,0.25)$) {$E_i$};
  \node[text=mygreen] at ($(B)!0.62!(A)+(0.95,-0.15)$) {$\alpha_i$};
\end{tikzpicture}

\end{figure}

Lemma \ref{lem:basic_bounds} combined with \eqref{eq:ratio_formula} yields
\begin{align}
\chi_{12}^2\alpha_1\alpha_2&\le 4\notag\\
\chi_{30}^2\alpha_3\alpha_0&\ge 4\label{eq:chi30bound}\\
\chi^2_{30}\chi_{12}^2\alpha_0\alpha_1\alpha_2\alpha_3 &\le 16\notag
\end{align}
where the first two inequalities are equalities if the origin lies on the dotted line in Figure \ref{fig:remainingcase}.
From  \eqref{eq:work_horse} and \eqref{eq:Sformula} we obtain
\begin{equation}
\label{eq:raw_version}
\frac{S_{01}}{\alpha_0\alpha_1 r_0^2r_1^2}+\frac{S_{12}}{\alpha_1\alpha_2 r_1^2r_2^2}+\frac{S_{23}}{\alpha_2 \alpha_3 r_2^2r_3^2}+ \frac{S_{30}}{\alpha_3\alpha_0 r_3^2r_0^2}=K_X^2
\end{equation}
Combining Theorem \ref{thm:achievement}\eqref{it:thm2_3}
with Lemma \ref{lem:forbidden_area} we see that the assumption that the origin is in the forbidden region is equivalent to 
\begin{align}
\frac{S_{30}}{\alpha_0 r_0^2}&\ge 2 & \frac{S_{30}}{\alpha_3 r_3^2}&\ge 2\\
\frac{S_{23}}{\alpha_2 r_2^2}&\ge 2 &\frac{S_{01}}{\alpha_1 r_1^2}&\ge 2 \label{eq:Sbound}
\end{align}
or, rewritten using \eqref{eq:chi_formula},
\begin{align}
\frac{r_0}{r_3}&\le \frac{\chi_{30}\alpha_3}{2} & \frac{r_3}{r_0}&\le \frac{\chi_{30}\alpha_0}{2}\label{eq:forbidden}\\
\frac{r_2}{r_3}&\le \frac{\chi_{23}\alpha_3}{2}&\frac{r_1}{r_0}&\le \frac{\chi_{01}\alpha_0}{2}
\end{align}
Finally,  we note that \eqref{eq:r_red} can be rewritten as
\begin{equation}
\label{eq:r12_bound}
\alpha_1 r_1^2+\alpha_2 r_2^2\le \alpha_0 r_0^2+\alpha_3 r_3^2.
\end{equation}
It is clear from \eqref{eq:K_theory} that there are only a finite number of possibilities for $(\alpha_i)_i$.
Likewise from \eqref{eq:chi30bound} we obtain that there is also a finite number of possibilities for $\chi_{12}$, $\chi_{30}$.

We will now prove that there are only a finite number of possibilities for $r_0$, $r_1$, $r_2$, $r_3$.  By  \eqref{eq:work_horse} this 
implies that there are also only a finite number of possibilities for $\chi_{01}$, $\chi_{23}$.

We consider the following two cases separately. 
\subsubsection{Case 1}
First we assume that the origin lies on the dotted interval in Figure \ref{fig:remainingcase}. 
Since the origin
lies in the forbidden region, it must be necessarily the midpoint of this interval. Hence by
 Theorem \ref{thm:achievement}\eqref{it:thm2_3}
we get 
\begin{align*}
\alpha_2 r_2^2&=\alpha_1 r_1^2\\
\alpha_3 r_3^2&=\alpha_0 r_0^2
\end{align*}
So it is sufficient to show that there are only a finite number of possibilities for $r_0$, $r_1$.  Using \eqref{eq:raw_version} we obtain
\begin{align*}
K^2_X&=\frac{S_{01}}{\alpha_0 \alpha_1 r_0^2r_1^2}+\frac{S_{12}}{\alpha_1\alpha_0 r_1^2r_0^2}+\frac{S_{23}}{\alpha_1 \alpha_0 r_1^2r_0^2}+ \frac{S_{30}}{\alpha_1\alpha_0 r_1^2r_0^2}\\
&=\frac{S_{01}+S_{12}+S_{23}+S_{30}}{\alpha_0 \alpha_1 r_0^2r_1^2}\\
&=\frac{2(\alpha_0r_0^2+\alpha_1r_1^2+\alpha_2r_2^2+\alpha_3 r_3^2)}{\alpha_0\alpha_1 r_0^2r_1^2},
\end{align*}
where in the last line we have used that $S_{01}+S_{12}+S_{23}+S_{30}=4\Delta(P)$ together with Theorem \ref{thm:achievement}\eqref{it:thm2_3}. 

So we finally obtain a Markov type equation
\begin{equation}
\label{eq:markov_var}
\alpha_0 r_0^2+\alpha_1 r_1^2=\frac{K^2_X\alpha_0\alpha_1}{4} r_0^2r_1^2.
\end{equation}
This has a finite number of solutions. Indeed, after dividing by $r_0^2r_1^2$ we may apply Remark \ref{rem:typical}.
\subsubsection{Case 2}
Now we assume that the origin does not lie on the dotted interval in Figure \ref{fig:remainingcase}.  
Since the first two inequalities in \eqref{eq:chi30bound} become strict, they can
be strengthened to
\begin{align}
\chi_{12}^2\alpha_1\alpha_2&\le 3\label{eq:chi12bound1}\\
\chi_{30}^2\alpha_3\alpha_0&\ge 5\label{eq:chi30bound1}
\end{align}
From \eqref{eq:chi12bound1} we get 
\begin{equation}
\label{eq:chivalue}
\chi_{12}=1
\end{equation}
From \eqref{eq:inequalities},  \eqref{eq:ratio_formula} and \eqref{eq:chi30bound1} combined with  Theorem \ref{thm:achievement}\eqref{it:thm2_3}
we get
\begin{equation}
\label{eq:tricky_bound1}
2\le \frac{S_{01}}{\alpha_1 r_1^2}\le \gamma:=\frac{1}{\frac{1}{2}-\frac{2}{\chi_{30}^2\alpha_3\alpha_0}}
\le \frac{1}{\frac{1}{2}-\frac{2}{(\sqrt{5})^2}}=10
\end{equation}
Similarly
\begin{equation}
\label{eq:tricky_bound2}
2\le \frac{S_{23}}{\alpha_2 r_2^2}\le  \gamma:=\frac{1}{\frac{1}{2}-\frac{2}{\chi_{30}^2\alpha_3\alpha_0}}\le 10
\end{equation}
So we get using \eqref{eq:chi_formula},  \eqref{eq:chivalue},\eqref{eq:tricky_bound1} and \eqref{eq:tricky_bound2}
\begin{multline}
\label{eq:Klower}
K_X^2=\frac{S_{01}}{\alpha_0 \alpha_1 r_0^2r_1^2}+\frac{S_{12}}{\alpha_1\alpha_2 r_1^2r_2^2}+\frac{S_{23}}{\alpha_2 \alpha_3 r_2^2r_3^2}+ \frac{S_{30}}{\alpha_3\alpha_0 r_3^2r_0^2}\\
\ge \frac{2}{\alpha_0 r_0^2}+\frac{1}{r_1r_2}+\frac{2}{\alpha_3 r_3^2}+\frac{\chi_{30}}{r_3r_0}
\end{multline}
and
\begin{multline}
\label{eq:Kupper}
K_X^2=\frac{S_{01}}{\alpha_0 \alpha_1 r_0^2r_1^2}+\frac{S_{12}}{\alpha_1\alpha_2 r_1^2r_2^2}+\frac{S_{23}}{\alpha_2 \alpha_3 r_2^2r_3^2}+ \frac{S_{30}}{\alpha_3\alpha_0 r_3^2r_0^2}\\
\le 
\frac{\gamma}{\alpha_0 r_0^2}+\frac{1}{r_1r_2}+ \frac{\gamma}{\alpha_3 r_3^2}
+ \frac{\chi_{30}}{r_3r_0}
\end{multline}
We claim that \eqref{eq:forbidden},  \eqref{eq:Klower}, \eqref{eq:Kupper} and \eqref{eq:r12_bound} have only a finite number of solutions for $r_0,r_1,r_2,r_3$.
From \eqref{eq:Klower} we obtain
\[
\frac{1}{r_1r_2}\le
\begin{cases}
1/2&\text{if $K_X^2=1$}\\
1&\text{otherwise}
\end{cases}
\]
Then from \eqref{eq:Kupper} we obtain
\begin{equation}
\label{eq:sanitized}
 \frac{10}{\alpha_0 r_0^2}+\frac{10}{\alpha_3 r_3^2}
+ \frac{\chi_{30}}{r_3r_0}\ge \frac{1}{2}
\end{equation}
Multipying by $r_0r_3$ yields
\[
 \frac{10r_3}{\alpha_0 r_0}+\frac{10r_0}{\alpha_3 r_3}
+ \chi_{30}\ge \frac{1}{2}r_0r_3
\]
and then using the upper bounds for $r_0/r_3$, $r_3/r_0$ obtained from \eqref{eq:forbidden} we get
\begin{equation}
  \label{eq:r0r3finite}
  r_0r_3\le 22\chi_{30}
\end{equation}
which only has a finite number
of solutions for $r_0$, $r_3$. From \eqref{eq:r12_bound} we then obtain that the number of solutions for $r_0,r_1,r_2,r_3$ is finite. 
\subsection{Five blocks}
\label{sec:5blocks}
In this section our mission is to show that no five block collection can satisfy the conditions of Theorem \ref{thm:minimal_list}.
\subsubsection{Case 1}
Here we assume that the $P$ is like on the left of Figure \ref{fig:red_5_blocks}.
In other words, there are two consecutive non-admissible non-straight vertices,
which means that by Lemma \ref{lem:prelim} after rotation we are in the situation as in Figure \ref{fig:5bl_basic_case_analyzed},
where we have given the origin coordinates $(0,0)$.

\begin{figure}[h!]
\centering
\caption{} \label{fig:5bl_basic_case_analyzed}
\begin{tikzpicture}[scale=0.8,line cap=round,line join=round,>=Stealth]
  \definecolor{myblue}{RGB}{150,180,220}
  \def\r{2.0}

  \coordinate (A) at (0,{2*\r});        
  \coordinate (B) at (0,{\r});          
  \coordinate (C) at ({\r},0);          
  \coordinate (D) at ({2*\r},0);        
  \coordinate (E) at ({3*\r},{4*\r});
  \coordinate (O) at ({\r},{\r});       

  \draw[black,line width=0.95pt,->] (A) -- node[midway,left=3pt] {$m_1$} (B);
  \draw[black,line width=0.95pt,->] (B) -- node[midway,below left=2pt] {$m_2$} (C);
  \draw[black,line width=0.95pt,->] (C) -- node[midway,below=3pt] {$m_3$} (D);
  \draw[black,line width=0.95pt,->] (D) -- node[midway,right=3pt] {$m_4$} (E);
  \draw[black,line width=0.95pt,->] (E) -- node[midway,above=3pt] {$m_0$} (A);

  \draw[myblue,line width=0.70pt] (A)--(D); 
  \draw[myblue,line width=0.70pt] (B)--(O);
  \draw[myblue,line width=0.70pt] (O)--(C);
  \draw[myblue,line width=0.70pt] (O)--(E);

  \node[left=6pt]  at (A) {$(-1,1)$};
  \node[left=6pt]  at (B) {$(-1,0)$};
  \node[below left=1pt] at (C) {$(0,-1)$};
  \node[below right=1pt] at (D) {$(1,-1)$};
  \fill[black] (O) circle (2.1pt);
  \node[below=2pt] at (O) {$o$};
  \node[right=6pt] at (O) {$(0,0)$};
  \node[right=4pt] at (E) {$(x,y)$};

\end{tikzpicture}
\end{figure}

We have
\begin{align*}
  R_0&=x+y\\
  R_1&=1\\
  R_2&=1\\
  R_3&=1\\
  R_4&=x+y
  \end{align*}
and 
\begin{align*}
  S_{01}&=x+1\\
  S_{12}&=1\\
  S_{23}&=1\\
  s_{34}&=y+1\\
  S_{40}&=2(x+y)
\end{align*}
Specialising \eqref{eq:work_horse} to the current setting and combining it with \eqref{eq:Rformula}\eqref{eq:Sformula} yields:
\begin{equation}
  \label{eq:work_horse:current}
\frac{S_{01}}{R_0R_1}+
\frac{S_{12}}{R_1R_2}+
\frac{S_{23}}{R_2R_3}+
\frac{S_{34}}{R_3R_4}+
\frac{S_{40}}{R_4R_0}=K_X^2.
\end{equation}
Substituting what we know yields
\[
  \frac{x+1}{x+y}+\frac{1}{1}+\frac{1}{1}+\frac{y+1}{x+y}+\frac{2(x+y)}{(x+y)^2}=K_X^2,
\]
which is equivalent to
\begin{equation}
  \label{eq:final5case1}
  4=(K^2_X-3)(x+y).
\end{equation}
Now $x,y\in \NN$ and the assumption that $P$ has two narrowing long edges and no parallel edges yields $x\ge 2$, $y\ge 2$.
Then \eqref{eq:final5case1} has a single solution:
\[
  x=y=2, \qquad K_X^2=4
\]
However since $(x,y)$ is not primitive in $L_\ZZ$, this contradicts Theorem \ref{thm:achievement}\eqref{it:thm2_2}.

\subsubsection{Case 2}
Here we assume that the $P$ is like on the right of Figure \ref{fig:red_5_blocks}.
In this case, after a suitable rotation, $P$ looks as in Figure \ref{fig:5block_case2} 

\begin{figure}
\centering
\caption{} \label{fig:5block_case2}

\begin{tikzpicture}[scale=1.0,line cap=round,line join=round]
  \usetikzlibrary{calc,intersections}
  \definecolor{myblue}{RGB}{150,180,220}
  \definecolor{mygreen}{RGB}{50,140,105}

  \coordinate (V40) at (-4.2,0.0);
  \coordinate (V34) at (-2.6,2.05);
  \coordinate (V23) at ( 0.0,3.05);
  \coordinate (V12) at ( 2.6,2.05);
  \coordinate (V01) at ( 4.2,0.0);

  \draw[black,line width=0.95pt] (V40)--(V34)--(V23)--(V12)--(V01)--cycle;

  \node[below left]  at (V40) {$l_{4,0}$};
  \node[left]        at (V34) {$l_{3,4}$};
  \node[above]       at (V23) {$l_{2,3}$};
  \node[right]       at (V12) {$l_{1,2}$};
  \node[below right] at (V01) {$l_{0,1}$};

  \path (V40)--node[below=3pt] {$m_0$} (V01);
  \path (V01)--node[right=2pt] {$m_1$} (V12);
  \path (V12)--node[above right=1pt] {$m_2$} (V23);
  \path (V23)--node[above left=1pt] {$m_3$} (V34);
  \path (V34)--node[left=2pt] {$m_4$} (V40);

  \draw[myblue,line width=0.70pt] (V40)--(V23);
  \draw[myblue,line width=0.70pt] (V01)--(V23);

  \coordinate (Mleft)  at ($(V40)!0.5!(V23)$);
  \coordinate (Mright) at ($(V01)!0.5!(V23)$);

  \draw[myblue,line width=0.70pt] (Mleft)--(Mright);
  \path[name path=aSeg] (Mleft)--(Mright);

  \coordinate (P0) at ($(V40)!0.5!(V01)$);

  \coordinate (dir4) at ($(P0) + 2.2*(V34) - 2.2*(V40)$); 
  \coordinate (dir1) at ($(P0) + 2.2*(V12) - 2.2*(V01)$); 
  \path[name path=t4] (P0)--(dir4);
  \path[name path=t1] (P0)--(dir1);

  \path[name intersections={of=aSeg and t4, by=I4}]; 
  \path[name intersections={of=aSeg and t1, by=I1}]; 

  \draw[myblue,line width=0.70pt] (P0)--(I4);
  \draw[myblue,line width=0.70pt] (P0)--(I1);

  \coordinate (d2) at ($(Mleft)  + 3.0*(V12) - 3.0*(V23)$); 
  \coordinate (d3) at ($(Mright) + 3.0*(V34) - 3.0*(V23)$); 
  \path[name path=s2] (Mleft)--(d2);
  \path[name path=s3] (Mright)--(d3);

  \path[name intersections={of=s2 and s3, by=Q}]; 

  \path[name intersections={of=t1 and s2, by=Jleft}];
  \path[name intersections={of=t4 and s3, by=Jright}];

  \fill[mygreen,opacity=0.60] (I1)--(I4)--(Jright)--(Q)--(Jleft)--cycle;
  \draw[myblue,line width=0.70pt] (I1)--(I4)--(Jright)--(Q)--(Jleft)--cycle;

\coordinate (QdownL) at ($(Q)!0.18!(d2)$); 
\coordinate (QdownR) at ($(Q)!0.18!(d3)$);

\draw[myblue,line width=0.70pt] (Mleft)--(QdownL);
\draw[myblue,line width=0.70pt] (Mright)--(QdownR);

  \coordinate (O) at ($($(I1)!0.55!(I4)$)!0.35!(Q)$);
  \fill[black] (O) circle (2.1pt);
  \node[right=2pt] at (O) {$o$};

\end{tikzpicture}

\end{figure}

The forbidden region has been coloured green.
Note that $l_{12}$, $l_{34}$ are non-admissible.
We will apply Lemma \ref{lem:basic_bounds} 
to the two $4$-vertex polygons obtained by respectively omitting $l_{12}$ and $l_{23}$.
By Lemma \ref{lem:nonadm_bigger} the origin is still in the forbidden
region for these smaller polygons. 

By \eqref{eq:quiver_arrows} and \eqref{eq:bound1}
we get 
\begin{align*}
\chi^2_{12}\alpha_1\alpha_2&\le 4\\
\chi^2_{34}\alpha_3\alpha_4&\le 4
\end{align*}
These inequalities must be strict since, as Figure \ref{fig:5block_case2} shows,  the origin is not on the intervals
$[l_{01},l_{23}]$ and $[l_{23},l_{40}]$.
So we get in fact
\begin{align*}
\chi^2_{12}\alpha_1\alpha_2&\le 3\\
\chi^2_{34}\alpha_3\alpha_4&\le 3
\end{align*}
Hence
\begin{equation}
  \label{eq:chi1234}
\chi_{12}=\chi_{34}=1
\end{equation}
We also have by \eqref{eq:chi_rel1} and \eqref{eq:quiver_arrows}:
\begin{equation}
\label{eq:16eq}
\begin{aligned}
\chi^2_{12}\alpha_1\alpha_2 \frac{\Delta(l_{01},l_{23},l_{40})^2}{\Delta(o,l_{23},l_{40}) \Delta(o,l_{40},l_{01})}&\le 16\\
\chi^2_{34}\alpha_3\alpha_4 \frac{\Delta(l_{01},l_{23},l_{40})^2}{\Delta(o,l_{01},l_{23}) \Delta(o,l_{40},l_{01})}&\le 16
\end{aligned}
\end{equation}

We will need the following lemma.
\begin{lemma}
Assume that $x,y,z\in \RR_{>0}$ satisfy
\begin{equation}
\label{eq:simple_markov}
x^2+y^2+z^2=xyz
\end{equation}
Then $x>2$ and
\begin{equation}
\label{eq:yzbound}
y,z\ge \frac{2x}{\sqrt{x^2-4}}
\end{equation}
\end{lemma}
\begin{proof}
  By symmetry it is sufficient to prove \eqref{eq:yzbound} only for $z$.
We see that $y$ satisfies the quadratic equation
\[
y^2-(xz)y+(x^2+z^2)=0
\]
Since $y$ is real, this means the discriminant must be non-negative.  Hence
\[
(xz)^2- 4(x^2+z^2)\ge 0
\]
Alternatively,
\begin{equation}
\label{eq:alternatively}
(x^2-4)z^2\ge 4x^2
\end{equation}
This is impossible if $x\le 2$, hence $x>2$. Moreover  \eqref{eq:alternatively} is equivalent to \eqref{eq:yzbound} for~$z$.
\end{proof}
\begin{corollary} \label{cor:simple_markov}
Let $\lambda>2$.
Assume that $x,y,z\in \RR_{>0}$ satisfy
\[
x^2+y^2+z^2=xyz
\]
and $x\le \lambda$. Then
\[
y,z\ge \frac{2\lambda}{\sqrt{\lambda^2-4}}
\]
\end{corollary}
\begin{proof} This follows from the fact that the right hand side of \eqref{eq:yzbound} is a decreasing function of $x$.
\end{proof}

From \eqref{eq:16eq} we obtain in particular
\begin{equation}
  \label{eq:naive_bound}
\begin{aligned}
 \frac{\Delta(l_{01},l_{23},l_{40})}{\sqrt{\Delta(o,l_{23},l_{40}) \Delta(o,l_{40},l_{01})}}&\le 4\\
\frac{\Delta(l_{01},l_{23},l_{40})}{\sqrt{\Delta(o,l_{01},l_{23}) \Delta(o,l_{40},l_{01})}}&\le 4
\end{aligned}
\end{equation}
Direct substitution shows that
{\tiny
\[
  (x,y,z):=\left( \frac{\Delta(l_{01},l_{23},l_{40})}{\sqrt{\Delta(o,l_{23},l_{40}) \Delta(o,l_{40},l_{01})}}, \frac{\Delta(l_{01},l_{23},l_{40})}
    {\sqrt{\Delta(o,l_{01},l_{23}) \Delta(o,l_{40},l_{01})}},
 \frac{\Delta(l_{01},l_{23},l_{40})}{\sqrt{\Delta(o,l_{23},l_{40}) \Delta(o,l_{01},l_{23})}}\right)
\]
}
satisfies \eqref{eq:simple_markov}. Hence we get by Corollary \eqref{cor:simple_markov} and \eqref{eq:naive_bound}:
\begin{equation}
\label{eq:bound99}
  \frac{\Delta(l_{01},l_{23},l_{40})}{\sqrt{\Delta(o,l_{23},l_{40}) \Delta(o,l_{40},l_{01})}}\ge \frac{2\times 4}{\sqrt{4^2-4}}=\frac{4}{\sqrt{3}}
\end{equation}
and similarly
\begin{align*}
 \frac{\Delta(l_{01},l_{23},l_{40})}{\sqrt{\Delta(o,l_{01},l_{23}) \Delta(o,l_{40},l_{01})}}&\ge \frac{4}{\sqrt{3}}
\end{align*}
Applying \eqref{eq:inequalities} to the polygon obtained by deleting $l_{34}$, together with Lemma \ref{lem:nonadm_bigger}, yields
\begin{equation}
  \label{eq:a1}
\frac{{\Delta}(l_{12},l_{23},l_{40})}{\Delta(o,l_{12},l_{23})}\le \frac{1}{\frac{1}{2}-2\frac{\Delta(o,l_{23},l_{40}) \Delta(o,l_{40},l_{01})}{{\Delta}(l_{01},l_{23},l_{40})^2}}
\end{equation}
We also observe that
\begin{equation}
  \label{eq:a2}
\Delta(l_{12},l_{23},l_{34})\le \Delta(l_{12},l_{23},l_{40})
\end{equation}
and, since $\omega$ is normalised,  we get from Theorem \ref{thm:achievement}\eqref{it:thm2_3}
\begin{equation}
  \label{eq:a3}
\alpha_2r_2^2=2\Delta(o,l_{12},l_{23}).
\end{equation}
Combining \eqref{eq:a1}, \eqref{eq:a2} and \eqref{eq:a3} yields
\begin{equation}
\label{eq:bound100}
2\frac{{\Delta}(l_{12},l_{23},l_{34})}{\alpha_2 r_2^2}\le \frac{1}{\frac{1}{2}-2\frac{\Delta(o,l_{01},l_{23})\Delta(o,l_{40},l_{01})}{{\Delta}(l_{01},l_{23},l_{40})^2}}.
\end{equation}
Combining \eqref{eq:bound100} with  \eqref{eq:bound99} yields
\[
\frac{S_{23}}{\alpha_2 r_2^2}=2\frac{{\Delta}(l_{12},l_{23},l_{34})}{\alpha_2 r_2^2}\le \frac{1}{\frac{1}{2}-2\frac{3}{16}}=8
\]
and similarly by symmetry
\[
\frac{S_{23}}{\alpha_3 r_3^2}=2\frac{{\Delta}(l_{12},l_{23},l_{34})}{\alpha_3 r_3^2}\le 8.
\]
Multiplying these yields by \eqref{eq:chi_formula}
\[
\chi_{23}^2\alpha_2\alpha_3 \le 8^2,
\]
so that in particular
\[
\chi_{23}\le 8.
\]
The same symmetry argument applied to the polygon obtained by deleting $l_{34}$ gives
\begin{equation}
\label{eq:same_way1}
\frac{S_{01}}{\alpha_1 r_1^2}=\frac{2\Delta(l_{01},l_{12},l_{40})}{\alpha_1 r_1^2}\le  \frac{1}{\frac{1}{2}-2\frac{\Delta(o,l_{23},l_{40}) \Delta(o,l_{40},l_{01})}{{\Delta}(l_{01},l_{23},l_{40})^2}}=8.
\end{equation}
Similarly,  using symmetry for $P$,  we get
\begin{equation}
\label{eq:same_way2}
\frac{S_{40}}{\alpha_1 r_4^2}=\frac{2\Delta(l_{01},l_{34},l_{40})}{\alpha_4 r_4^2}\le 8.
\end{equation}
Now we restate \eqref{eq:work_horse} in a convenient form, using \eqref{eq:chi_formula} (where we have already substituted \eqref{eq:chi1234}):
\begin{equation}
\label{eq:workhorse2}
\frac{1}{r_1r_2}+\frac{\chi_{23}}{r_2r_3}+\frac{1}{r_3r_4}+ \frac{S_{40}}{\alpha_4 r_4^2\alpha_0 r_0^2} +\frac{S_{01}}{\alpha_0 r_0^2 \alpha_1r_1^2}=K_X^2
\end{equation}
So in particular
\begin{equation}
\label{eq:3terms}
\frac{1}{r_1r_2}+\frac{\chi_{23}}{r_2r_3}+\frac{1}{r_3r_4}<K_X^2.
\end{equation}
If we apply Lemma \ref{lem:weird_lemma} to $S_{1}, S_{3}=\{1/n\mid n\ge 1\}$, $S_2=\{\chi_{23}/n\mid n\ge 1\}$
we find 
from \eqref{eq:workhorse2} that there exists $\epsilon>0$ such that
\[
\frac{S_{40}}{\alpha_4 r_4^2\alpha_0 r_0^2} +\frac{S_{01}}{\alpha_0 r_0^2 \alpha_1r_1^2}\ge \epsilon
\]
and hence by \eqref{eq:same_way1} and \eqref{eq:same_way2} we get
\[
\frac{16}{\alpha_0 r_0^2}=\frac{8}{\alpha_0 r_0^2}+\frac{8}{\alpha_0 r_0^2}\ge \frac{S_{40}}{\alpha_4 r_4^2\alpha_0 r_0^2} +\frac{S_{01}}{\alpha_0 r_0^2 \alpha_1r_1^2} \ge \epsilon.
\]
Hence there are only a finite number of possibilities for $r_0$. From \eqref{eq:max_inequality} we obtain
\begin{equation}
\label{eq:lower_r_bound}
r_1,r_2,r_3,r_4\le r_0.
\end{equation}
Hence there are only a finite number of possibilities for $r_0,r_1,r_2,r_3,r_4$. It then follows from \eqref{eq:work_horse}\eqref{eq:full_gram} that there are only a finite
number of Gram matrices.

\medskip

In now suffices to do an exhaustive search to show that there are no Gram matrices that satisfy the conditions of Theorem \ref{thm:minimal_list}.
However while the number of possibilities is finite, it is still very large, so we have to cut it down.

The first optimisation we can do is that once we have a bound on $r_0$, we have a bound on $r_1,r_2,r_3,r_4$.  Hence we can now compute an upper bound for \eqref{eq:3terms}
by brute force
(i.e.\ just looping over those $r_1,r_2,r_3,r_4$ that satisfy \eqref{eq:lower_r_bound}). This allows us to compute a new upper bound for $r_0$ and we can repeat this procedure until it converges. 

A second optimisation is to observe that Lemma \ref{lem:tool4} in fact implies that there exist $1\le i\le j\le 4$ such that
\begin{equation}
  \label{eq:from_tool}
  r_1>r_2>\cdots>r_i=r_{i+1}=\cdots=r_j<r_{j+1}<\cdots<r_4
\end{equation}
which allows us to prune more potential solutions.

A third optimisation is to invoke \eqref{eq:chi_formula},  i.e.
\begin{equation}
\label{eq:chi_formula2}
\chi_{i,i+1}=\frac{S_{i,i+1}}{\alpha_i \alpha_{i+1} r_i r_{i+1}}.
\end{equation}
Then the bound we have used
\[
\frac{S_{40}}{\alpha_4 r_4^2}\le 8
\]
becomes
\begin{equation}
  \label{eq:chi_40bound}
\frac{\chi_{40}\alpha_0 r_0}{ r_4}\le 8.
\end{equation}
Similarly, we get
\[
\frac{\chi_{01}\alpha_0 r_0}{ r_1}\le 8.
\]
These give upper bounds for $\chi_{40}$, $\chi_{01}$, which are very useful.

A fourth optimisation we can do is to note that using $(r_1,r_2,r_3,r_4,\chi_{12}=1, \chi_{23}, \chi_{34}=1)$ we can compute the Gram matrix for 
the subsequence 
\begin{equation}
\label{eq:subsequence}
(E_1^{(0)},\ldots, E_1^{(\alpha_1-1)},\ldots, E_4^{(0)},\ldots, E_4^{(\alpha_4-1)}).
\end{equation}
The fact that this Gram matrix must have integral entries is a strong restriction.
Moreover, we can also compute the Gram matrix of the dual exceptional collection,  which gives the arrows/relations for the corresponding quiver.  In this case there
are no relations (the quiver $Q$, which may be computed using Theorem \ref{thm:achievement}\eqref{it:thm2_5},  has no reverse arrows between the exceptional objects listed in \eqref{eq:subsequence}\footnote{Recall that according to Proposition \ref{prop:quiver_rule} relations in the quiver corresponding to $\EE$ correspond to reverse arrows in $Q$.}).
This allows us to eliminate more possibilities quickly.

\subsection{Six or more blocks}
\label{sec:6ormoreblocks}
We finally show that no six- or more block collection can satisfy the conditions of Theorem \ref{thm:minimal_list}. Lemma \ref{lem:three_adm} combined with
Lemma \ref{lem:prelim} leaves us only a single possibility to examine, namely,  the one shown in Figure \ref{fig:6blocks}, where the forbidden region has been coloured green.

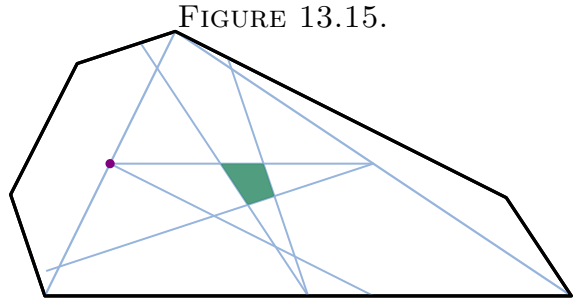
\begin{figure}[h!]
\centering
\caption{}\label{fig:6blocks}

\begin{tikzpicture}[scale=0.7,line cap=round,line join=round]
  \usetikzlibrary{calc,intersections}
  \definecolor{myblue}{RGB}{150,180,220}
  \definecolor{mygreen}{RGB}{50,140,105}
  
  \coordinate (V0) at (0.000,0.000);
  \coordinate (V1) at (-0.650,1.927);
  \coordinate (V2) at (0.614,4.416);
  \coordinate (V3) at (2.478,5.028);
  \coordinate (V4) at (8.763,1.868);
  \coordinate (V5) at (10.000,0.000);

  \draw[black,very thick] (V0)--(V1)--(V2)--(V3)--(V4)--(V5)--cycle;

  \begin{scope}
    \clip (V0)--(V1)--(V2)--(V3)--(V4)--(V5)--cycle;

    \draw[myblue,line width=0.85pt] (V0)--(V3);
    \draw[myblue,line width=0.85pt] (V3)--(V5);

    \coordinate (M1) at ($(V0)!0.5!(V3)$);
    \coordinate (M2) at ($(V3)!0.5!(V5)$);
    \coordinate (Mb) at ($(V0)!0.5!(V5)$);
   
    \draw[myblue,line width=0.75pt] (M1)--(M2);

    \coordinate (MbLfar)  at ($(Mb) + 30*(V1) - 30*(V0)$); 
    \coordinate (MbRfar)  at ($(Mb) + 30*(V4) - 30*(V5)$); 
    \coordinate (M2far)   at ($(M2) + 30*(V2) - 30*(V3)$); 
    \coordinate (M1far)   at ($(M1) + 30*(V4) - 30*(V3)$); 
    \coordinate (M1bfar)  at ($(M1) + 30*(V1) - 30*(V2)$); 

    \path[name path=chord]  (M1)--(M2);
    \path[name path=rayL]   (Mb)--(MbLfar);
    \path[name path=rayR]   (Mb)--(MbRfar);
    \path[name path=rayM2]  (M2)--(M2far);
    \path[name path=rayM1]  (M1)--(M1far);
    \path[name path=rayM1b] (M1)--(M1bfar);

    \path[name intersections={of=chord and rayL,  by=TopR}];
    \path[name intersections={of=chord and rayR,  by=TopL}];
    \path[name intersections={of=rayM2 and rayL,  by=BotR}];
    \path[name intersections={of=rayM2 and rayR,  by=BotL}];

    \path[name intersections={of=rayM1 and rayM2,  by=X12}]; 
    \path[name intersections={of=rayM1b and rayL,  by=XL}];  
 
    \def\eps{0.08} 

    \coordinate (StopL) at ($(TopR)!\eps!(MbLfar)$);
    \coordinate (StopR) at ($(TopL)!\eps!(MbRfar)$);
    \draw[myblue,line width=0.75pt] (Mb)--(StopL);
    \draw[myblue,line width=0.75pt] (Mb)--(StopR);

    \coordinate (StopM2) at ($(BotR)!\eps!(M2far)$);
    \draw[myblue,line width=0.75pt] (M2)--(StopM2);
    
    \coordinate (StopM1) at ($(X12)!\eps!(M1far)$);
    \draw[myblue,line width=0.75pt] (M1)--(StopM1);

    \coordinate (StopM1b) at ($(XL)!\eps!(M1bfar)$);
  
    \fill[mygreen,opacity=0.85] (TopL)--(TopR)--(BotR)--(BotL)--cycle;
    
     \fill[violet] (M1) circle (2.5pt);

  \end{scope}

  \draw[black,very thick] (V0)--(V1)--(V2)--(V3)--(V4)--(V5)--cycle;
\end{tikzpicture}
\end{figure}
 However this case cannot not occur
since according to Lemma \ref{lem:prelim} the origin should be in the purple dot,
which is not in the forbidden region.

\bibliography{bibs}
\bibliographystyle{amsplain}
\end{document}